\documentclass[12pt]{amsart}
\usepackage{cases}
\usepackage{cases}
\usepackage{cases}
\usepackage{mathrsfs}
\usepackage{amsfonts}
\usepackage{amscd}
\usepackage{latexsym,amssymb}
\usepackage{a4wide}
\usepackage{epic,eepic,latexsym, amssymb, amscd, amsfonts, xypic, floatflt}
\usepackage{amsthm}
\usepackage{amsmath}
\usepackage{amscd}
\usepackage{graphicx}
\usepackage[mathscr]{eucal}
\usepackage{verbatim}

\input xy
\xyoption{all}
\numberwithin{equation}{section}
\begin{document}
\def\@alph#1{\ifcase#1\or \or $'$\or $''$\fi}

\title[Congruent Skein Relations]{Congruent Skein Relations for Colored HOMFLY-PT invariants and Colored Jones Polynomials}
\author[Qingtao Chen, Kefeng Liu, Pan Peng and Shengmao Zhu]{Qingtao Chen,
Kefeng Liu, Pan Peng and Shengmao Zhu}
\address{Department of Mathematics \\
ETH Zurich \\
8092 Zurich \\ Switzerland } \email{qingtao.chen@math.ethz.ch}

\address{Center of Mathematical Sciences \\
Zhejiang University, Box 310027 \\
Hangzhou, P. R. China. }
\address{Department of mathematics \\
University of California at Los Angeles, Box 951555\\
Los Angeles, CA, 90095-1555.}
\email{liu@math.ucla.edu}
\address{Department of Mathematics \\
University of Arizona \\
617 N. Santa Rita Ave. \\
Tucson, AZ, 85721.}
\email{pengpan@gmail.com}
\address{Center of Mathematical Sciences \\
Zhejiang University, Box 310027 \\
Hangzhou, P. R. China. }
\email{zhushengmao@gmail.com}

\begin{abstract}
Colored HOMFLY-PT invariant, the generalization of the colored Jones
polynomial, is one of the most important quantum invariants of
links. This paper is devoted to investigating the basic structures
of the colored HOMFLY-PT invariants of links. By using the HOMFLY-PT
skein theory, firstly, we show that the (reformulated) colored
HOMFLY-PT invariants actually lie in the ring
$\mathbb{Z}[(q-q^{-1})^2,t^{\pm 1}]$. Secondly, we establish some
symmetric formulas for colored HOMFLY-PT invariants of links, which
include the rank-level duality as an easy consequence. Finally,
motivated by the Labastida-Mari\~no-Ooguri-Vafa conjecture for
framed links, we propose congruent skein relations for
(reformulated) colored HOMFLY-PT invariants which are the
generalizations of the skein relation for classical HOMFLY-PT
polynomials. Then we study the congruent skein relation for colored
Jones polynomials. In fact, we obtain a succinct formula for the
case of knot. As an application, we prove a vanishing result for
Reshetikhin-Turaev invariants of a family of 3-manifolds. Finally we study the congruent skein relations for $SU(n)$ quantum invariants.
\end{abstract}

\maketitle
\tableofcontents

\theoremstyle{plain} \newtheorem{thm}{Theorem}[section] \newtheorem{theorem}[%
thm]{Theorem} \newtheorem{lemma}[thm]{Lemma} \newtheorem{corollary}[thm]{%
Corollary} \newtheorem{proposition}[thm]{Proposition} \newtheorem{conjecture}%
[thm]{Conjecture} \theoremstyle{definition} \newtheorem{remark}[thm]{Remark}
\newtheorem{remarks}[thm]{Remarks} \newtheorem{definition}[thm]{Definition}
\newtheorem{example}[thm]{Example}






\section{Introduction}

The HOMFLY-PT polynomial is a two variables link invariant which was first
discovered by Freyd-Yetter, Lickorish-Millet, Ocneanu, Hoste and
Przytychi-Traczyk. In \cite{Jones}, Jones constructed the HOMFLY-PT
polynomial by studying the representation of Hecke algebra. Let $\mathcal{L}$
be an oriented link in $S^3$, the HOMFLY-PT polynomial $P(\mathcal{L};q,t)$
satisfies the following skein relation,
\begin{align}  \label{classicalskein}
tP(\mathcal{L}_+;q,t)-t^{-1}P(\mathcal{L}_-;q,t)=(q-q^{-1})P(\mathcal{L}%
_0;q,t)
\end{align}
with the initial value $P(U;q,t)=1$, we will use the notation $U$ to denote
the unknot throughout this paper. We denote by $(\mathcal{L}_+,\mathcal{L}_-,%
\mathcal{L}_0)$ the Conway triple of an oriented link.

One can calculate the HOMFLY-PT polynomial for any given oriented
link recursively through the above formula (1.1). Based on the work
\cite{Turaev} of Turaev, the HOMFLY-PT polynomial can be obtained
from the quantum invariant
associated with the fundamental representation of the quantum group $%
U_q(sl_N)$ by letting $q^N=t$. More generally, if we consider the quantum
invariants associated with arbitrary irreducible representations of $%
U_q(sl_N)$, by letting $q^N=t$, we get the colored HOMFLY-PT invariants $W_{%
\vec{A}}(\mathcal{L};q,t)$. See \cite{LZ} for detailed definition of the
colored HOMFLY-PT invariants through quantum group invariants of $U_q(sl_N)$%
. The colored HOMFLY-PT invariants have an equivalent definition through the
satellite invariants in HOMFLY-PT skein theory which, we refer to \cite%
{Ai,Lu} for a nice explanation of this equivalence. By using this approach,
for a link $\mathcal{L}$ with $L$-components $\mathcal{K}_\alpha$, $%
\alpha=1,...,L$. Let $\vec{\lambda}=(\lambda^1,...,\lambda^L)\in \mathcal{P}%
^L$, where $\mathcal{P}^L=\mathcal{P}\times\cdots\times \mathcal{P}$ and $%
\mathcal{P}$ denotes the set of all the partitions of positive integers. The
colored HOMFLY-PT invariant of $\mathcal{L}$ colored by $\vec{\lambda}$ is
given by
\begin{align}
W_{\vec{\lambda}}(\mathcal{L};q,t)=q^{-\sum_{\alpha=1}^{L}\kappa_{\lambda^%
\alpha}w(\mathcal{K}_\alpha)} t^{-\sum_{\alpha=1}^L|\lambda^\alpha|w(%
\mathcal{K}_\alpha)} \mathcal{H}(\mathcal{L}\star
\otimes_{\alpha=1}^LQ_{\lambda^\alpha}),
\end{align}
where $\mathcal{H}(\mathcal{L}\star \otimes_{\alpha=1}^LQ_{\lambda^\alpha})$
denotes the HOMFLY-PT polynomial of the link $\mathcal{L}$ decorated by the
element $\otimes_{\alpha=1}^LQ_{\lambda^\alpha}$, where each $%
Q_{\lambda^\alpha}$ is in the skein of annulus $\mathcal{C}_+$. For two
partitions $\lambda$ and $\mu$, we let $P_{\mu}=\sum_{\lambda}\chi_{%
\lambda}(\mu)Q_{\lambda}$, where $\chi_{\lambda}(\mu)$ is the value of the
character $\chi_{\lambda}$ of symmetric group at the conjugate class $%
C_{\mu} $. From the view of HOMFLY-PT skein theory, the element $P_{\mu}\in
\mathcal{C}_+$ takes a simple form and has nice properties(see Section 2 for
the detailed descriptions of the skein elements $Q_{\lambda}$ and $P_\mu$).
So it is natural to study the following reformulated colored HOMFLY-PT
invariants which are given by
\begin{align}
\mathcal{Z}_{\vec{\mu}}(\mathcal{L};q,t)=\mathcal{H}(\mathcal{L}\star
\otimes_{\alpha=1}^LP_{\mu^\alpha}), \ \check{\mathcal{Z}}_{\vec{\mu}}(%
\mathcal{L};q,t)=[\vec{\mu}]\mathcal{Z}_{\vec{\mu}}(\mathcal{L};q,t),
\end{align}
where $[\vec{\mu}]=\prod_{\alpha=1}^{L}\prod_{j=1}^{l(\mu^\alpha)}(q^{\mu^%
\alpha_j}-q^{-\mu^\alpha_j})$.

\subsection{Integrality}
In the first part of this paper, we obtain an integrality theorem
for the (reformulated) colored HOMFLY-PT invariants by applying the
HOMFLY-PT skein theory. By definition, a priori the colored HOMFLY-PT invariants lie in the ring $%
\Lambda=\mathbb{Q}[q^{\pm 1}, t^{\pm 1} ]$ with the elements $q^{k}-q^{-k}$
admitted as denominators for $k\geq 1$. However, we can show that the
reformulated colored HOMFLY-PT invariants actually belong to the subring $%
\mathbb{Z}[z^2,t^{\pm 1}]$, where we use notation $z=q-q^{-1}$ throughout
this paper. More precisely, we have
\begin{theorem}
For any link $\mathcal{L}$ with $L$ components, and
$\vec{\mu}=(\mu^1,...,\mu^L)\in \mathcal{P}^L$,
\begin{align}
\check{\mathcal{Z}}_{\vec{\mu}}(\mathcal{L};q,t)\in
\mathbb{Z}[z^2,t^{\pm 1}].
\end{align}
\end{theorem}

\subsection{Symmetries}

In the HOMFLY-PT skein theory, the two elements $Q_{\lambda}$ and $P_\mu$
satisfy the relation $Q_{\lambda}=\sum_{\mu}\frac{\chi_{\lambda}(\mu)}{%
z_{\mu}}P_{\mu}. $ So we have the close relationship between the colored
HOMFLY-PT invariants $W_{\vec{\lambda}}(\mathcal{L};q,t)$ and reformulated
colored HOMFLY-PT invariants $\mathcal{Z}_{\vec{\mu}}(\mathcal{L};q,t)$. As
the applications of Theorem 1.1, we establish the following symmetric
properties:
\begin{theorem}
Given a link $\mathcal{L}$ with $L$ components, and
$\vec{\lambda}=(\lambda^1,..,\lambda^L)\in \mathcal{P}^L$, we have
\begin{align}
W_{\vec{\lambda}}(\mathcal{L};q^{-1},t)&=(-1)^{\|\vec{\lambda}\|}W_{\vec{\lambda}^{t}}(\mathcal{L};q,t),\\
W_{\vec{\lambda}}(\mathcal{L};-q^{-1},t)&=W_{\vec{\lambda}^t}(\mathcal{L};q,t),\\
W_{\vec{\lambda}}(\mathcal{L};q,-t)&=(-1)^{\|\vec{\lambda}\|}W_{\vec{\lambda}}(\mathcal{L};q,t).
\end{align}
\end{theorem}
\begin{remark}
These symmetries in Theorem 1.2 are very general. For example,
combing (1.5) and (1.7), we obtain the symmetry:
\begin{align}
W_{\vec{\lambda}^t}(\mathcal{L};q^{-1},-t)=W_{\vec{\lambda}}(\mathcal{L};q,t)
\end{align}
which is referred as the rank-level duality in  \cite{LMV,LP2}.
Moreover, for a knot $\mathcal{K}$, if we use $\tilde{\mathcal{K}}$
to denote the mirror of $\mathcal{K}$, then we have
$W_{\lambda}(\tilde{\mathcal{K}};q,t)=W_{\lambda}(\mathcal{K};q^{-1},t^{-1})$
(See formula (4.19) in \cite{MV}). By formula (1.5), it is
straightforward to obtain
\begin{align}
W_{\lambda}(\tilde{\mathcal{K}};q,t)=(-1)^{|\lambda|}W_{\lambda^t}(\mathcal{K};q,t^{-1}),
\end{align}
which is just the formula (5) in \cite{So}.

\end{remark}

\subsection{Congruent skein relations}

\subsubsection{Background}

The seminal work \cite{Witten} of E. Witten showed that Chern-Simons gauge
theory provides a natural way to study the quantum invariants. In this
framework, the expectation value of Wilson loop along a link $\mathcal{L}$
in $S^3$ gives a topological invariant of the link depending on the
representation of the gauge group. N. Reshetikhin and V. Turaev \cite{RT}
gave a mathematical construction of this link invariant by using the
representation theory of the quantum group. In particular, the gauge group $%
SU(N)$ with irreducible representation will give rise to the colored
HOMFLY-PT invariant of the link $\mathcal{L}$. In another fundamental work
of Witten \cite{Witten2}, the $U(N)$ Chern-Simons gauge theory on a
three-manifold $M$ was interpreted as an open topological string theory on $%
T^*M$ with $N$ topological branes wrapping $M$ inside $T^*M$. Furthermore,
Gopakumar-Vafa \cite{GV} conjectured that the large $N$ limit of $SU(N)$
Chern-Simons gauge theory on $S^3$ is equivalent to the closed topological
string theory on the resolved conifold. This highly nontrivial string
duality was first checked for the case of the unknot by Ooguri-Vafa \cite{OV}%
. Later, a series of work \cite{LMV,LM} based on the large $N$
Chern-Simons/topological string duality, conjectured an expansion of the
Chern-Simons partition functions in terms of an infinite sequence of integer
invariants, which is called the Labastida-Mari\~no-Ooguri-Vafa (LMOV)
conjecture. This integrality conjecture serves as an essential evidence of
the Chern-Simons/topological string duality and was proved in \cite{LP1}.
When considering the framing dependence for $U(N)$ Chern-Simons gauge
theory, the integrality structure is even more amazing as described in \cite%
{MV}. In \cite{LP3}, two authors K. Liu and P. Peng paved a new way to study
this framing dependence integrality structure conjecture (we call it framed
LMOV conjecture in the following). In this framework, the framed LMOV
conjecture provides us the interesting congruent skein relation for the
reformulated colored HOMFLY-PT invariant $\check{\mathcal{Z}}_{\vec{\mu}}(%
\mathcal{L};q,t)$.

\subsubsection{Formulations}

In particular, when $\vec{\mu}=((p),...,(p))$ with $L$ row partitions $(p)$,
for $p\in\mathbb{Z}_+$ . We use the notation $\check{\mathcal{Z}}_{p}(%
\mathcal{L};q,t)$ to denote the reformulated colored HOMFLY-PT invariant $%
\check{\mathcal{Z}}_{((p),...,(p))}(\mathcal{L};q,t)$ for simplicity. By
this definition, for a link $\mathcal{L}$ with $L$-components, $\check{%
\mathcal{Z}}_1(\mathcal{L};q,t)$ is equal to the classical (framing
dependence) HOMFLY-PT polynomial $\mathcal{H}(\mathcal{L};q,t)$ by
multiplying a factor $[1]^L$, i.e $\check{\mathcal{Z}}_1(\mathcal{L}%
;q,t)=[1]^L\mathcal{H}(\mathcal{L};q,t)$. The skein relation for the
classical HOMFLY-PT polynomial leads to the skein relation for $\check{%
\mathcal{Z}}_1(\mathcal{L};q,t)$ as follow:
\begin{align}
\check{\mathcal{Z}}_1(\mathcal{L}_+;q,t)-\check{\mathcal{Z}}_1(\mathcal{L}%
_-;q,t)=\check{\mathcal{Z}}_1(\mathcal{L}_0;q,t),
\end{align}
when the crossing is the self-crossing of a component of the link $\mathcal{L%
}$, and
\begin{align}
\check{\mathcal{Z}}_1(\mathcal{L}_+;q,t)-\check{\mathcal{Z}}_1(\mathcal{L}%
_-;q,t)=[1]^2\check{\mathcal{Z}}_1(\mathcal{L}_0;q,t),
\end{align}
when the crossing is the linking of two different components of the link $%
\mathcal{L}$.

We use the notation $\mathcal{L}^{+1}$ to denote the link obtained by adding
a positive kink to one of the component of link $\mathcal{L}$. By the
relation (2.5),
\begin{align}
\check{\mathcal{Z}}_{1}(\mathcal{L}^{+1};q,t)=t\check{\mathcal{Z}}_{1}(%
\mathcal{L};q,t).
\end{align}
By using the framing change formulas showed in Section 5, we
establish the follow formula for
$\check{\mathcal{Z}}_p(\mathcal{L};q,t)$:
\begin{theorem}
When $p$ is a prime,
\begin{align}
\check{\mathcal{Z}}_{p}(\mathcal{L}^{+1};q,t)\equiv
(-1)^{p-1}t^{p}\check{\mathcal{Z}}_{p}(\mathcal{L};q,t) \mod
\{p\}^2.
\end{align}
\end{theorem}

We see that $\check{\mathcal{Z}}_p(\mathcal{L};q,t)$ has nice properties and
the similar behaviors with $\check{\mathcal{Z}}_1(\mathcal{L};q,t)$. A
natural question is if there exists the similar skein relation for $\check{%
\mathcal{Z}}_p(\mathcal{L};q,t)$? Motivated by the new approach to the
framed LMOV conjecture \cite{LP3}, we propose the following congruent skein
relation for the reformulated colored HOMFLY-PT invariant $\check{\mathcal{Z}%
}_p(\mathcal{L};q,t)$ as follow:
\begin{conjecture} For any link $\mathcal{L}$  and a prime number $p$, we have
\begin{align}
\check{\mathcal{Z}}_{p}(\mathcal{L}_+;q,t)-\check{\mathcal{Z}}_{p}(\mathcal{L}_-;q,t)\equiv
(-1)^{p-1}\check{\mathcal{Z}}_{p}(\mathcal{L}_0;q,t) \mod \{p\}^2,
\end{align}
when the crossing is the self-crossing of a knot, and
\begin{align}
\check{\mathcal{Z}}_{p}(\mathcal{L}_+;q,t)-\check{\mathcal{Z}}_{p}(\mathcal{L}_-;q,t)\equiv
(-1)^{p-1} p[p]^2\check{\mathcal{Z}}_{p}(\mathcal{L}_{0};q,t) \mod
 \{p\}^2[p]^2.
\end{align}
when the crossing is the linking of  two different components of the
link $\mathcal{L}$. Where the notation $A\equiv B \mod C$ denotes $
\frac{A-B}{C} \in \mathbb{Z}[(q-q^{-1})^2,t^{\pm 1}]. $ And
$[p]=q^p-q^{-p}$, $\{p\}=(q^p-q^{-p})/(q-q^{-1})$.
\end{conjecture}

\subsubsection{Evidence}

The complete proof of the above two relations is unknown, but some partial
results will be given in Section 6 and appendix. For examples, we have
\begin{theorem}
Let $\mathcal{K}_+$ be the knot obtained by adding a positive kink
to knot $\mathcal{K}$, and $\mathcal{K}_-$  be the knot obtained by
adding a negative kink to knot $\mathcal{K}$, let
$\mathcal{K}_0=\mathcal{K}\otimes U$, then
$(\mathcal{K}_+,\mathcal{K}_-,\mathcal{K}_0)$ forms a Conway triple
and the congruent skein relation (1.14) holds.
\end{theorem}

\begin{theorem}
For the Conway triple $\mathcal{L}_{+}=4_1$,
$\mathcal{L}_{-}$=unknot with two negative kinks and
$\mathcal{L}_{0}=T(2,-2)$ with one positive kink. The congruent
skein relation (1.14) holds for $p=2,3$.
\end{theorem}

\begin{theorem}
Consider the $(p,q)$ torus link $T(p,q)$.  Let $k\in \mathbb{Z}$,
for the Conway triple $\mathcal{L}_+=T(2,2k+1)$,
$\mathcal{L}_-=T(2,2k-1)$, $\mathcal{L}_0=T(2,2k)$, the congruent
skein relation (1.14) holds for $p=2$.

Similarly, for the Conway triple $\mathcal{L}_+=T(2,2k)$,
$\mathcal{L}_-=T(2,2k-2)$, $\mathcal{L}_0=T(2,2k-1)$, the congruent
skein relation (1.15) holds for $p=2$.
\end{theorem}

\subsection{Colored Jones polynomials}

Colored Jones polynomial can be viewed as the special case of the colored
HOMFLY-PT invariant:
\begin{align}
J_{N}(\mathcal{L};q)=\frac{q^{-2lk(\mathcal{L})N(N+1)}W_{(N)(N),...,(N)}(%
\mathcal{L},q,q^{2})}{W_{(N)}(U;q,q^{2})}
\end{align}
So it is natural to consider if there exists the similar congruent skein
relation for colored Jones polynomials $J_{N}(\mathcal{L};q)$. Please note
that we use a little different symbol for the colored Jones polynomial in
this paper, here $N$ denotes the Sym$^N(V)$ ($V$ is the fundamental
representation), which is the $N+1$-dimension irreducible representation of $%
U_q(sl_2\mathbb{C})$. We prove the following congruent skein relation for
any knot $\mathcal{K}$.
\begin{theorem}
For any positive integers $N, k$ and $N\geq k\geq 0$,
\begin{align} \label{typei1}
J_{N}(\mathcal{K}_{+};q)-J_{N}(\mathcal{K}_{-};q)\equiv
J_{k}(\mathcal{K}_{+};q)-J_{k}(\mathcal{K}_{-};q) \mod \
[N-k][N+k+2].
\end{align}
\end{theorem}
In fact, by using the cyclotomic expansion formula for colored Jones
polynomial due to K. Habiro \cite{Hab}, we prove the following equivalent
result.
\begin{theorem}
For any knot $\mathcal{K}$, and positive integers $N\geq k\geq 0$,
\begin{align}
J_{N}(\mathcal{K};q)-J_k(\mathcal{K};q)\equiv 0 \mod [N-k][N+k+2].
\end{align}
\end{theorem}
In particular, taking $k=0$, we obtain
\begin{theorem}
For any knot $\mathcal{K}$, $i\in \mathbb{Z}$,
$$J_{N}(\mathcal{K};e^{\frac{\pi i\sqrt{-1}}{N}})=1,\ J_{N}(\mathcal{K};e^{\frac{\pi i\sqrt{-1}}{N+2}})=1.$$
\end{theorem}
If we take $N=1$, $i=1$ in Theorem 1.10, we have $J_{1}(\mathcal{K};e^{\frac{%
\pi \sqrt{-1}}{3}})=1$ which is a famous result due to V. Jones (see 12.4 in
\cite{Jones} and compare the notation $J_1(\mathcal{L})$ and the Jones
polynomial $V_{\mathcal{L}}(t)$ in \cite{Jones}).

As an important application of Theorem 1.10, we obtain a vanishing result of
the Reshetikhin-Turaev invariant for certain $3$ dimensional oriented closed
manifolds. Let us denote by $M_{p}$ the $3$-manifold obtained from $S^{3}$
by doing a $p$-surgery ($p\in \mathbb{Z}$) along a knot $\mathcal{K}\subset
S^{3}$. Let $q(s)=e^{\frac{s\pi \sqrt{-1}}{r}}$, $J_{n}(\mathcal{K};q(s))$
be the $n+1$-dimension colored Jones polynomial of $\mathcal{K}$ discussed
above at the roots of unity $q=q(s)$. According to \cite{RT, Turaev3}, for
odd integer $r\geq 3$, the Reshetikhin-Turaev invariants $\tau
_{r}(M_{p};q(s))$ of $M_{p}$ can be calculated by the following formula:
\begin{align}
\tau _{r}(M_{p};q(s))=C(r)\sum_{n=0}^{r-2}\left( \sin \frac{s(n+1)\pi }{r}%
\right) ^{2}e^{-\frac{sp(n^{2}+2n)\pi \sqrt{-1}}{2r}}J_{n}(\mathcal{K};q(s)),
\end{align}
where $C(r)$ denotes certain function depending only on $r$.

Inspired by the numerical phenomenon observed in \cite{CY}, we prove the
following
\begin{theorem}[Vanishing of Reshetikhin-Turaev invariants]
For odd $r\geq 3$, if $p=4k+2$ ($k\in \mathbb{Z}$) and odd $s$, then
we have the vanishing of Reshetikhin-Turaev invariants as follows
\begin{align}
\tau_{r}(M_p;q(s))=0.
\end{align}
\end{theorem}
This vanishing result surprises us a bit. After the communications
with E. Witten , he told us that there should be a physical
interpretation behind this phenomenon \cite{Witten3}.

We also propose the following congruent skein relations for the $SU(n)$
quantum invariant for $n\geq 3$.
\begin{conjecture} \label{Conjsu(n)}
For a knot $\mathcal{K}$, for any positive integer $N, k$ and $N\geq
k\geq 0$, we have
\begin{align}
J_{N}^{SU(n)}(\mathcal{K}_{+};q)-J_{N}^{SU(n)}(\mathcal{K}_{-};q)
\equiv
J_{k}^{SU(n)}(\mathcal{K}_{+};q)-J_{k}^{SU(n)}(\mathcal{K}_{-};q)
\mod \ [N-k].
\end{align}
\begin{align}
&J_{N}^{SU(n)}(\mathcal{K}_{+};q)-J_{N}^{SU(n)}(\mathcal{K}_{-};q)\\\nonumber
&\equiv
J_{k}^{SU(n)}(\mathcal{K}_{+};q)-J_{k}^{SU(n)}(\mathcal{K}_{-};q)
\mod \ [N+k+n].
\end{align}
\begin{align}
J_{N}^{SU(n)}(\mathcal{K}_{+};q)-J_{N}^{SU(n)}(\mathcal{K}_{-};q)
\equiv
J_{k}^{SU(n)}(\mathcal{K}_{+};q)-J_{k}^{SU(n)}(\mathcal{K}_{-};q)
\mod \ [n-1].
\end{align}
\end{conjecture}

Finally these congruent skein relations for $SU(n))$ quantum invariants leads to Volume Conjectures for $SU(n)$ quantum invariants recently studied in \cite{CLZ}.

\bigskip

The rest of this paper is organized as follows. In Section 2, we
introduce the HOMFLY-PT skein model to give the definition of
(reformulated) colored HOMFLY-PT invariants. In Section 3, we prove
the integrality theorem for the reformulated colored HOMFLY-PT
invariants. In Section 4, we establish the symmetries for colored
HOMFLY-PT invariants, including the rank-level duality as
applications. In Section 5, We provide more results on the
(reformulated) colored HOMFLY-PT invariants such as framing changing
formulas, which could be seen as the preparation for our study of
the congruent skein relations. In Section 6, we first propose a
conjecture of congruent skein relation for reformulated colored
HOMFLY-PT invariants, then we prove them in some special cases.
Furthermore, we show an application of the congruent skein
relations. Colored Jones polynomial can be regarded as the special
case of the colored HOMFLY-PT invariant. In Section 7, we prove a
congruent skein relation for colored Jones polynomial by using the
cyclotomic expansion of the colored Jones polynomial of knot. Then
we show a vanishing result of the Reshetikhin-Turaev invariants for
certain 3-manifols as an interesting application of this congruent
skein relation for colored Jones polynomials. Furthermore, we
conjectured link case of congruent skein relations for colored Jones
polynomials and knot case of congruent skein relations for $SU(n)$
quantum invariants. In the appendix, we provide some sample examples
to illustrate the congruent skein relations for reformulated colored
HOMFLY-PT invariants, colored Jones polynomials and $SU(n)$
quantum invariants.
\\

{\bf Acknowledgements.}  We would like to thank  R. Kashaev, J.
Murakami, N. Reshetikhin, E. Witten and Tian Yang for valuable discussions with us.
Q. Chen thank Dror Bar-Natan and Scott Morrison for communicating on
KnotTheory, Package of Mathematica. and Q. Chen also thank CMS at
Zhejiang University for their hospitality. Both Q. Chen and S. Zhu
thank Shanghai Center for Mathematical Science for their
hospitality. The research of S. Zhu is supported by the National
Science Foundation of China grant No. 11201417 and the China
Postdoctoral Science special Foundation No. 2013T60583.

\section{Colored HOMFLY-PT invariants}

\subsection{Partitions and symmetric functions}

A partition $\lambda$ is a finite sequence of positive integers $%
(\lambda_1,\lambda_2,..)$ such that $\lambda_1\geq \lambda_2\geq\cdots$. The
length of $\lambda$ is the total number of parts in $\lambda$ and denoted by
$l(\lambda)$. The weight of $\lambda$ is defined by $|\lambda|=%
\sum_{i=1}^{l(\lambda)}\lambda_i$. If $|\lambda|=d$, we say $\lambda$ is a
partition of $d$ and denoted as $\lambda\vdash d$. The automorphism group of
$\lambda$, denoted by Aut($\lambda$), contains all the permutations that
permute parts of $\lambda$ by keeping it as a partition. Obviously, Aut($%
\lambda$) has the order $|\text{Aut}(\lambda)|=\prod_{i=1}^{l(\lambda)}m_i(%
\lambda)! $ where $m_i(\lambda)$ denotes the number of times that $i$ occurs
in $\lambda$.

Every partition is identified to a Young diagram. The Young diagram of $%
\lambda$ is a graph with $\lambda_i$ boxes on the $i$-th row for $%
j=1,2,..,l(\lambda)$, where we have enumerated the rows from top to bottom
and the columns from left to right. Given a partition $\lambda$, we define
the conjugate partition $\lambda^t$ whose Young diagram is the transposed
Young diagram of $\lambda$: the number of boxes on $j$-th column of $%
\lambda^t$ equals to the number of boxes on $j$-th row of $\lambda$, for $%
1\leq j\leq l(\lambda)$.

The following numbers associated with a given partition $\lambda$ are used
frequently in this article:
\begin{align}
z_\lambda=\prod_{j=1}^{l(\lambda)}j^{m_{j}(\lambda)}m_j(\lambda)! \quad
\text{and} \quad
k_{\lambda}=\sum_{j=1}^{l(\lambda)}\lambda_j(\lambda_j-2j+1).
\end{align}
Obviously, $k_\lambda$ is an even number and $k_\lambda=-k_{\lambda^t}$.

In the following, we will use the notation $\mathcal{P}_+$ to denote the set
of all the partitions of positive integers. Let $\emptyset$ be the partition
of $0$, i.e. the empty partition. Define $\mathcal{P}=\mathcal{P}_+\cup
\{\emptyset\}$, and $\mathcal{P}^L$ the $L$ tuple of $\mathcal{P}$.

The power sum symmetric function of infinite variables $x=(x_1,..,x_N,..)$
is defined by $p_{n}(x)=\sum_{i}x_i^n. $ Given a partition $\lambda$, define
$p_\lambda(x)=\prod_{j=1}^{l(\lambda)}p_{\lambda_j}(x). $ The Schur function
$s_{\lambda}(x)$ is determined by the Frobenius formula
\begin{align}  \label{Frobeniusformula}
s_\lambda(x)=\sum_{\mu}\frac{\chi_{\lambda}(\mu)}{z_\mu}p_\mu(x).
\end{align}
where $\chi_\lambda$ is the character of the irreducible representation of
the symmetric group $S_{|\lambda|}$ corresponding to $\lambda$, we have $%
\chi_{\lambda}(\mu)=0$ if $|\mu|\neq |\lambda|$. The orthogonality of
character formula gives
\begin{align}  \label{orthog}
\sum_\lambda\frac{\chi_\lambda(\mu) \chi_\lambda(\nu)}{z_\mu}=\delta_{\mu
\nu}.
\end{align}

\subsection{HOMFLY-PT skein theory}

We follow the notations in \cite{HM}. Define the coefficient ring $\Lambda=%
\mathbb{Z}[q^{\pm 1}, t^{\pm 1} ]$ with the elements $q^{k}-q^{-k}$ admitted
as denominators for $k\geq 1$. Let $F$ be a planar surface, the framed
HOMFLY-PT skein $\mathcal{S}(F)$ of $F$ is the $\Lambda$-linear combination
of the orientated tangles in $F$, modulo the two local relations as showed
in Figure 1 where $z=q-q^{-1}$,
\begin{figure}[!htb]
\begin{center}
\includegraphics[width=150 pt]{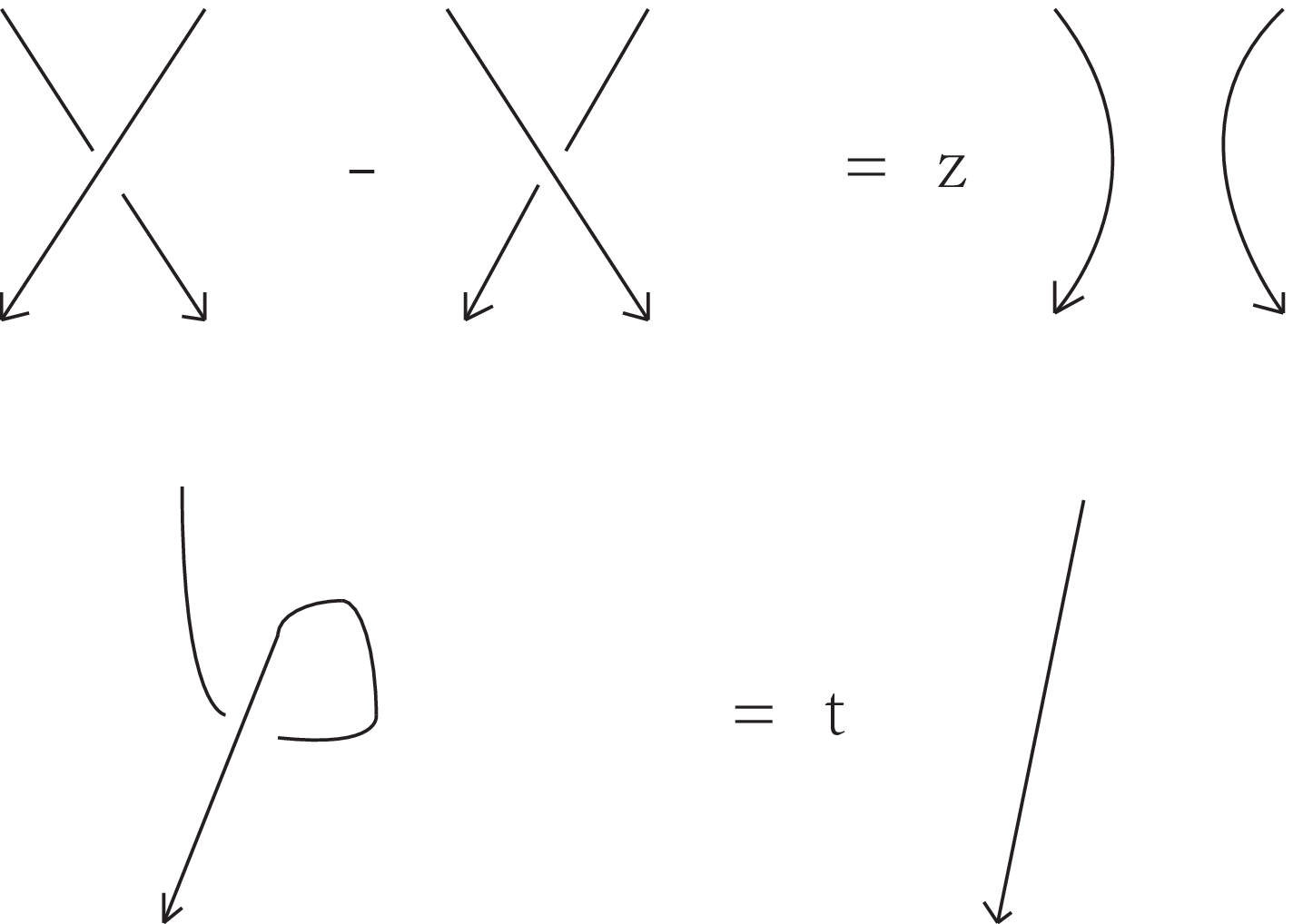}
\end{center}
\caption{}
\end{figure}
It is easy to follow that the removal an unknot is equivalent to time a
scalar $s=\frac{t-t^{-1}}{q-q^{-1}}$, i.e we have the relation showed in
Figure 2.
\begin{figure}[!htb]
\begin{center}
\includegraphics[width=80 pt]{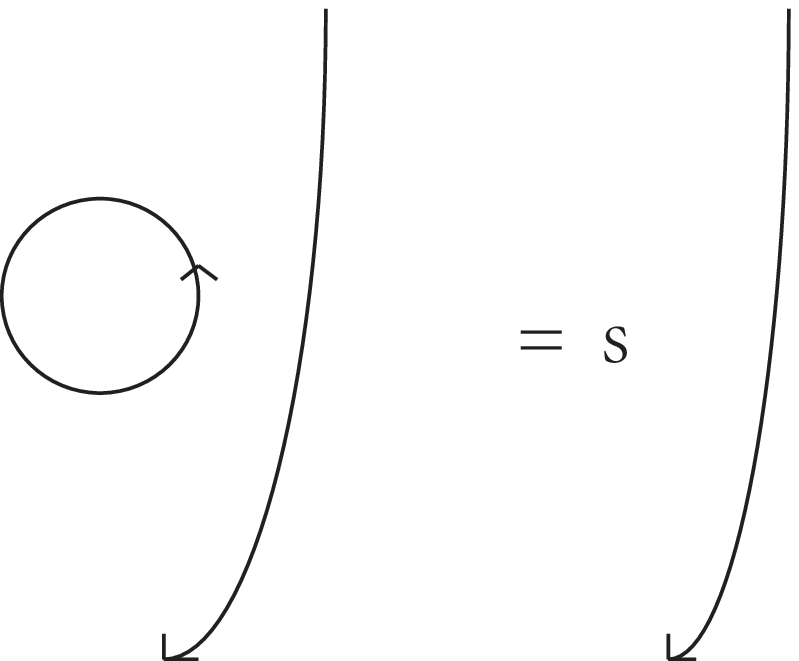}
\end{center}
\caption{}
\end{figure}

\subsubsection{The plane}

When $F=\mathbb{R}^2$, it is easy to follow that every element in $\mathcal{S%
}(F)$ can be represented as a scalar in $\Lambda$. For a link $\mathcal{L}$
with a diagram $D_{\mathcal{L}}$, the resulting scalar $\langle D_{\mathcal{L%
}} \rangle \in \Lambda$ is the (framed unreduced) HOMFLY-PT polynomial $%
\mathcal{H}(\mathcal{L};q,t)$ of the link $\mathcal{L}$. I.e. $\mathcal{H}(%
\mathcal{L};q,t)=\langle D_\mathcal{L}\rangle$. We use the convention $%
\langle \ \rangle=1$ for the empty diagram, so $\mathcal{H}(U;q,t)=\frac{%
t-t^{-1}}{q-q^{-1}}$. The two relations showed in Figure 1 lead to
\begin{align}
\mathcal{H}(\mathcal{L}_+;q,t)-\mathcal{H}(\mathcal{L}_-;q,t)=z\mathcal{H}(%
\mathcal{L}_0;q,t), \\
\mathcal{H}(\mathcal{L}^{+1};q,t)=t\mathcal{H}(\mathcal{L};q,t) \ \text{and}
\ \mathcal{H}(\mathcal{L}^{-1};q,t)=t^{-1}\mathcal{H}(\mathcal{L};q,t).
\end{align}
The classical HOMFLY-PT polynomial of a link $\mathcal{L}$ is given by
\begin{align}
P(\mathcal{L};q,t)=\frac{t^{-w(\mathcal{L})}\mathcal{H}(\mathcal{L};q,t)}{%
\mathcal{H}(U;q,t)},
\end{align}
where $w(\mathcal{L})$ denotes the writhe number of link $\mathcal{L}$.

\subsubsection{The rectangle}

When $F$ is a rectangle with $n$ inputs at the top and $n$ outputs at the
bottom. Let $H_n$ be the skein $\mathcal{S}(F)$ of $n$-tangles. See Figure 3
for an element in $H_{n}$.
\begin{figure}[!htb]
\begin{center}
\includegraphics[width=80 pt]{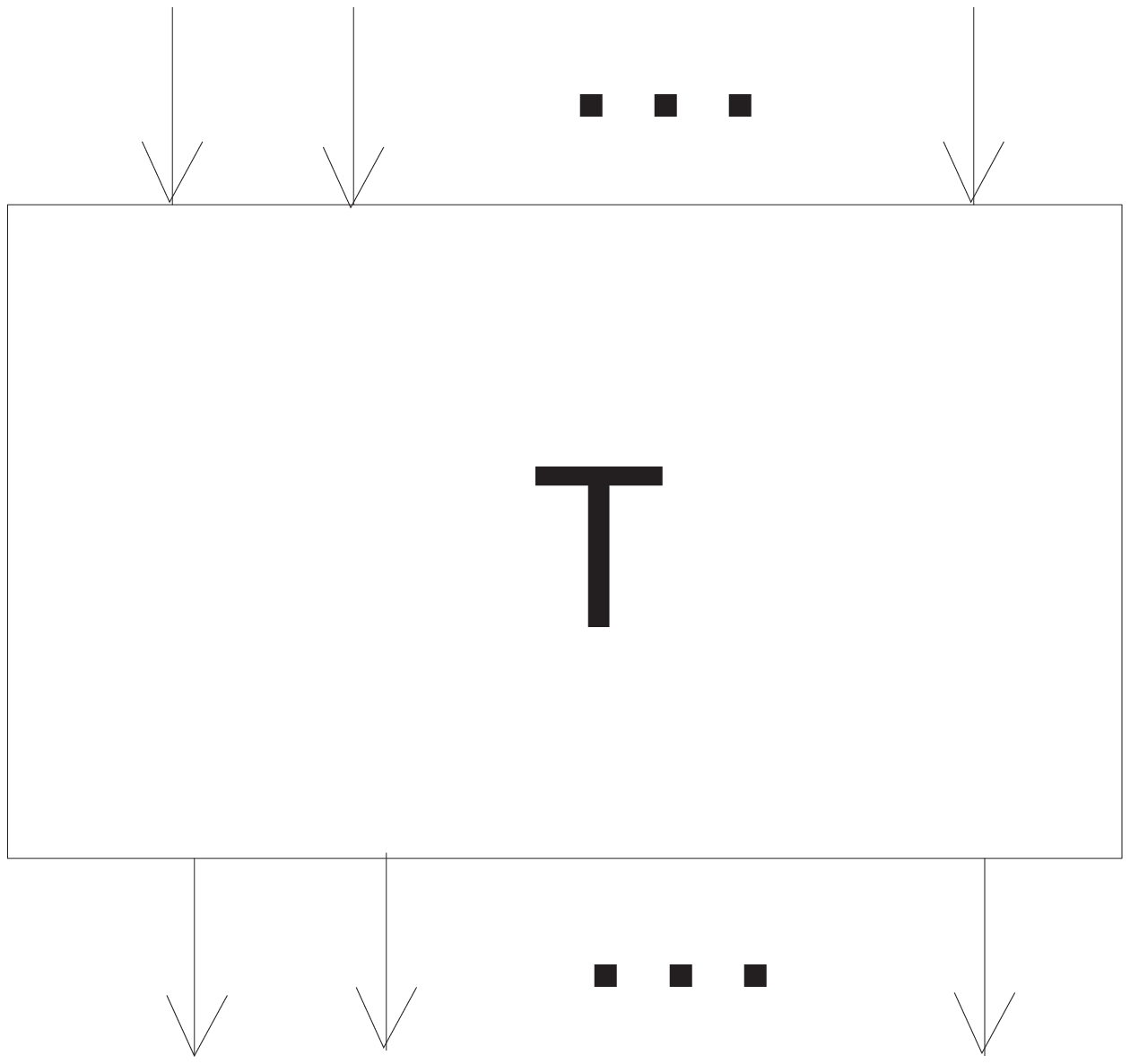}
\end{center}
\caption{}
\end{figure}

Composing $n$-tangles by placing one above another induces a product which
makes $H_n$ into the Hecke algebra $H_n(z)$ with the coefficients ring $%
\Lambda$, where $z=q-q^{-1}$. $H_n(z)$ has a presentation generated by the
elementary brads $\sigma_i$ subjects to the braid relations
\begin{align}
\sigma_i\sigma_{i+1}\sigma_i=\sigma_{i+1}\sigma_i\sigma_{i+1} \\
\sigma_i\sigma_j=\sigma_j\sigma_i, |i-j|\geq 1.  \notag
\end{align}
and the quadratic relations $\sigma_i^2=z\sigma_i+1$.

\subsubsection{The annulus}

When $F=S^1\times I$ is the annulus, we denote $\mathcal{C}=\mathcal{S}%
(S^1\times I)$. $\mathcal{C}$ is a commutative algebra with the product
induced by placing the annulus one outside other. For any element $T\in H_n$%
, we use $\hat{T}$ to denote the closure of $T$ as showed in Figure 4.

\begin{figure}[!htb]
\begin{center}
\includegraphics[width=120 pt]{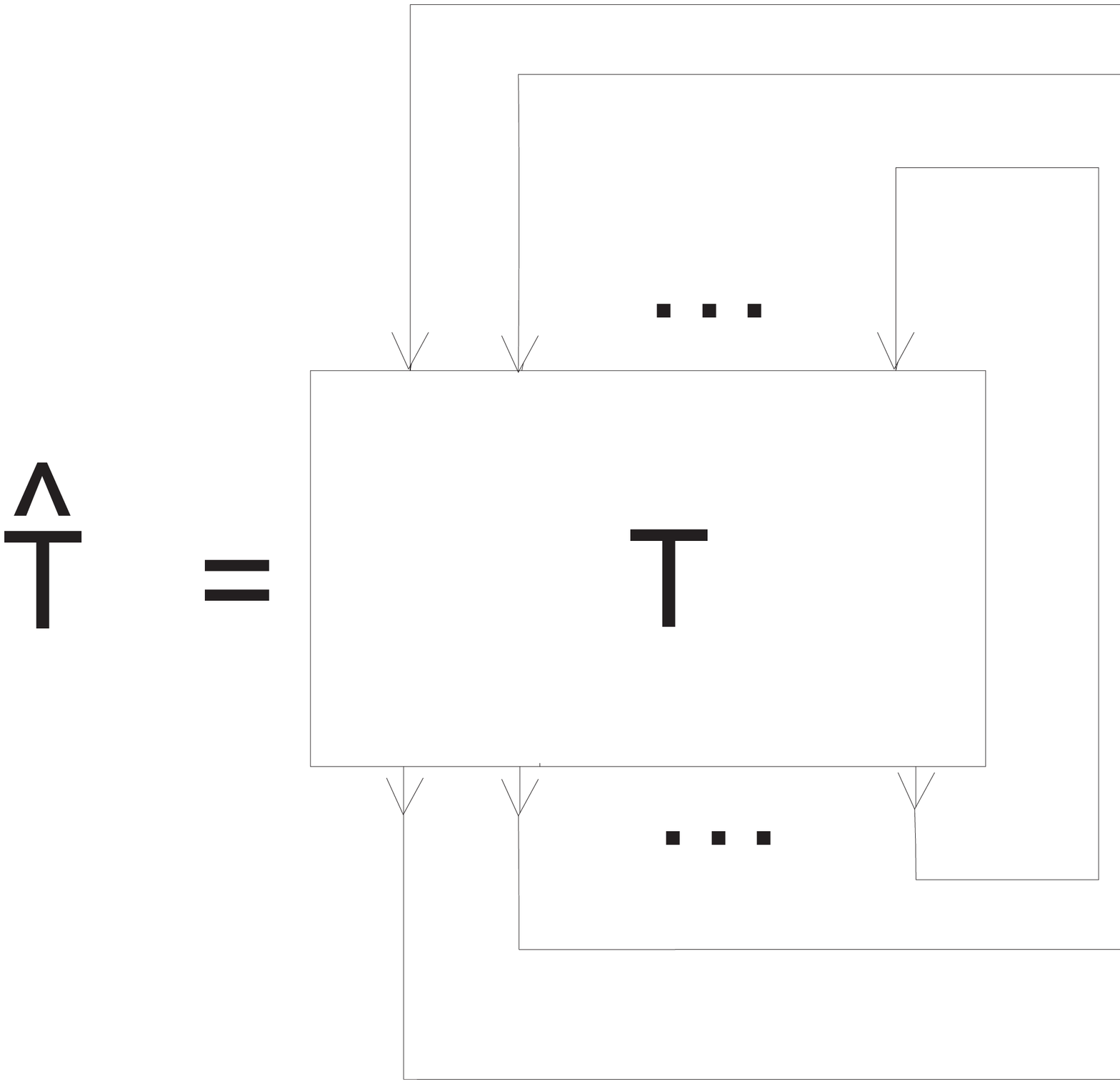}
\end{center}
\caption{}
\end{figure}

It is clear that $\hat{T}\in \mathcal{C}$. As an algebra, $\mathcal{C}$ is
freely generated by the set $\{A_m: m\in \mathbb{Z}\}$, $A_m$ for $m\neq 0$
is the closure of the braid $\sigma_{|m|-1}\cdots \sigma_2\sigma_1$, and $%
A_0 $ is the empty diagram \cite{Turaev2}. It follows that $\mathcal{C}$
contains two subalgebras $\mathcal{C}_{+}$ and $\mathcal{C}_{-}$ which are
generated by $\{A_m: m\in \mathbb{Z}, m\geq 0\}$ and $\{A_m:m\in \mathbb{Z},
m\leq 0\}$. We denote the image of the closure map $H_n \rightarrow \mathcal{%
C}_+$ as $\mathcal{C}_n$. Thus $\mathcal{C}_+=\cup_{n\geq 0}\mathcal{C}_n$.
The linear subspace $\mathcal{C}_n$ has a useful interpretation as the space
of symmetric polynomials of degree $n$ in variables $x_1,..,x_N$, for large
enough $N$. $\mathcal{C}_+$ can be viewed as the algebra of the symmetric
functions.

\subsection{Basic elements in the skein of annulus $\mathcal{C}_+$}

\subsubsection{Turaev's geometrical basis of $\mathcal{C}_+$}

The element $A_m\in \mathcal{C}_+$ is the closure of the braid $%
\sigma_{m-1}\cdots \sigma_2\sigma_1\in H_m$. Its mirror image $\bar{A}_m$ is
the closure of the braid $\sigma_{m-1}^{-1}\cdots \sigma_2^{-1}\sigma_1^{-1}$%
. Given a partition $\lambda=(\lambda_1,..,\lambda_{l})$ of $m$ with length $%
l$, we define the monomial $A_\lambda=A_{\lambda_1}\cdots A_{\lambda_{l}}$.
Then the monomials $\{A_\lambda\}_{\lambda \vdash m}$ becomes a basis of $%
\mathcal{C}_m$ which is called the Turaev's geometric basis of $\mathcal{C}%
_+ $.

Moreover, let $A_{i,j}$ be the closure of the braid $\sigma_{i+j}%
\sigma_{i+j-1}\cdots \sigma_{j+1}\sigma_{j}^{-1}\cdots \sigma_1^{-1}$. We
define the element $X_m$ in $\mathcal{C}_m$ as $X_m=%
\sum_{i=0}^{m-1}A_{i,m-1-i}$. There exist some explicit geometric relations
between the elements $\bar{A}_m$, $A_m$ and $X_m$ \cite{MM}.

\subsubsection{Symmetric function basis of $\mathcal{C}_+$}

The subalgebra $\mathcal{C}_+\subset \mathcal{C}$ can be interpreted as the
ring of symmetric functions in infinite variables $x_1,..,x_N,..$ \cite{L}.
The correspondence of the power sum symmetric function $p_m(x)$ in $\mathcal{%
C}_m$ is denoted by $P_m$. Moreover, we have the identity
\begin{align}
[m]P_m=zX_m.
\end{align}
Denoted by $Q_{\lambda}$ the closures of idempotent elements $e_{\lambda}$
in the Hecke algebra $H_m$ \cite{Ai}. It was showed by Lukac \cite{L} that $%
Q_\lambda$ represent the Schur functions in the interpretation as symmetric
functions. Hence $\{Q_{\lambda}\}_{\lambda\vdash m}$ forms a basis of $%
\mathcal{C}_m$. Furthermore, the Frobenius formula (\ref{Frobeniusformula})
gives
\begin{align}
Q_\lambda=\sum_{\mu}\frac{\chi_{\lambda}(C_\mu)}{z_{\mu}}P_{\mu},
\end{align}
where $P_{\mu}=\prod_{i=1}^{l(\mu)}P_{\mu_i}$.

\subsection{Notations}

For brevity, the following notations will be used throughout the paper.
\begin{align}  \label{notation}
[d]=q^{d}-q^{-d}, \ \Delta_d=\frac{q^d+q^{-d}}{2}, \ \{d\}=\frac{[d]}{[1]}.
\end{align}
In particular, $[1]=z$. For $\vec{\lambda}=(\lambda^1,\lambda^2,...,%
\lambda^L), \vec{\mu}=(\mu^1,\mu^2,...,\mu^L)\in \mathcal{P}^L$, we
introduce
\begin{align*}
\|\vec{\lambda}\|&=\sum_{\alpha=1}^L|\lambda^\alpha|, \ \vec{\lambda}%
^t=((\lambda^1)^t,...,(\lambda^L)^t), \ z_{\vec{\lambda}}=\prod_{%
\alpha=1}^{L}z_{\lambda^\alpha}, \chi_{\vec{\lambda}}(\vec{\mu}%
)=\prod_{\alpha=1}^L\chi_{\lambda^\alpha}(\mu^\alpha), \\
Q_{\vec{\lambda}}&=\otimes_{\alpha=1}^L Q_{\lambda^\alpha}, \ P_{\vec{\lambda%
}}=\otimes_{\alpha=1}^L P_{\lambda^\alpha}, \ [\vec{\mu}]=\prod_{\alpha=1}^L%
\prod_{j=1}^{l(\mu^\alpha)}[\mu_j^\alpha].
\end{align*}

\subsection{Definitions of the colored HOMFLY-PT invariants}

Let $\mathcal{L}$ be a framed link with $L$ components with a fixed
numbering. For diagrams $Q_1,..,Q_L$ in the skein model of annulus with the
positive oriented core $\mathcal{C}_+$, a link $\mathcal{L}$ decorated with $%
Q_1,...,Q_L$, denoted by $\mathcal{L}\star \otimes_{i=1}^{L} Q_i $,
is constructed by replacing every annulus $\mathcal{L}$ by the
annulus with the diagram $Q_i$ such that the orientations of the
cores match. Each $Q_i$ has a small backboard neighborhood in the
annulus which makes the decorated link
$\mathcal{L}\star\otimes_{i=1}^{L}Q_i$ into a framed link (see
Figure 5 for a framed trefoil $\mathcal{K}$ decorated by skein
element $\mathcal{Q}$).

\begin{figure}[!htb]
\begin{align*}
\mathcal{K} \qquad\qquad\qquad\quad \mathcal{Q}
\qquad\qquad\qquad\quad
\mathcal{K}\star \mathcal{Q}\\
\includegraphics[width=50 pt]{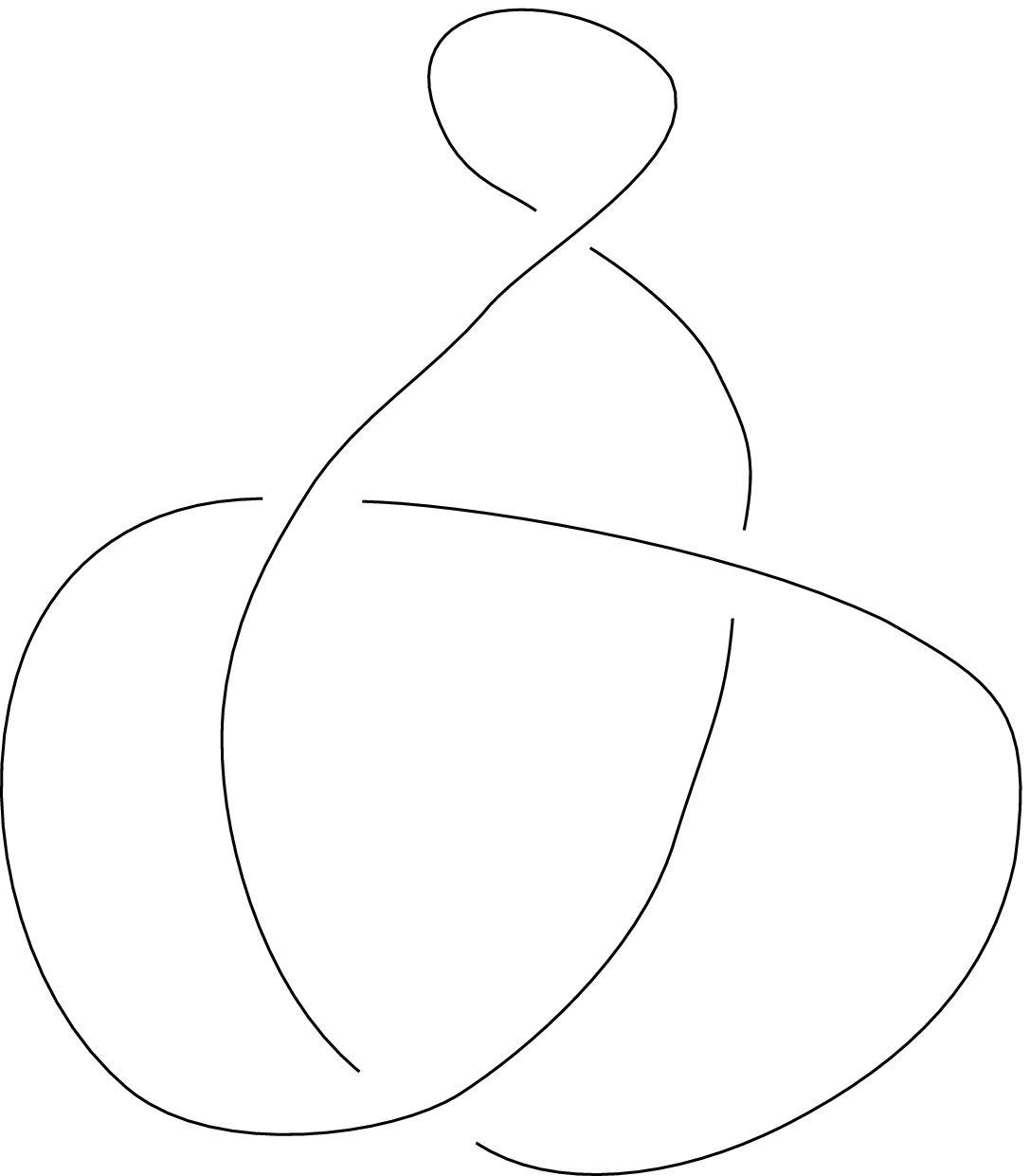}\qquad\qquad \includegraphics[width=50
pt]{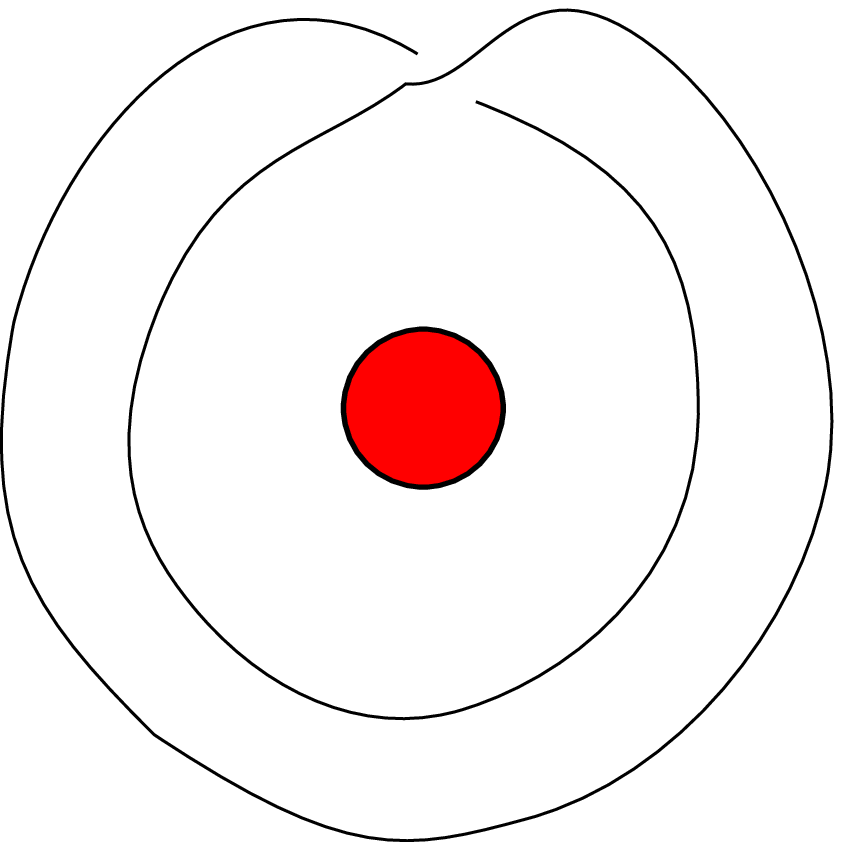} \qquad\quad \includegraphics[width=50
pt]{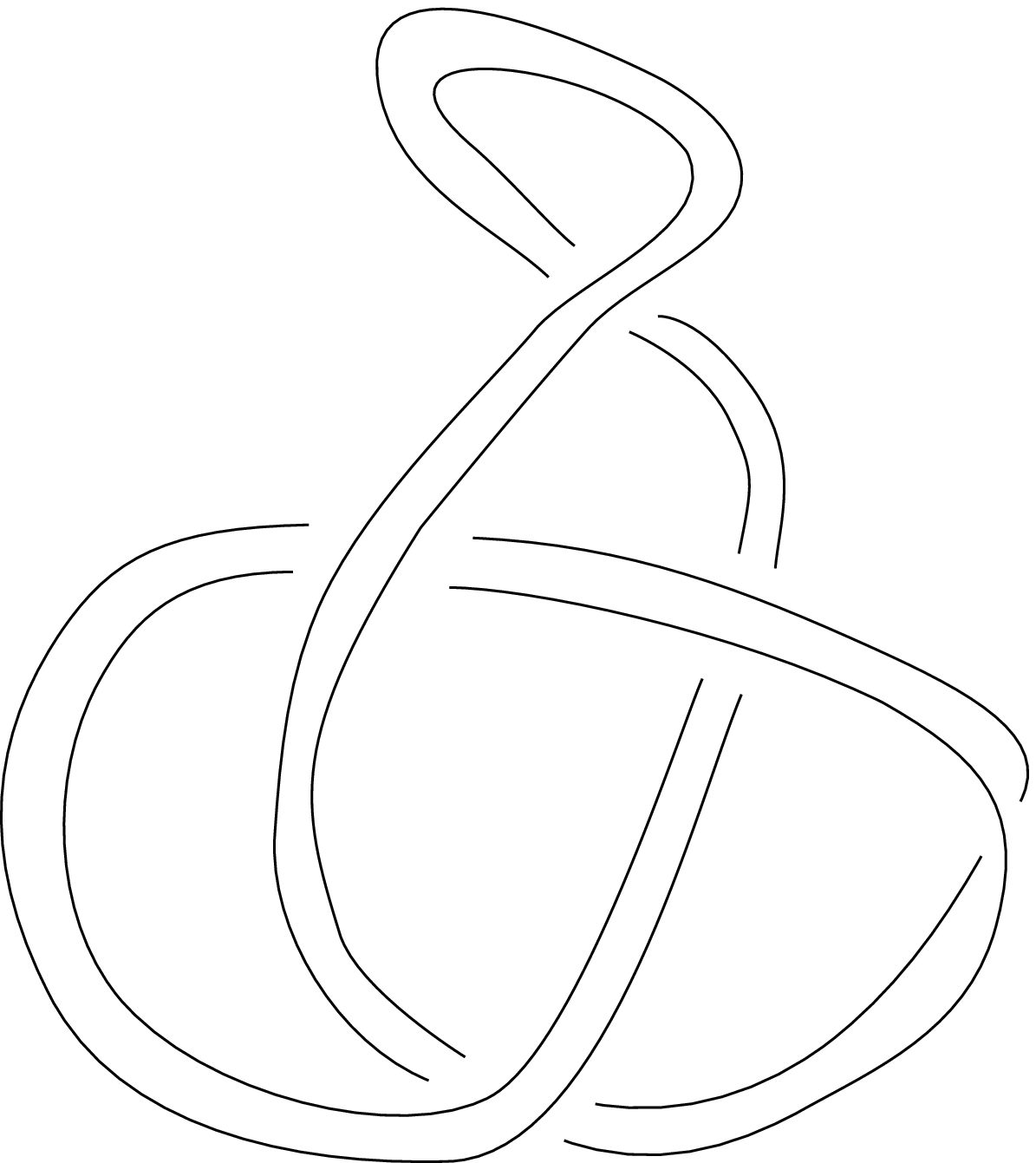}
\end{align*}
\caption{}
\end{figure}

In particular, when $Q_{\lambda^\alpha}\in \mathcal{C}_{d_\alpha}$,
where $\lambda^\alpha$ is the partition of a positive integer
$d_\alpha$, for $\alpha=1,..,L$. The
framed colored HOMFLY-PT invariant of $\mathcal{L}$ decorated by $Q_{\vec{%
\lambda}}$ is defined to be the framed HOMFLY-PT invariant of the decorated
link $\mathcal{L}\star Q_{\vec{\lambda}}$, i.e. $\mathcal{H} (\mathcal{L}%
\star Q_{\vec{\lambda}})$. By adding a framing factor to eliminate
the framing dependency
\begin{definition}
The (framing-independence) colored HOMFLY-PT invariant is given by
\begin{align}
W_{\vec{\lambda}}(\mathcal{L};q,t)=q^{-\sum_{\alpha=1}^{L}\kappa_{\lambda^\alpha}|w(\mathcal{K_\alpha})|}
t^{-\sum_{\alpha=1}^{L}|\lambda^\alpha|w(\mathcal{K}_\alpha)}\mathcal{H}(\mathcal{L}\star
Q_{\vec{\lambda}}).
\end{align}
\end{definition}
Some basic properties of the colored HOMFLY-PT invariant $W_{\vec{\lambda}}(%
\mathcal{L};q,t)$ are given in \cite{LP1,Zhu}. For convenience, we
also study the following reformulated colored HOMFLY-PT invariants.
\begin{definition}
The reformulated colored HOMFLY-PT invariants are defined as:
\begin{align} \label{reformulated invariant}
\mathcal{Z}_{\vec{\mu}}(\mathcal{L};q,t)=\sum_{\vec{\lambda}}\chi_{\vec{\lambda}}(\vec{\mu})\mathcal{H}(\mathcal{L}\star
Q_{\vec{\lambda}})=\mathcal{H}(\mathcal{L}\star P_{\vec{\mu}}), \
\check{\mathcal{Z}}_{\vec{\mu}}(\mathcal{L};q,t)=[\vec{\mu}]\mathcal{Z}_{\vec{\mu}}(\mathcal{L};q,t).
\end{align}
\end{definition}
\begin{remark}
From the view of the HOMFLY-PT skein theory, the reformulated
colored HOMFLY-PT invariants
$\mathcal{Z}_{\vec{\mu}}(\mathcal{L};q,t)$ or
$\check{\mathcal{Z}}_{\vec{\mu}}(\mathcal{L};q,t)$ are simpler than
the colored HOMFLY-PT invariant
$W_{\vec{\lambda}}(\mathcal{L};q,t)$, since the expression of
$P_{\vec{\mu}}$ is simpler than $Q_{\vec{\lambda}}$ and has the nice
property, see \cite{MM} for the similar statement. Therefore, it is
natural to study the reformulated colored HOMFLY-PT invariant
$\mathcal{Z}_{\vec{\mu}}(\mathcal{L};q,t)$ or
$\check{\mathcal{Z}}_{\vec{\mu}}(\mathcal{L};q,t)$ instead of
$W_{\vec{\lambda}}(\mathcal{L};q,t)$.
\end{remark}

\section{Integrality property}

We have introduced in Section 2, $P_m$ is the correspondence of the
symmetric power function in the skein $\mathcal{C}_+$. The geometric
representation of $P_m$ is given by
\begin{align}
[m]P_m=zX_m=z\sum_{i=0}^{m-1}A_{i,m-1-i},
\end{align}
where $A_{i,m-1-i}$ is the closure of the braid $\sigma_{m-1}\cdots
\sigma_{m-i}\sigma_{m-i-1}^{-1}\cdots \sigma_{m-i-j}^{-1}$. Let
\begin{align}  \label{Pm}
\check{P}_m=[m]P_m, \ \check{P}_{\mu}=\prod_{i=1}^{l(\mu)}\check{P}_{\mu_i}.
\end{align}
In particular, $P_1$ is the identity element in $\mathcal{C}_+$. Thus $%
\check{\mathcal{Z}}_{((1),..,(1))}(\mathcal{L};q,t)=z^L\mathcal{H}(\mathcal{L%
};q,t)$, for a link $\mathcal{L}$ with $L$-components. For convenience, in
the following, we will also use the notation $\check{\mathcal{Z}}(\mathcal{L}%
;q,t)$ to denote $\check{\mathcal{Z}}_{((1),..,(1))}(\mathcal{L};q,t)$ and
the notation $\check{\mathcal{Z}}_{p}(\mathcal{L};q,t)$ to denote $\check{%
\mathcal{Z}}_{((p),...,(p))}(\mathcal{L};q,t)$ . By the skein relation (2.4)
of $\mathcal{H}(\mathcal{L})$, we have
\begin{align}  \label{Skeinrelation}
\check{\mathcal{Z}}(\mathcal{L}_+;q,t)-\check{\mathcal{Z}}(\mathcal{L}%
_-;q,t)=z^{2\epsilon} \check{\mathcal{Z}}(\mathcal{L}_0;q,t).
\end{align}
where $\epsilon=0$ if the crossing is a self-crossing of a knot, $\epsilon=1$
if this crossing is a linking between two components of a link. Let $U$ be
the unknot, we have $\check{\mathcal{Z}}(U;q,t)=t-t^{-1}$. By using (\ref%
{Skeinrelation}), it is obvious that
\begin{lemma}\label{checkZ}
For any link $\mathcal{L}$, $\check{\mathcal{Z}}(\mathcal{L};q,t)\in
\mathbb{Z}[z^2,t^{\pm 1}]$.
\end{lemma}
Furthermore, we have
\begin{theorem}
For any link $\mathcal{L}$ with $L$ components, and
$\vec{\mu}=(\mu^1,...,\mu^L)\in \mathcal{P}^L$,
\begin{align}
\check{\mathcal{Z}}_{\vec{\mu}}(\mathcal{L};q,t)\in
\mathbb{Z}[z^2,t^{\pm 1}].
\end{align}
\end{theorem}

\begin{proof}
We first consider the knot case, for a knot $\mathcal{K}$ and a partition $%
\mu$ of length $l(\mu)=l$. We will show that $\check{\mathcal{Z}}_\mu(%
\mathcal{K})\in \mathbb{Z}[z^2,t^{\pm 1}]$.

By the formula (\ref{Pm}) and Lemma \ref{checkZ},
\begin{align}
\check{\mathcal{Z}}_{(m)}(\mathcal{K})=\mathcal{H}(\mathcal{K}\star\check{P}%
_m)=z\mathcal{H}( \mathcal{K}\star X_m )=\check{\mathcal{Z}}(\mathcal{K}%
\star X_m)\in \mathbb{Z}[z^2,t^{\pm 1}]
\end{align}
We let $\mathcal{K}_{(l)}$ be the $l$-cabling of the knot $\mathcal{K}$
which is a link of $l$ components. Then
\begin{align}
\check{\mathcal{Z}}_\mu(\mathcal{K})=\mathcal{H}(\mathcal{K}\star \check{P}%
_{\mu})=z^l\mathcal{H}(\mathcal{K}_{(l)}\star \otimes_{j=1}^l X_{\mu_j})=%
\check{\mathcal{Z}}(\mathcal{K}_{(l)}\star \otimes_{j=1}^l X_{\mu_j})\in
\mathbb{Z}[z^2,t^{\pm 1}].
\end{align}

As to the case of the link $\mathcal{L}$ with $L$ components. Let $\vec{l}%
=(l^1,...,l^L)$, where $l^\alpha=l(\mu^\alpha)$. We let $\mathcal{L}_{(\vec{l%
})}$ be the $\vec{l}$-cabling of the link $\mathcal{L}$ which is a link of $%
\sum_{\alpha=1}^Ll^\alpha$ components. Similarly, we have
\begin{align}
\check{\mathcal{Z}}_{\vec{\mu}}(\mathcal{L})&=\mathcal{H}(\mathcal{L}\star
\check{P}_{\vec{\mu}}) \\
&=z^{\sum_{\alpha=1}^L l^\alpha}\mathcal{H}( \mathcal{L}_{(\vec{l})}\star
\otimes_{\alpha=1}^L \otimes_{j=1}^{l^{\alpha}} X_{\mu^\alpha_j})  \notag \\
&=\check{\mathcal{Z}}(\mathcal{L}_{(\vec{l})}\star \otimes_{\alpha=1}^L
\otimes_{j=1}^{l^{\alpha}} X_{\mu^\alpha_j})\in \mathbb{Z}[z^2,t^{\pm 1}].
\notag
\end{align}
\end{proof}

\section{Symmetries, Level-rank duality}

In this section, we will prove the symmetries for colored HOMFLY-PT
invariants as showed in the introduction.

\begin{theorem}
Given a link $\mathcal{L}$ with $L$ components, and
$\vec{\lambda}=(\lambda^1,..,\lambda^L)\in \mathcal{P}^L$, we have
the following symmetry:
\begin{align}
W_{\vec{\lambda}}(\mathcal{L};q^{-1},t)=(-1)^{\|\vec{\lambda}\|}W_{\vec{\lambda}^{t}}(\mathcal{L};q,t)
\end{align}
\end{theorem}

\begin{proof}
We only prove the case for a knot $\mathcal{K}$, it is straightforward to
generalize the proof to link case. For $\mu\in \mathcal{P}$, by Theorem 3.2,
\begin{align}
\check{\mathcal{Z}}_\mu(\mathcal{K};q,t)\in \mathbb{Z}[z^{2}, t^{\pm 1}],
\end{align}
so $\check{\mathcal{Z}}_\mu(\mathcal{K};q^{-1},t)=\check{\mathcal{Z}}_\mu(%
\mathcal{K};q,t)$. By $\check{\mathcal{Z}}_\mu(\mathcal{K};q,t)=[\mu]%
\mathcal{Z}_{\mu}(\mathcal{K};q,t)$, we have
\begin{align}
\mathcal{Z}_{\mu}(\mathcal{K};q^{-1},t)=(-1)^{l(\mu)}\mathcal{Z}_{\mu}(%
\mathcal{K};q,t).
\end{align}
Therefore,
\begin{align}
\mathcal{H}(\mathcal{K}\star Q_\lambda;q^{-1},t)&=\sum_{\mu}\frac{%
\chi_{\lambda}(\mu)}{z_\mu}\mathcal{Z}_{\mu}(\mathcal{K};q^{-1},t) \\
&=\sum_{\mu}\frac{\chi_{\lambda}(\mu)}{z_\mu}(-1)^{l(\mu)}\mathcal{Z}_{\mu}(%
\mathcal{K};q,t)  \notag \\
&=(-1)^{|\lambda|}\sum_{\mu}\frac{\chi_{\lambda^t}(\mu)}{z_{\mu}}\mathcal{Z}%
_{\mu}(\mathcal{K};q,t)  \notag \\
&=(-1)^{|\lambda|}\mathcal{H}(\mathcal{K}\star Q_{\lambda^t};q,t)  \notag
\end{align}
where we have used the fact
\begin{align}
\chi_{\lambda^{t}}(\mu)=(-1)^{|\lambda|-l(\mu)}\chi_{\lambda}(\mu).
\end{align}
Moreover, since $\kappa_\lambda=-\kappa_{\lambda^{t}}$ and (2.11), finally,
we get
\begin{align}
W_{\lambda}(\mathcal{K};q^{-1},t)=(-1)^{|\lambda|}W_{\lambda^{t}}(\mathcal{K}%
;q,t).
\end{align}
\end{proof}

Similarly, by Theorem 3.2, we also have
\begin{align}
\mathcal{Z}_{\mu}(\mathcal{K};-q^{-1},t)=(-1)^{|\mu|+l(\mu)}\mathcal{Z}%
_{\mu}(\mathcal{K};q,t).
\end{align}
Thus, with the same procedure as in (4.4),
\begin{align}
\mathcal{H}(\mathcal{K}\star Q_\lambda;-q^{-1},t)=\mathcal{H}(\mathcal{K}%
\star Q_{\lambda^{t}};q,t).
\end{align}
Since $\kappa_{\lambda}$ is even and $\kappa_{\lambda}=-\kappa_{\lambda^t}$,
by (2.11), we have
\begin{theorem}
Given a link $\mathcal{L}$ with $L$ components, and
$\vec{\lambda}=(\lambda^1,..,\lambda^L)\in \mathcal{P}^L$, then
\begin{align}
W_{\vec{\lambda}}(\mathcal{L};-q^{-1},t)=W_{\vec{\lambda}^t}(\mathcal{L};q,t).
\end{align}
\end{theorem}

In order to show the third symmetry, for brevity, we introduce the function $%
OE$ to judge the parity of an integer $n\in \mathbb{Z}$. In other words, $%
OE(n)=odd$ for $n$ odd, and $OE(n)=even$ for $n$ even. We also use the
notation $L(\mathcal{L})$ to denote the number of the components of link $%
\mathcal{L}$.

\begin{lemma}
For a link $\mathcal{L}$, all the degrees of $t$ in
$\mathcal{H}(\mathcal{L};q,t)$ are odd if
$OE(L(\mathcal{L})+w(\mathcal{L}))=odd$ or even  if
$OE(L(\mathcal{L})+w(\mathcal{L}))=even$.
\end{lemma}

\begin{proof}
By using the skein relation
\begin{align}
\mathcal{H}(\mathcal{L}_+)-\mathcal{H}(\mathcal{L}_-)=(q-q^{-1})\mathcal{H}(%
\mathcal{L}_0)
\end{align}
to resolve the crossings, and combing the Reidemeister moves of types II and
III, any link $\mathcal{L}$ will eventually becomes a finite sum of the
links $U^{\otimes k}\otimes T^{\otimes l}$ for some $k, l\in \mathbb{Z}_+$,
where $U$ and $T$ denote the unknot and a unknot with a positive kink
respectively. We use the notation $S(\mathcal{L})$ to denote the set of all $%
U^{\otimes k}\otimes T^{\otimes l}$ that will appear in the final states
after resolving all the crossings in $\mathcal{L}$. In other words, for any
link $\mathcal{L}$, we always have the following expansion form:
\begin{align}
\mathcal{H}(\mathcal{L})=\sum_{U^{\otimes k}\otimes T^{\otimes l}\in S(%
\mathcal{L})} a_{k,l}\mathcal{H}(U^{\otimes k}\otimes T^{\otimes l})
\end{align}
where all the coefficients $a_{k,l}$ are constants (only depends on $q$)
independent of $t$.

On the other hand side, it is easy to see the following identity always
holds for any Conway triple $(\mathcal{L}_+,\mathcal{L}_-,\mathcal{L}_0)$:
\begin{align}
OE(L(\mathcal{L}_+)+w(\mathcal{L}_+))=OE(L(\mathcal{L}_-)+w(\mathcal{L}%
_-))=OE(L(\mathcal{L}_0)+w(\mathcal{L}_0))
\end{align}
Therefore, for all $U^{\otimes k}\otimes T^{\otimes l}\in S(\mathcal{L})$,
all the integers $L(U^{\otimes k}\otimes T^{\otimes l})+w(U^{\otimes
k}\otimes T^{\otimes l})=k+2l$ have the same parity. It is clear that
\begin{align}
\mathcal{H}(U^{\otimes k}\otimes T^{\otimes l})=\left(\frac{t-t^{-1}}{%
q-q^{-1}}\right)^{k+l}t^{l},
\end{align}
and the argument holds. Hence it also holds for $\mathcal{L}$ by formula
(4.11).
\end{proof}

Next, we will prove the third symmetry.
\begin{theorem}
Given a link $\mathcal{L}$ with $L$ components, and
$\vec{\lambda}=(\lambda^1,..,\lambda^L)\in \mathcal{P}^L$, we have
\begin{align}
W_{\vec{\lambda}}(\mathcal{L};q,-t)=(-1)^{\|\vec{\lambda}\|}W_{\vec{\lambda}}(\mathcal{L};q,t).
\end{align}
\end{theorem}

\begin{proof}
We only prove the case for a knot $\mathcal{K}$, it is straightforward to
write the proof for link case. By the formula (3.1) for $P_m$, we have
\begin{align}
P_{m}=\frac{1}{\{m\}}\left(\sum_{i=0}^{m-1}\widehat{\sigma_{m-1}\sigma_{m-2}%
\cdots\sigma_{m-1-i}^{-1}\cdots\sigma_2^{-1} \sigma_{1}^{-1}}\right).
\end{align}
For a general term $\widehat{\sigma_{m-1}\sigma_{m-2}\cdots%
\sigma_{m-1-i}^{-1}\cdots\sigma_2^{-1} \sigma_{1}^{-1}}$, the writhe number
is given by
\begin{align}
w(\widehat{\sigma_{m-1}\cdots\sigma_{m-1-i}^{-1}\cdots \sigma_{1}^{-1}}%
)=2i-(m-1).
\end{align}

So for a knot $\mathcal{K}$,
\begin{align}
w(\mathcal{K}\star \widehat{\sigma_{m-1}\cdots\sigma_{m-1-i}^{-1}\cdots
\sigma_{1}^{-1}})=mw(\mathcal{K})+2i-(m-1).
\end{align}
and the number of components is
\begin{align}
L(\mathcal{K}\star \widehat{\sigma_{m-1}\cdots\sigma_{m-1-i}^{-1}\cdots
\sigma_{1}^{-1}})=1.
\end{align}

Therefore,
\begin{align}
&L(\mathcal{K}\star \widehat{\sigma_{m-1}\cdots\sigma_{m-1-i}^{-1}\cdots
\sigma_{1}^{-1}})+w(\mathcal{K}\star \widehat{\sigma_{m-1}\cdots%
\sigma_{m-1-i}^{-1}\cdots \sigma_{1}^{-1}}) \\
&=m(w(\mathcal{K})-1)+2(i+1).  \notag
\end{align}
By Lemma 4.3, we have
\begin{align}
&\mathcal{H}(\mathcal{K}\star \widehat{\sigma_{m-1}\cdots\sigma_{m-1-i}^{-1}%
\cdots \sigma_{1}^{-1}};q,-t) \\
&=(-1)^{m(w(\mathcal{K})-1)}\mathcal{H}(\mathcal{K}\star \widehat{%
\sigma_{m-1}\cdots\sigma_{m-1-i}^{-1}\cdots \sigma_{1}^{-1}};q,t)  \notag
\end{align}
Hence,
\begin{align}
\mathcal{Z}_{m}(\mathcal{K};q,-t)=(-1)^{m(w(\mathcal{K})-1)}\mathcal{Z}_{m}(%
\mathcal{K};q,t).
\end{align}

More general, for a partition $\mu$, one also has
\begin{align}
\mathcal{Z}_{\mu}(\mathcal{K};q,-t)=(-1)^{|\mu|(w(\mathcal{K})-1)}\mathcal{Z}%
_{\mu}(\mathcal{K};q,t).
\end{align}
Thus,
\begin{align}
\mathcal{H}(\mathcal{K}\star Q_{\lambda},q,-t)=(-1)^{|\lambda|(w(\mathcal{K}%
)-1)}\mathcal{H}(\mathcal{K}\star Q_{\lambda},q,t)
\end{align}
Therefore,
\begin{align}
W_{\lambda}(\mathcal{K}; q,-t)&=q^{-\kappa_\lambda w(\mathcal{K}%
)}t^{-|\lambda|w(\mathcal{K})}(-1)^{-|\lambda|w(\mathcal{K})}\mathcal{H}(%
\mathcal{K}\star Q_{\lambda},q,-t) \\
&=(-1)^{|\lambda|}W_{\lambda}(\mathcal{K}; q,t)  \notag
\end{align}
\end{proof}

Combing the above two identities in Theorem 4.1 and Theorem 4.4, we obtain
the following rank-level duality as showed in \cite{LMV,LP2}:
\begin{corollary}
Given a link $\mathcal{L}$ with $L$ components, and
$\vec{\lambda}=(\lambda^1,..,\lambda^L)\in \mathcal{P}^L$,
\begin{align}
W_{\vec{\lambda}^t}(\mathcal{L};q^{-1},-t)=W_{\vec{\lambda}}(\mathcal{L};q,t).
\end{align}
\end{corollary}

\section{More results on reformulated colored HOMFLY-PT invariants}

In this section, we will provide more basic results of the
reformulated colored HOMFLY-PT invariants. These formulas are very
useful in proof of the LMOV conjecture \cite{LP1,LP3} which
motivates us to study the congruent skein relations.

\subsection{Framing change formula}

Let $\mathcal{C}_+$ be the HOMFLY-PT skein of annulus with positive
orientation. We have showed that $\{Q_\lambda\}$ forms a basis of this
skein. We define the framing map $\mathfrak{f}^\tau: \mathcal{C}%
_+\rightarrow \mathcal{C}_+ $ as the linear map defined on the basis $%
\{Q_{\lambda}\}$ by
\begin{align}
\mathfrak{f}^\tau (Q_\lambda)=q^{\kappa_\lambda
\tau}t^{|\lambda|\tau}Q_{\lambda}.
\end{align}
For $\vec{\tau}=(\tau^1,..,\tau^L)$,
\begin{align}
\mathfrak{f}^{\vec{\tau}}(Q_{\vec{\lambda}})=\otimes_{\alpha=1}^L\mathfrak{f}%
^{\tau^\alpha}(Q_{\lambda^\alpha}).
\end{align}

We define the reformulated colored HOMFLY-PT invariant $\mathcal{Z}_{\vec{\mu%
}}(\mathcal{L};q,t;\vec{\tau})$ with framing $\vec{\tau}$ as follow
\begin{align}
\mathcal{Z}_{\vec{\mu}}(\mathcal{L};q,t;\vec{\tau})=\mathcal{H}(\mathcal{L}%
\star \mathfrak{f}^{\vec{\tau}}(P_{\vec{\mu}})).
\end{align}

\begin{lemma} \label{Phi-lemma}
Given $\mu,\nu\in \mathcal{P}_+$, we define the following function
\begin{align}
\phi_{\mu,\nu}(x)=\sum_{\lambda}\chi_\lambda(\mu)\chi_\lambda(\nu)x^{\kappa_\lambda},
\end{align}
then we have $\phi_{(d),\mu}=[d\mu]_x/[d]_x$.
\end{lemma}

\begin{proof}
If $\lambda$ is a hook partition $(a+1,1,1,...,1)$, then $%
\kappa_\lambda=(a+1)a-(b+1)b$, where $b+1$ is the length of the partition
and we may also write it as $(a|b)$. Note that

\makeatletter \let\@@@alph\@alph

\makeatother
\begin{subnumcases}
{\chi_\lambda((d))=} (-1)^b, & if $\lambda$ is a hook partition
$(a|b)$,
\\\nonumber 0, & otherwise;
\end{subnumcases}
\makeatletter\let\@alph\@@@alph\makeatother  and $\kappa_{(a|b)}=(a-b)d$. By
problem $14$ at page $49$ of \cite{Mac}, taking $t=y$, we have
\begin{align}  \label{Phi-identity}
\prod_{i}\frac{1-y^{-1}x_i}{1-yx_i}=E(-y^{-1})H(y)=1+(y-y^{-1})\sum_{\mu}%
\frac{p_{\mu}(x)}{z_\mu}\sum_{a,b\geq 0}\chi_{(a|b)}(\mu)(-1)^by^{a-b}.
\end{align}
since $s_{(a|b)}(x)=\sum_{\mu}\frac{\chi_{(a|b)}(mu)}{z_\mu}p_\mu(x)$. On
the other hand,
\begin{align}
E(-y^{-1})H(y)&=\frac{H(y)}{H(y^{-1})} \\
&=\exp\left(\sum_{r\geq 1}\frac{p_r(x)}{r}(y^r-y^{-r})\right)  \notag \\
&=\prod_{r\geq 1}\exp\left(\frac{p_r(x)}{r}(y^r-y^{-r})\right)  \notag \\
&=\prod_{r\geq 1}\sum_{m_r\geq 0}\frac{p_r(x)^{m_r}(y^r-y^{-r})^{m_r}}{%
r^{m_r}m_r!}  \notag \\
&=\sum_{\mu}\frac{p_{\mu}(x)}{z_\mu}\prod_{j=1}^{l(\mu)}(y^{\mu_j}-y^{-%
\mu_j})  \notag
\end{align}
Comparing the coefficients of $p_\mu(x)$ in (\ref{Phi-identity}), we obtain
\begin{align}
\sum_{a+b+1=|\mu|}\chi_{(a|b)}(\mu)(-1)^by^{a-b}=\frac{\prod_{j=1}^{l(%
\mu)}(y^{\mu_j}-y^{-\mu_j})}{y-y^{-1}}
\end{align}
Letting $y=x^d$, we complete the proof.
\end{proof}

The definition of framing map implies
\begin{align}  \label{Framingformula}
\mathfrak{f}^\tau(P_d)&=\sum_{|\lambda|=d}\chi_{\lambda}((d))\mathfrak{f}%
^\tau_\lambda\cdot Q_\lambda \\
&=\sum_{|\lambda|=d}\chi_\lambda((d))q^{\kappa_\lambda\tau}t^{d\tau}\sum_{|%
\mu|=d}\frac{\chi_\lambda(\mu)}{z_\mu}P_\mu  \notag \\
&=t^{d\tau}\sum_{|\mu|=d}\frac{P_\mu}{z_\mu}\phi_{(d),\mu}(q^\tau).  \notag
\end{align}

Therefore, by Lemma \ref{Phi-lemma}, we obtain
\begin{theorem}[Framing change formula]
Given any framing $\tau$, we have
\begin{align}
\mathcal{Z}_{(d)}(\mathcal{K};\tau)=t^{d\tau}\sum_{|\mu|=d}\frac{\mathcal{Z}_\mu(\mathcal{K})}{z_\mu}\frac{[d\tau
\mu]}{[d\tau]},
\end{align}
for given $\tau_0$ and $\tau$, we have
\begin{align}
\mathcal{Z}_{(d)}(\mathcal{K};\tau_0+\tau)=t^{d\tau}\sum_{|\mu|=d}\frac{\mathcal{Z}_{\mu}(\mathcal{K};\tau_0)}
{z_\mu}\frac{[d\tau\mu]}{[d\tau]}.
\end{align}
\end{theorem}
As to link case, we also have
\begin{corollary}
For given $\vec{\tau}, \vec{\tau}_0$ and
$\vec{d}=((d_1),...,(d_L))$, we have
\begin{align}
\mathcal{Z}_{\vec{d}}(\mathcal{L};\vec{\tau}_0+\vec{\tau})=t^{\sum_{\alpha=1}^L
d_\alpha\tau^\alpha}\sum_{|\mu^\alpha|=d_\alpha}\frac{\mathcal{Z}_{\vec{\mu}}(\mathcal{L};\vec{\tau}_0)}{z_{\vec{\mu}}}\cdot
\prod_{\alpha=1}^L\frac{[d_\alpha\tau^\alpha\mu^\alpha]}{[d_\alpha\tau^\alpha]}.
\end{align}
\end{corollary}

\subsection{Congruent framing change formula}

By simple computations, we have the following result:
\begin{lemma} \label{lemma3.1}
If $m$ is even, then $q^{md}+q^{-md}\in \mathbb{Z}[[d]^2]$; if $m$
is odd, we have
\begin{align}
q^{md}+q^{-md}\in (q^{d}+q^{-d})+(q^{d}+q^{-d})[d]^2\cdot
\mathbb{Z}[[d]^2].
\end{align}
\end{lemma}
\begin{example}
Let $k\in \mathbb{Z}$, we have
\begin{align}
&q^{2k}+q^{-2k}\in \mathbb{Z}[[1]^2], \ q^{2k+1}+q^{-(2k+1)}\in (q+q^{-1})+(q+q^{-1})[1]^2\cdot \mathbb{Z}[[1]^2]\\
&q^{2k+1}-q^{-(2k+1)}=[1]\cdot\left(\sum_{i=1}^{k}(q^{2i}+q^{-2i})+1\right)\in
[1]\mathbb{Z}[[1]^2],\\
&[d]^2=(q^{d}-q^{-d})^2=q^{2d}+q^{-2d}-2\in \mathbb{Z}[[1]^2],\\
&\frac{q^{2k+1}\pm q^{-(2k+1)}}{q\pm q^{-1}}\in
\mathbb{Z}[[1]^2],\\
& \left(\frac{q^{d}- q^{-d}}{q- q^{-1}}\right)^{2k}\in
\mathbb{Z}[[1]^2].
\end{align}
\end{example}
\begin{definition}
Let $f(q,t), g(q,t)\in \mathbb{Z}[q^{\pm 1},t^{\pm 1}]$, if there
exists a function $h(z)\in \mathbb{Q}[z^2]$, where $z=q-q^{-1}$,
such that $f(q,t)-g(q,t)=\{d\}^2\cdot h(z)$. Then we define
\begin{align}
f(q,t)\equiv g(q,t)  \mod \{d\}^2.
\end{align}
\end{definition}

\begin{lemma}\label{algebraiclemma}
Let $m$ and $d$ are two integers, we have
\begin{align}
\Delta_{dm}&\equiv\Delta_d^m \mod\{d\}^2, \ \text{when} \ m=2k \ \text{or} \ m=1,\\
\frac{[md]}{[d]}&\equiv m \Delta_d^{m-1} \mod \{d\}^2, \ \text{when}
\ m=2k+1 \ \text{or} \ m=2.
\end{align}
\end{lemma}

\begin{proof}
It is clear when $m=1$, $m=2$, the Lemma \ref{algebraiclemma} holds. By
Lemma \ref{lemma3.1}, when $m=2k$, $\Delta_{dm}=\frac{q^{2kd}+q^{-2kd}}{2}%
=c_1[d]^{2k}+c_2[d]^{2(k-1)}+\cdots+1, $ and $\Delta_{d}^{2k}=\left(\frac{%
q^{d}+q^{-d}}{2}\right)^{2k}=c^{\prime}_1[d]^{2k}+c^{\prime}_2[d]^{2k-2}+%
\cdots+1. $ Thus $\Delta_{2kd}-\Delta_{d}^{2k}\in [d]^2\cdot \mathbb{Q}%
[[d]^2]=[d]^2\cdot \mathbb{Q}[z^2].$

Similarly, when $m=2k+1$,
\begin{align}
\frac{[(2k+1)d]}{[d]}&=\frac{q^{(2k+1)d}-q^{-(2k+1)d}}{q^{d}-q^{-d}} \\
&=q^{2kd}+q^{(2k-2)d}+\cdots+1+\cdots+q^{-(2k-2)d}+q^{-2kd}  \notag \\
&=c_1[d]^{2k}+\cdots+(2k+1).  \notag
\end{align}
and $(2k+1)\Delta_d^{2k}=c^{\prime}_{1}[d]^{2k}+\cdots+(2k+1)$. Thus $\frac{%
[(2k+1)d]}{[d]}-(2k+1)\Delta_{d}^{2k}\in [d]^2\cdot \mathbb{Q}%
[[d]^2]=[d]^2\cdot \mathbb{Q}[z^2]. $
\end{proof}

\begin{lemma} \label{algebraiclemma2}
Let $p$ be a prime number, we have
\begin{align}
\triangle_p^{p-1} \equiv (-1)^{p-1} \mod \{p\}^2.
\end{align}
\end{lemma}

\begin{proof}
By definition, it is easy to see
\begin{align}
\Delta_{p}^{2k}=\left(\frac{q^{p}+q^{-p}}{2}\right)^{2k}=1+c_1[p]^2+\cdots
c_{2k-2}[p]^{2k-2}+c_{2k}[p]^{2k}.
\end{align}
So when $p$ is an odd prime, then $\triangle_p^{p-1} \equiv 1 \mod
\{p\}^2.$

When $p=2$, we have
\begin{align}
\triangle_p^{p-1}=\triangle_2^{1}=\frac{q^2+q^{-2}}{2}=\frac{(q+q^{-1})^2}{2}%
-1\equiv -1 \mod \{2\}^2.
\end{align}
\end{proof}

By Lemma \ref{algebraiclemma}, we have for any prime number $p$,
\begin{align}
\frac{[p^2]}{p[p]}\equiv\Delta_p^{p-1} \mod \{p\}^2.
\end{align}
Furthermore, if $p$ is a prime number, $p$ is coprime to any $\mu_i$, $%
i=1,2,...,l(\mu)$, for $|\mu|=p$ and $l(\mu)\geq 2$, a direct counting of
the unit roots gives the following structure
\begin{align}
\frac{[p\mu]}{[\mu]}\in \{p\}^2 \cdot\mathbb{Z} [z^2].
\end{align}

On the other hand side, the formula (\ref{Framingformula}) shows
\begin{align}
\mathfrak{f}^{1+\tau}(\check{P}_d)=t^d\left(\mathfrak{f}^\tau(\check{P}_d)%
\frac{[d^2]}{d[d]}+\sum_{|\mu|=d} \frac{\mathfrak{f}^\tau(\check{P}_\mu)}{%
z_\mu}\frac{[d\mu]}{[\mu]}\right).
\end{align}
So combing the Lemma \ref{algebraiclemma} and \ref{algebraiclemma2}, we
obtain the following results:
\begin{theorem} \label{checkframing}
Given a knot $\mathcal{K}$ and a framing  $\tau$, if $p$ is a prime
number,  then
\begin{align}
\check{\mathcal{Z}}_{p}(\mathcal{K};\tau)\equiv
t^{p\tau}(-1)^{\tau(p-1)}\check{\mathcal{Z}}_{p}(\mathcal{K}) \mod
\{p\}^2.
\end{align}
\end{theorem}
We use the notation $\mathcal{L}^{+1}$ to denote the link obtained by adding
a positive kink to the $\alpha$-th component of link $\mathcal{L}$, then
\begin{align}
\check{\mathcal{Z}}_{p}(\mathcal{L}^{+1})=\check{\mathcal{Z}}_{p}(\mathcal{L}%
;\vec{\tau}=(0,..,\tau^\alpha=1,0..,0)).
\end{align}
Therefore, we have
\begin{theorem}
When $p$ is prime,
\begin{align}
\check{\mathcal{Z}}_{p}(\mathcal{\mathcal{L}}^{+1})\equiv
(-1)^{p-1}t^p\check{\mathcal{Z}}_{p}(\mathcal{\mathcal{L}}) \mod
\{p\}^2.
\end{align}
\end{theorem}

\section{Congruent skein relations for (reformulated) colored HOMFLY-PT
invariants}

We have showed in Section 3, the reformulated colored HOMFLY-PT invariant $%
\check{\mathcal{Z}}_p(\mathcal{L};q,t)$ has nice properties. $\check{%
\mathcal{Z}}_p(\mathcal{L};q,t)$ can be viewed as the generalization of the
classical HOMFLY-PT polynomial $\mathcal{H}(\mathcal{L};q,t)$. By this
definition, for a link $\mathcal{L}$ with $L$-components, $\check{\mathcal{Z}%
}_1(\mathcal{L};q,t)$ is equal to the classical (framing dependence)
HOMFLY-PT polynomial $\mathcal{H}(\mathcal{L};q,t)$ by multiplying a factor $%
z^L$, i.e $\check{\mathcal{Z}}_1(\mathcal{L};q,t)=z^L\mathcal{H}(\mathcal{L}%
;q,t)$. The skein relation for $\mathcal{H}(\mathcal{L};q,t)$ leads to the
skein relation for $\check{\mathcal{Z}}_1(\mathcal{L};q,t)$ as follow:
\begin{equation*}
\check{\mathcal{Z}}_1(\mathcal{L}_+;q,t)-\check{\mathcal{Z}}_1(\mathcal{L}%
_-;q,t)=\check{\mathcal{Z}}_1(\mathcal{L}_0;q,t),
\end{equation*}
when the crossing is the self-crossing of a component of the link $\mathcal{L%
}$, and
\begin{equation*}
\check{\mathcal{Z}}_1(\mathcal{L}_+;q,t)-\check{\mathcal{Z}}_1(\mathcal{L}%
_-;q,t)=z^2\check{\mathcal{Z}}_1(\mathcal{L}_0;q,t),
\end{equation*}
when the crossing is the linking of two different components of the link $%
\mathcal{L}$.

We expect $\check{\mathcal{Z}}_p$ has the similar behavior with $\check{%
\mathcal{Z}}_1$. So it is natural to ask if there exist the skein relation
for $\check{\mathcal{Z}}_p$? Based on the new approach to LMOV conjecture
\cite{LP3}, we propose the following congruent skein relations for the
reformulated colored HOMFLY-PT invariants $\check{\mathcal{Z}}_p$.
\begin{conjecture}[Congruent skein relations]
\label{higherorderskein} For any link $\mathcal{L}$  and a prime
number $p$,  we have
\begin{align} \label{Higherorderskein1}
\check{\mathcal{Z}}_{p}(\mathcal{L}_+)-\check{\mathcal{Z}}_{p}(\mathcal{L}_-)\equiv
(-1)^{p-1}\check{\mathcal{Z}}_{p}(\mathcal{L}_0) \mod \{p\}^2,
\end{align}
when the crossing is the self-crossing of a knot, and
\begin{align} \label{Higherorderskein2}
\check{\mathcal{Z}}_{p}(\mathcal{L}_+)-\check{\mathcal{Z}}_{p}(\mathcal{L}_-)\equiv
(-1)^{p-1} p[p]^2\check{\mathcal{Z}}_{p}(\mathcal{L}_{0}) \mod
 \{p\}^2[p]^2.
\end{align}
when the crossing is the linking of  two different components of the
link $\mathcal{L}$. Where the notation $A\equiv B \mod C$ denotes $
\frac{A-B}{C} \in \mathbb{Z}[(q-q^{-1})^2,t^{\pm 1}]. $
\end{conjecture}

\subsection{Confirmations}

The congruent skein relations (\ref{Higherorderskein1}) and (\ref%
{Higherorderskein2}) in Conjecture \ref{higherorderskein} have been tested
by a lot of examples (see Appendix 8 for some sample examples). In the
following, we show some partial results for Conjecture \ref{higherorderskein}%
.
\begin{theorem}
Let $\mathcal{K}_+$ be the knot obtained by adding a positive kink
to knot $\mathcal{K}$, and $\mathcal{K}_-$  be the knot obtained by
adding a negative kink to knot $\mathcal{K}$, let
$\mathcal{K}_0=\mathcal{K}\otimes U$, then
$(\mathcal{K}_+,\mathcal{K}_-,\mathcal{K}_0)$ forms a Conway triple
and the congruent skein relation (6.1) holds.
\end{theorem}

\begin{proof}
According to the Theorem \ref{checkframing}, we have
\begin{align}
\check{\mathcal{Z}}_p(\mathcal{K}_+)&=\check{\mathcal{Z}}_p(\mathcal{K}%
;\tau=1)\equiv t^p (-1)^{p-1}\check{\mathcal{Z}}_p(\mathcal{K}) \mod \{p\}^2,
\\
\check{\mathcal{Z}}_p(\mathcal{K}_-)&=\check{\mathcal{Z}}_p(\mathcal{K}%
;\tau=-1)\equiv t^{-p} (-1)^{p-1}\check{\mathcal{Z}}_p(\mathcal{K}) \mod \{%
p\}^2.  \notag
\end{align}
Therefore,
\begin{align}
\check{\mathcal{Z}}_p(\mathcal{K}_+)-\check{\mathcal{Z}}_p(\mathcal{K}%
_-)\equiv (-1)^{p-1}\check{\mathcal{Z}}_p(\mathcal{K}\otimes U) \mod \{p\}^2.
\end{align}
since $\check{\mathcal{Z}}_p(U)=t^p-t^{-p}$.
\end{proof}

\begin{theorem}
For the Conway triple $\mathcal{L}_{+}=4_1$, $\mathcal{L}_{-}$=
unknot with two negative kinks and $\mathcal{L}_{0}=T(2,-2)$ with
one positive kink. The congruent skein relation (6.1) holds for
$p=2,3$.
\end{theorem}

\begin{proof}
See examples 8.1 and 8.2 in Appendix 8.
\end{proof}

\begin{theorem}
Let $k\in \mathbb{Z}$, for the triple $\mathcal{L}_+=T(2,2k+1)$,
$\mathcal{L}_-=T(2,2k-1)$, $\mathcal{L}_0=T(2,2k)$, the congruent
skein relation (\ref{Higherorderskein1}) holds for $p=2$.

Similarly, for the triple $\mathcal{L}_+=T(2,2k)$,
$\mathcal{L}_-=T(2,2k-2)$, $\mathcal{L}_0=T(2,2k-1)$, the congruent
skein relation (\ref{Higherorderskein2}) holds for $p=2$.
\end{theorem}

\begin{proof}
In fact, we need to prove the following two formulas:
\begin{align}  \label{Higherorderskein1T2k}
\check{\mathcal{Z}}_{(2)}(T(2,2k+1))-\check{\mathcal{Z}}_{(2)}(T(2,2k-1))+%
\check{\mathcal{Z}}_{(2)(2)}(T(2,2k))\equiv 0 \mod \{2\}^{2}.
\end{align}
\begin{align}  \label{Higherorderskein2T2k}
&\check{\mathcal{Z}}_{(2)(2)}(T(2,2k))-\check{\mathcal{Z}}
_{(2)(2)}(T(2,2k-2)) \\
&+2[2]^{2}\check{\mathcal{Z}}_{(2)}(T(2,2k-1)) \equiv 0 \mod \{%
2\}^{2}[2]^{2}.  \notag
\end{align}

Given a partition $\lambda$, we define $s_{\lambda}=\sum_{\mu}\frac{%
\chi_\lambda(\mu)}{z_\mu}\prod_{i=1}^{l(\mu)}\frac{t^{\mu_i}-t^{-\mu_i}}{%
q^{\mu_i}-q^{-\mu_i}}. $ By formula (5.21) in the paper of \cite{LZ}, we
have
\begin{align*}
W_{(2)(2)}(T(2,2k))&=q^{8k}s_{(4)}+s_{(3,1)}+q^{-4k}s_{(2,2)}, \\
W_{(2)(1,1)}(T(2,2k))&=q^{4k}s_{(3,1)}+q^{-4k}s_{(2,1,1)}, \\
W_{(1,1)(1,1)}(T(2,2k))&=q^{4k}s_{(4)}+s_{(2,1,1)}+q^{-8k}s_{(1,1,1,1)}, \\
W_{(2)}(T(2,2k+1))&=q^{4k+2}(q^{4k+2}s_{(4)}-q^{-(4k+2)}s_{(3,1)}+q^{-(8k+4)}s_{(2,2)}),
\\
W_{(1,1)}(T(2,2k+1))&=q^{-(4k+2)}(q^{8k+4}s_{(2,2)}-q^{4k+2}s_{(2,1,1)}+q^{-(4k+2)}s_{(1,1,1,1)}).
\end{align*}

Therefore, we obtain

$\check{\mathcal{Z}}_{(2)(2)}(T(2,2k))=[2]^{2}\mathcal{Z}_{(2)(2)}(T(2,2k))$

$%
=[2]^{2}(W_{(2)(2)}(T(2,2k))-2W_{(2)(1,1)}(T(2,2k))+W_{(1,1)(1,1)}(T(2,2k)))
$

$%
=[2]^{2}(q^{8k}s_{(4)}+(1-2q^{4k})s_{(3,1)}+(q^{4k}+q^{-4k})s_{(2,2)}+(1-2q^{-4k})s_{(2,1,1)}+q^{-8k}s_{(1,1,1,1)})
$

$=\frac{1}{4}(q^{4k}+2+q^{-4k})(q^{4k}-q^{-4k})\frac{t^{4}-t^{-4}}{%
q^{4}-q^{-4}}(q^{2}-q^{-2})^{2}$

$+\frac{1}{3}(q^{6k}-q^{-6k})(q^{2k}-q^{-2k})\frac{t^{3}-t^{-3}}{q^{3}-q^{-3}%
}\frac{t-t^{-1}}{q-q^{-1}}(q^{2}-q^{-2})^{2}$

$+\frac{1}{8}(q^{8k}+4q^{4k}-2+4q^{-4k}+q^{-8k})(t^{2}-t^{-2})^{2}$

$+\frac{1}{4}%
[(q^{8k}-q^{-8k})-2(q^{4k}-q^{-4k})](t^{2}-t^{-2})(t-t^{-1})^{2} \frac{%
(q+q^{-1})^{2}}{q^{2}-q^{-2}}$

$+\frac{1}{24}(q^{2k}-q^{-2k})^{4}\left(\frac{t-t^{-1}}{q-q^{-1}}\right)
^{4}(q^{2}-q^{-2})^{2}$

$=\frac{1}{4}(q^{2k}+q^{-2k})^{2}\frac{q^{4k}-q^{-4k}}{q^{4}-q^{-4}}
(t^{4}-t^{-4})(q^{2}-q^{-2})^{2}$

$+\frac{1}{3}\frac{q^{6k}-q^{-6k}}{q^{6}-q^{-6}}\frac{q^{2k}-q^{-2k}}{%
q^{2}-q^{-2}}
(q^{3}+q^{-3})(q+q^{-1})(t^{3}-t^{-3})(t-t^{-1})(q^{2}-q^{-2})^{2}$

$+\frac{1}{8}(q^{8k}+4q^{4k}-2+4q^{-4k}+q^{-8k})(t^{2}-t^{-2})^{2}$

$+\frac{1}{4}\left(\frac{q^{2k}-q^{-2k}}{q^{2}-q^{-2}}\right)^{2}\frac{%
q^{4k}-q^{-4k}}{q^{4}-q^{-4}}
(q^{2}+q^{-2})(t^{2}-t^{-2})(t-t^{-1})^{2}(q+q^{-1})^{2}(q^{2}-q^{-2})^{2}$

$+\frac{1}{24}\left( \frac{q^{2k}-q^{-2k}}{q^{2}-q^{-2}}%
\right)^{4}(q+q^{-1})^{4}\left( t-t^{-1}\right) ^{4}(q^{2}-q^{-2})^{2}$

Obviously, $\frac{q^{6k}-q^{-6k}}{q^{6}-q^{-6}}$, $\frac{q^{4k}-q^{-4k}}{%
q^{4}-q^{-4}}$, $\frac{q^{2k}-q^{-2k}}{q^{2}-q^{-2}}$ are polynomials of $%
q-q^{-1}$.

Moreover

$(q^{3}+q^{-3})(q+q^{-1})\equiv (q+q^{-1})^{2}(q^{2}-q^{-2})^{2}$

$\equiv (q+q^{-1})^{4}\left( t-t^{-1}\right) ^{4}(q^{2}-q^{-2})^{2} \equiv 0 %
\mod \{2\}^{2}[2]^{2}$.

So we get

$\check{\mathcal{Z}}_{(2)(2)}(T(2,2k))\equiv \frac{1}{4}(q^{2k}+q^{-2k})^{2}%
\frac{q^{4k}-q^{-4k}}{q^{4}-q^{-4}}(t^{4}-t^{-4})(q^{2}-q^{-2})^{2}$

$+\frac{1}{8} (q^{8k}+4q^{4k}-2+4q^{-4k}+q^{-8k})(t^{2}-t^{-2})^{2} \mod \{%
2\}^{2}[2]^{2}$

$=\frac{1}{4}(q^{2k}+q^{-2k})^{2}\frac{q^{4k}-q^{-4k}}{q^{4}-q^{-4}}%
(t^{4}-t^{-4})(q^{2}-q^{-2})^{2}$

$+\frac{1}{8}
(q^{8k}-4q^{4k}+6-4q^{-4k}+q^{-8k}+8q^{4k}-16+8q^{-4k}+8)(t^{2}-t^{-2})^{2} %
\mod \{2\}^{2}[2]^{2}$

$=\frac{1}{4}(q^{2k}+q^{-2k})^{2}\frac{q^{4k}-q^{-4k}}{q^{4}-q^{-4}}%
(t^{4}-t^{-4})(q^{2}-q^{-2})^{2}$

$+\frac{1}{8}
(q^{2k}-q^{-2k})^{4}(t^{2}-t^{-2})^{2}+(q^{2k}-q^{-2k})^{2}(t^{2}-t^{-2})^{2}+(t^{2}-t^{-2})^{2} %
\mod \{2\}^{2}[2]^{2}$

$=\frac{1}{4}(q^{2k}+q^{-2k})^{2}\frac{q^{4k}-q^{-4k}}{q^{4}-q^{-4}}
(t^{4}-t^{-4})(q^{2}-q^{-2})^{2}$

$+\frac{1}{8}\left(\frac{q^{2k}-q^{-2k}}{ q^{2}-q^{-2}}%
\right)^{4}(t^{2}-t^{-2})^{2}(q^{2}-q^{-2})^{4}$

$+\left(\frac{q^{2k}-q^{-2k} }{q^{2}-q^{-2}}%
\right)^{2}(t^{2}-t^{-2})^{2}(q^{2}-q^{-2})^{2}+(t^{2}-t^{-2})^{2} \mod \{%
2\}^{2}[2]^{2}$

$\equiv \frac{1}{4}(q^{2k}+q^{-2k})^{2}\frac{q^{4k}-q^{-4k}}{q^{4}-q^{-4}}
(t^{4}-t^{-4})(q^{2}-q^{-2})^{2}$

$+\left(\frac{q^{2k}-q^{-2k}}{q^{2}-q^{-2}}
\right)^{2}(t^{2}-t^{-2})^{2}(q^{2}-q^{-2})^{2}+(t^{2}-t^{-2})^{2} \mod \{%
2\}^{2}[2]^{2} $

We also have
\begin{align}  \label{ZT22k}
\check{\mathcal{Z}}_{(2)(2)}(T(2,2k))\equiv (t^{2}-t^{-2})^{2} \mod [2]^{2}
\end{align}

Moreover, with the similar analysis, we have
\begin{align}  \label{ZT22k1}
&\check{\mathcal{Z}}_{(2)}(T(2,2k+1))\equiv \left( \frac{1}{4}%
(q^{4k+2}+q^{-(4k+2)})^{2}\frac{t^{4}-t^{-4}}{q^{4}-q^{-4}}\right. \\
&\left.+\frac{ 1}{8}(q^{4k+2}-q^{-(4k+2)})(q^{2k+1}-q^{-(2k+1)})^{2}\left(
\frac{ t^{2}-t^{-2}}{q^{2}-q^{-2}}\right) ^{2}\right)(q^{2}-q^{-2}) \mod \{%
2\}^{2}  \notag
\end{align}

Then, it is easy to get
\begin{align}
\check{\mathcal{Z}}_{(2)}(T(2,2k+1))-\check{\mathcal{Z}}_{(2)}(T(2,2k-1))%
\equiv \frac{1}{2}(q^{2}+q^{-2})\left( t^{2}-t^{-2}\right) ^{2} \mod \{%
2\}^{2}
\end{align}

Combing the formula (\ref{ZT22k}), we obtain
\begin{align}
&\check{\mathcal{Z}}_{(2)}(T(2,2k+1))-\check{\mathcal{Z}}_{(2)}(T(2,2k-1)) \\
&\equiv \frac{1}{2}(q^{2}+q^{-2})\check{\mathcal{Z}}_{(2)(2)}(T(2,2k)) \mod
\{2\}^{2}  \notag \\
&\equiv -\check{\mathcal{Z}}_{(2)(2)}(T(2,2k)) \mod \{2\}^{2}.  \notag
\end{align}
We finish the proof of the formula (\ref{Higherorderskein1T2k}).

By the formula of $\check{\mathcal{Z}}_{(2)(2)}(T(2,2k))$ and (\ref{ZT22k1}%
), we have

$\check{\mathcal{Z}}_{(2)(2)}(T(2,2k))-\check{\mathcal{Z}}%
_{(2)(2)}(T(2,2k-2))+2[2]^{2}\check{\mathcal{Z}}_{(2)}(T(2,2k-1))$

$\equiv\frac{1}{4}%
(q^{2k}+q^{-2k}-q^{2k-2}-q^{-2k+2})(q^{2k}+q^{-2k}+q^{2k-2}+q^{-2k+2})\frac{
q^{4k}-q^{-4k}}{q^{4}-q^{-4}}(t^{4}-t^{-4})(q^{2}-q^{-2})^{2}$

$+\frac{1}{4}(q^{2k-2}+q^{-2k+2})^{2}\frac{q^{4k}-q^{-4k}-q^{4k-4}+q^{-4k+4}%
}{q^{4}-q^{-4}}(t^{4}-t^{-4})(q^{2}-q^{-2})^{2}$

$%
+(q^{2k}-q^{-2k}-q^{2k-2}+q^{-2k+2})(q^{2k}-q^{-2k}+q^{2k-2}-q^{-2k+2})(t^{2}-t^{-2})^{2}
$

$+\frac{1}{2}(q^{4k-2}+q^{-(4k-2)})^{2}\frac{t^{4}-t^{-4}}{q^{2}+q^{-2}}%
(q^{2}-q^{-2})^{2}$

$+\frac{1}{4}(q^{4k-2}-q^{-(4k-2)})(q^{2k-1}-q^{-(2k-1)})^{2}\left(
t^{2}-t^{-2}\right) ^{2}(q^{2}-q^{-2})$

$\equiv \frac{1}{4}
(q^{2k-1}-q^{-2k+1})(q-q^{-1})(q^{2k-1}+q^{-2k+1})(q+q^{-1})\frac{%
q^{4k}-q^{-4k}}{q^{4}-q^{-4}}(t^{4}-t^{-4})(q^{2}-q^{-2})^{2}$

$+\frac{1}{4}(q^{2k-2}+q^{-2k+2})^{2}\frac{(q^{4k-2}+q^{-4k+2})(q^{2}-q^{-2})%
}{q^{4}-q^{-4}}(t^{4}-t^{-4})(q^{2}-q^{-2})^{2}$

$%
+(q^{2k-1}+q^{-2k+1})(q-q^{-1})(q^{2k-1}-q^{-2k+1})(q+q^{-1})(t^{2}-t^{-2})^{2}
$

$+\frac{1}{2}(q^{4k-2}+q^{-(4k-2)})^{2}\frac{t^{4}-t^{-4}}{q^{2}+q^{-2}}%
(q^{2}-q^{-2})^{2}$

$+\frac{1}{4}(q^{4k-2}-q^{-(4k-2)})(q^{2k-1}-q^{-(2k-1)})^{2}%
\left(t^{2}-t^{-2}\right) ^{2}(q^{2}-q^{-2})$

$\equiv \frac{1}{4}(q^{2k-2}+q^{-2k+2})^{2}\frac{%
(q^{4k-2}+q^{-4k+2})(q^{2}-q^{-2})}{q^{4}-q^{-4}}
(t^{4}-t^{-4})(q^{2}-q^{-2})^{2}$

$+\frac{1}{2}(q^{4k-2}+q^{-(4k-2)})^{2}\frac{t^{4}-t^{-4}}{q^{2}+q^{-2}}%
(q^{2}-q^{-2})^{2}$

$%
+(q^{2k-1}+q^{-2k+1})(q-q^{-1})(q^{2k-1}-q^{-2k+1})(q+q^{-1})(t^{2}-t^{-2})^{2}
$

$+\frac{1}{4}(q^{4k-2}-q^{-(4k-2)})(q^{2k-1}-q^{-(2k-1)})^{2}%
\left(t^{2}-t^{-2}\right) ^{2}(q^{2}-q^{-2})$

$=I_{1}+I_{2}$

\bigskip

Next, we compute the following two terms

$I_{1}=\frac{1}{4}(q^{2k-2}+q^{-2k+2})^{2}\frac{
(q^{4k-2}+q^{-4k+2})(q^{2}-q^{-2})}{q^{4}-q^{-4}}
(t^{4}-t^{-4})(q^{2}-q^{-2})^{2}$

$+\frac{1}{2}(q^{4k-2}+q^{-(4k-2)})^{2}\frac{t^{4}-t^{-4}}{q^{2}+q^{-2}}%
(q^{2}-q^{-2})^{2}$

$=\frac{1}{4}(t^{4}-t^{-4})(q^{2}-q^{-2})^{2}\frac{q^{4k-2}+q^{-4k+2}} {%
q^{2}+q^{-2}}\left((q^{2k-2}+q^{-2k+2})^{2}+2(q^{4k-2}+q^{-(4k-2)})\right)$

$=\frac{1}{4}(t^{4}-t^{-4})(q^{2}-q^{-2})^{2}\frac{q^{4k-2}+q^{-4k+2}} {%
q^{2}+q^{-2}}\left(q^{4k-4}+q^{-4k+4}+2+2q^{4k-2}+2q^{-(4k-2)}\right)$

$=\frac{1}{4}(t^{4}-t^{-4})(q^{2}-q^{-2})^{2}\frac{q^{4k-2}+q^{-4k+2}}{
q^{2}+q^{-2}}%
\left(q^{4k-4}+q^{4k-2}+q^{-4k+4}+q^{-(4k-2)}+2+q^{4k-2}+q^{-(4k-2)}\right)$

$=\frac{1}{4}(t^{4}-t^{-4})(q^{2}-q^{-2})^{2}\frac{q^{4k-2}+q^{-4k+2}}{
q^{2}+q^{-2}}\left((q^{4k-3}+q^{-4k+3})(q+q^{-1})+(q^{2k-1}+q^{-(2k-1)})^{2}%
\right)$

$=\frac{1}{4}(t^{4}-t^{-4})(q+q^{-1})^{2}(q^{2}-q^{-2})^{2} \frac{%
q^{4k-2}+q^{-4k+2}}{q^{2}+q^{-2}}\left(\frac{q^{4k-3}+q^{-4k+3}}{q+q^{-1}}%
+\left(\frac{ q^{2k-1}+q^{-(2k-1)}}{q+q^{-1}}\right)^{2}\right)$

$\equiv 0 \mod \{2\}^{2}[2]^{2}.$

\bigskip and \bigskip

$
I_{2}=(q^{2k-1}+q^{-2k+1})(q-q^{-1})(q^{2k-1}-q^{-2k+1})(q+q^{-1})(t^{2}-t^{-2})^{2}
$

$+\frac{1}{4}(q^{4k-2}-q^{-(4k-2)})(q^{2k-1}-q^{-(2k-1)})^{2}\left(
t^{2}-t^{-2}\right) ^{2}(q^{2}-q^{-2})$

$=\frac{1}{4}\left(
t^{2}-t^{-2}%
\right)^{2}(q^{2}-q^{-2})(q^{2k-1}-q^{-2k+1})(4(q^{2k-1}+q^{-2k+1})$

$+(q^{4k-2}-q^{-(4k-2)})(q^{2k-1}-q^{-(2k-1)}) )$

$=\frac{1}{4}\left( t^{2}-t^{-2}\right)
^{2}(q^{2}-q^{-2})(q^{2k-1}-q^{-2k+1})(q^{2k-1}+q^{-2k+1})%
\left(4+(q^{2k-1}-q^{-(2k-1)})^{2}\right)$

$=\frac{1}{4}\left( t^{2}-t^{-2}\right)
^{2}(q^{2}-q^{-2})(q^{4k-2}-q^{-4k+2})(q^{2k-1}+q^{-(2k-1)})^{2}$

$=\frac{1}{4}\left( t^{2}-t^{-2}\right) ^{2}(q+q^{-1})^{2}(q^{2}-q^{-2})^{2}
\frac{q^{4k-2}-q^{-4k+2}}{q^{2}-q^{-2}}\left(\frac{q^{2k-1}+q^{-(2k-1)}}{%
q+q^{-1}}\right)^{2}$

$\equiv 0 \mod \{2\}^{2}[2]^{2}$

Therefore, we proved the formula (\ref{Higherorderskein2T2k}).
\end{proof}

\subsection{Consequence}

As the application of the congruent skein relations (\ref{Higherorderskein1}%
) and (\ref{Higherorderskein2}), we prove the following result.

We define the Adams operator $\Psi_d: \mathbb{Q}(q^\pm,
t^\pm)\longrightarrow \mathbb{Q}(q^\pm,t^\pm)$ as follow:
\begin{align}  \label{Adams}
\Psi_d(f(q,t))=f(q^d,t^d).
\end{align}
\begin{corollary}[Assuming Conjecture 6.1 is right] \label{Heckelifting}
Let $\mathcal{L}$ be a link with $L$ components
$\mathcal{K}_\alpha$, $\alpha=1,...,L$. Define
$\bar{w}(\mathcal{L})=\sum_{\alpha=1}^{L}w(\mathcal{K}_\alpha)$. For
any prime number $p$, we have
\begin{align}
\check{\mathcal{Z}}_{p}(\mathcal{L})\equiv
(-1)^{(p-1)\bar{w}(\mathcal{L})}
\Psi_p(\check{\mathcal{Z}}(\mathcal{L})) \mod \{p\}^2.
\end{align}
\end{corollary}

\begin{proof}
By the skein relation (\ref{Skeinrelation}) for $\check{\mathcal{Z}}$, when
the crossing is the linking between two different components of the link $%
\mathcal{L}$, we have
\begin{align}
\Psi_p(\check{\mathcal{Z}}(\mathcal{L}_+))-\Psi_p(\check{\mathcal{Z}}(%
\mathcal{L}_-)) =\Psi_p([1]^2\check{\mathcal{Z}}(\mathcal{L}_0))=[p]^2\Psi_p(%
\check{\mathcal{Z}}(\mathcal{L}_0)) \equiv 0 \mod \{p\}^2.
\end{align}
Let $\Psi_p(\check{\mathcal{Z}})(\mathcal{L})=\Psi_p(\check{\mathcal{Z}}(%
\mathcal{L}))$, then
\begin{align}
\Psi_p(\check{\mathcal{Z}})(\mathcal{L}_+)\equiv \Psi_p(\check{\mathcal{Z}})(%
\mathcal{L}_-) \mod \{p\}^2.
\end{align}
Similarly, when the crossing is the self-crossing of some knot component, we
have
\begin{align}
\Psi_p(\check{\mathcal{Z}})(\mathcal{L}_+)-\Psi_p(\check{\mathcal{Z}})(%
\mathcal{L}_-) \equiv \Psi_p(\check{\mathcal{Z}})(\mathcal{L}_0) \mod \{%
p\}^2.
\end{align}

By using the above two relations to resolve the crossings, and combing the
Reidemeister moves of types II and III, any link $\mathcal{L}$ will
eventually becomes a finite sum of $U^{\otimes k}\otimes T^{\otimes l}$ for
some $k, l\in \mathbb{Z}_+$, where $U$ and $T$ denotes the unknot and a
unknot with a positive kink respectively. Therefore, for any link $\mathcal{L%
}$, we have
\begin{align}  \label{Psip}
\Psi_p(\check{\mathcal{Z}}(\mathcal{L}))\equiv \sum_{k,l}a_{k,l}\Psi_p(%
\check{\mathcal{Z}}(U^{\otimes k}\otimes T^{\otimes l})) \mod \{p\}^2.
\end{align}

As to the invariant $\check{\mathcal{Z}}_p$, by the Conjecture 6.1, when the
crossing is the linking between two different components of the link, we get
\begin{align}  \label{pSkein}
\check{\mathcal{Z}}_p(\mathcal{L}_+)\equiv \check{\mathcal{Z}}_p(\mathcal{L}%
_-) \mod \{p\}^2.
\end{align}
When the crossing is the self-crossing of some knot component in the link,
we can write the relation in the following form
\begin{align}  \label{pskeinself}
\check{\mathcal{Z}}_p(\mathcal{L}_+)\equiv (-1)^{(p-1)(\bar{w}(\mathcal{L}%
_+)-\bar{w}(\mathcal{L}_-))}\check{\mathcal{Z}}_p(\mathcal{L}_-)
+(-1)^{(p-1)(\bar{w}(\mathcal{L}_+)-\bar{w}(\mathcal{L}_0))}\check{\mathcal{Z%
}}_p(\mathcal{L}_0) \mod \{p\}^2.
\end{align}

By using the relations (\ref{pSkein}) (\ref{pskeinself}) to resolve the
crossings of the link $\mathcal{L}$, combing the Reidemeister moves of types
II and III, similarly, we also have
\begin{align}  \label{Checkp}
\check{\mathcal{Z}}_p(\mathcal{L})\equiv \sum_{k,l}(-1)^{(p-1)(\bar{w}(%
\mathcal{L})-l)}a_{k,l}\check{\mathcal{Z}}_p(U^{\otimes k}\otimes T^{\otimes
l})
\end{align}
since $\bar{w}(U^{\otimes k}\otimes T^{\otimes l})=l$.

By Theorem \ref{checkframing}, we obtain
\begin{align}
\check{\mathcal{Z}}_p(T)=\check{\mathcal{Z}}_p(U;\tau=1)&\equiv t^p(-1)^{p-1}%
\check{\mathcal{Z}}_p(U) \mod \{p\}^2 \\
&= (-1)^{p-1}t^p\Psi_p(\check{\mathcal{Z}}(U)) \mod \{p\}^2  \notag \\
&=(-1)^{p-1}\Psi_p(\check{\mathcal{Z}}(T)) \mod \{p\}^2.  \notag
\end{align}
Combing the formulas (\ref{Psip}) and (\ref{Checkp}), we get
\begin{align}
\check{\mathcal{Z}}_{p}(\mathcal{L})\equiv (-1)^{(p-1)\bar{w}(\mathcal{L})}
\Psi_p(\check{\mathcal{Z}}(\mathcal{L})) \mod \{p\}^2.
\end{align}
\end{proof}

\begin{remark}
In fact, Corollary 6.5 is equivalent to the framed LMOV conjecture
in a special case.
\end{remark}

\section{Congruent skein relations for colored Jones polynomials}

\subsection{Colored Jones polynomials}

Let $J_N(\mathcal{K};q)$ be the normalized colored Jones invariant of a knot
$\mathcal{K}$ colored by the $N+1$-dimensional irreducible representation of
$SU(2)$, note that our notation $J_N(\mathcal{K};q)$ is the $J_{N+1}(%
\mathcal{K};q^2)$ in the paper of \cite{Mu}. In particular, $J_N(U;q)=1$.
Moreover, colored Jones polynomial is a special case of the colored
HOMFLY-PT invariant. We have:
\begin{align}  \label{jones-homfly}
J_{N}(\mathcal{L};q)&=\left( \frac{q^{-2lk(\mathcal{L})\kappa _{(N)}}t^{-2lk(%
\mathcal{L})N}W_{(N)(N),...,(N)}(\mathcal{L};q,t)}{s_{(N)}(q,t)}\right)
|_{t=q^{2}} \\
&=\frac{q^{-2lk(\mathcal{L})N(N-1)}q^{-4lk(\mathcal{L})N}W_{(N)(N),...,(N)}(%
\mathcal{L};q,q^{2})}{s_{(N)}(q,q^{2})}  \notag \\
&=\frac{q^{-2lk(\mathcal{L})N(N+1)}W_{(N)(N),...,(N)}(\mathcal{L},q,q^{2})}{%
s_{(N)}(q,q^{2})}  \notag
\end{align}
where $W_{(N)(N),...,(N)}(\mathcal{L},q,t)$ is the colored HOMFLY-PT
invariant defined in Section 2.

\subsection{Congruent skein relations}

Because of the above relationship between colored Jones and colored
HOMFLY-PT invariants (\ref{jones-homfly}), we search for the congruent skein
relation for colored Jones polynomial.

Fortunately, we have the following succinct result for a knot $\mathcal{K}$.
\begin{theorem}
For any positive integer $N, k$ and $N\geq k\geq 0$, we have
\begin{align} \label{typei1}
J_{N}(\mathcal{K}_{+};q)-J_{N}(\mathcal{K}_{-};q)\equiv
J_{k}(\mathcal{K}_{+};q)-J_{k}(\mathcal{K}_{-};q) \mod \
[N-k][N+k+2].
\end{align}
\end{theorem}
In fact, it is a consequence of the cyclotomic expansion for the colored
Jones polynomial due to K. Habiro \cite{Hab}. We prove the following Theorem
7.3 which is equivalent to Theorem 7.1.

For $N\geq 1$, we define
\begin{align}
C_{N,k}=\prod_{j=1}^{k}(q^{2N}+q^{-2N}-q^{2j}-q^{-2j}).
\end{align}
In particular, $C_{N+1,0}=1$, $C_{N+1,N+1}=0$, and
\begin{align*}
C_{N+1,1}&=[N+2][N], \\
C_{N+1,2}&=[N+3][N+2][N][N-1], \\
\cdots \\
C_{N+1,k}&=[N+k+1][N+k]\cdots [N+2][N][N-1]\cdots [N-k+1] \\
\cdots \\
C_{N+1,N}&=[2N+1][2N]\cdots[N+2][N][N-1]\cdots [1].
\end{align*}
In our notation, Habiro's cyclotomic expansion of colored Jones polynomial
states:
\begin{theorem}[K. Habiro \cite{Hab}]
For any knot $\mathcal{K}$, there exist $H_k\in
\mathbb{Z}[q,q^{-1}]$, independent of $N$ ($N\geq 0$). Such that
\begin{align}
J_{N}(\mathcal{K};q)=\sum_{k=0}^{N}C_{N+1,k}H_k=H_0+C_{N+1,1}H_1+C_{N+1,2}H_2+\cdots+C_{N+1,N}H_{N}.
\end{align}
In particular, $J_0(\mathcal{K};q)=H_0=1$.
\end{theorem}
In the above expansion of $J_N(\mathcal{K};q)$. For $\mathcal{K}=4_1$, we
have $H_0=H_1=\cdots =H_{N}=1$. For the unknot $U$, we have $H_0=1$, $%
H_1=\cdots =H_N=0$.
\begin{theorem}
For any knot $\mathcal{K}$, and $N\geq k\geq 0$, we have
\begin{align}
J_{N}(\mathcal{K};q)-J_k(\mathcal{K};q)\equiv 0 \mod [N-k][N+k+2].
\end{align}
\end{theorem}

\begin{proof}
It is clear that when $N=k$ the theorem holds. For $N>k\geq 0$, by direct
calculation, we have
\begin{align}
[N+2][N]-[k+2][k]&=[N-k][N+k+2], \\
[N+3][N-1]-[k+3][k-1]&=[N-k][N+k+2],  \notag \\
\cdots,  \notag \\
[N+k+1][N-k+1]-[2k+1][1]&=[N-k][N+k+2].  \notag
\end{align}
By the formula (7.4), we have
\begin{align}
J_N(\mathcal{K};q)&=\sum_{i=1}^N[N+1+i][N+i]\cdots [N+2][N]\cdots
[N-i+2][N-i+1]H_i \\
J_k(\mathcal{K};q)&=\sum_{i=1}^k[k+1+i][k+i]\cdots [k+2][k]\cdots[k-i+2][%
k-i+1]H_i.
\end{align}

\begin{align}
&J_N(\mathcal{K};q)-J_k(\mathcal{K};q) \\
&\equiv\sum_{i=1}^k\left([N+1+i][N+i]\cdots [N+2][N]\cdots
[N-i+2][N-i+1]\right.  \notag \\
&\left.-[k+1+i][k+i]\cdots [k+2][k]\cdots[k-i+2][k-i+1]\right)H_i \mod [%
N-k][N+k+2]  \notag \\
&=[N+2][N]\left(1+\sum_{i=2}^k[N+1+i][N+i]\cdots [N+3][N-1]\cdots
[N-i+2][N-i+1]\right)H_i  \notag \\
&-[k+2][k]\left(1+\sum_{i=2}^k[k+1+i][k+i]\cdots[k+3][k-1]\cdots
[k-i+2][k-i+1]\right)H_i  \notag \\
& \mod [N-k][N+k+2]  \notag \\
&\equiv [k+2][k]\sum_{i=2}^k\left([N+1+i][N+i]\cdots [N+3][N-1]\cdots
[N-i+2][N-i+1]\right.  \notag \\
&\left.-[k+1+i][k+i]\cdots[k+3][k-1] [k-i+2][k-i+1]\right)H_i \mod [%
N-k][N+k+2]  \notag \\
&\equiv \cdots  \notag \\
&\equiv [k+2][k][k+3][k-1]\cdots [2k][2][N-k][N+k+2]H_k \mod [N-k][N+k+2]
\notag \\
&\equiv 0 \mod [N-k][N+k+2].  \notag
\end{align}
\end{proof}
\begin{remark}
Essentially, this proof of Theorem 7.3 is a simple modification of
the proof for the case of figure-eight knot showed in version $2$ of this paper
(see the proof of Theorem 6.6 in arXiv:1402.3571v2).

An analogue result was also obtained in \cite{BS}. We would like to
thank P. Samuelson for pointing out this to us.
\end{remark}

\begin{remark}
The above proof shows a close relationship between the congruent
relation (7.5) and the cyclotomic expansion (7.4) for colored Jones
polynomials. By similar simple calculations, one can also easily recover a certain
weak form of the cyclotmic expansion for colored Jones polynomial. We known this right after the
submission of the second version of this article. In Section 7.4, we
also formulate the congruent relations for general $SU(n)$ quantum
invariants (see Conjecture 7.11), so it is natural to conjecture
there should exist the corresponding cyclotomic expansion for
$SU(n)$ invariants, which is recently proposed in \cite{CLZ} (See
Conjecture 2.3 in \cite{CLZ}).
\end{remark}

As a straightforward consequence of Theorem 7.3, we have
\begin{corollary}
For a knot $\mathcal{K}$, we have
\begin{align}
J_{N}(\mathcal{K};q)&\equiv 1 \mod [N][N+2].
\end{align}
\end{corollary}
Formula (7.10) shows for any knot $\mathcal{K}$, $i\in \mathbb{Z}$,
\begin{equation*}
J_{N}(\mathcal{K};e^{\frac{\pi i\sqrt{-1}}{N}})=1,\ J_{N}(\mathcal{K};e^{%
\frac{\pi i\sqrt{-1}}{N+2}})=1.
\end{equation*}
If we take $N=1$, $i=1$, we have $J_{1}(\mathcal{K};e^{\frac{\pi \sqrt{-1}}{3%
}})=1$ which is the famous result due to V. Jones (see 12.4 in \cite{Jones}
and compare the notation $J_1(\mathcal{L})$ and the Jones polynomial $V_{%
\mathcal{L}}(t)$ in \cite{Jones}, see Appendix 8.2).

\bigskip

As to the case of links, after testing a lot of examples, we propose the
following congruent skein relation for colored Jones polynomial. For
convenience, we define the following sets for $n\geq 1$.
\begin{align}
A_{n}=\{q|q^{n}=\pm 1\}, \ B_{n}=\{q|q^{n}=1\}, \ C_{n}=\{q|q^{n}=-1\}.
\end{align}

\begin{conjecture}
For any link $\mathcal{L}$ with $L$ ($L\geq 2$) components, we have
\begin{align}
&J_{N}(\mathcal{L}_{+};q)-J_{N}(\mathcal{L}_{-};q)\equiv 0 \mod [N], \label{typeI1} \\
 &J_{N}(\mathcal{L}_{+};q)-J_{N}(\mathcal{L}_{-};q)\equiv 0
\mod [N+2]\label{typeI2}.
\end{align}

Moreover:

$(i)$ If $L$ is an odd integer and $L\geq 3$, for $N>k\geq 1$, the
set of the roots of the equation
\begin{align} \label{equationJNJksu(2)}
J_{N}(\mathcal{L}_{+};q)-J_{N}(\mathcal{L}_{-};q)= J_{k}(
\mathcal{L}_{+};q)-J_{k}(\mathcal{L}_{-};q)
\end{align}
 contains the set $(A_{N-k}\cup
A_{N+k+2})-(A_{k+1}-A_{1})$.

$(ii)$ If $L$ is an even integer, for $N>k\geq 1$, the set of the
root of the equation (\ref{equationJNJksu(2)})
 contains the set $(B_{N-k}\cup
C_{N+k+2})-(A_{k+1}-A_{1})$ and the set of the root of the equation
\begin{align}\label{equationJNJksu(2)2}
J_{N}(\mathcal{L}_{-};q)-J_{N}(\mathcal{L}_{+};q)=
J_{k}(\mathcal{L}_{+};q)-J_{k}(\mathcal{L}_{-};q)
\end{align}
 contains $(C_{N-k}\cup
B_{N+k+2})-(A_{k+1}-A_{1})$.
\end{conjecture}

\subsection{Application}
As an important application of Theorem 7.3, we obtain a vanishing
result of the Reshetikhin-Turaev invariant for certain $3$
dimensional oriented closed manifolds. Let's fix notations first. We
denote by $M_{p}$ the 3-manifold obtained from $S^{3}$ by doing a
$p$-surgery ($p\in \mathbb{Z}$) along a
knot $\mathcal{K}\subset S^{3}$. Let $q(s)=e^{\frac{s\pi \sqrt{-1}}{r}}$, $%
J_{n}(\mathcal{K};q(s))$ be the $n+1$-dimensional colored Jones
polynomial of $\mathcal{K}$ discussed above evaluated at the roots
of unity $q=q(s)$. According to \cite{RT, Turaev3}, for odd $r\geq
3$, the Reshetikhin-Turaev invariants $\tau _{r}(M_{p};q(s))$ of
$M_{p}$ can be calculated by the following formula:

\begin{equation}
\tau _{r}(M_{p};q(s))=C(r)\sum_{n=0}^{r-2}\left( \sin \frac{s(n+1)\pi }{r}%
\right) ^{2}e^{-\frac{sp(n^{2}+2n)\pi
\sqrt{-1}}{2r}}J_{n}(\mathcal{K};q(s)),
\end{equation}
where $C(r)$ denotes certain function depending only on $r$.

Inspired by the numerical phenomenon observed in \cite{CY}, we prove
the following
\begin{theorem}[Vanishing of Reshetikhin-Turaev invariants]
For odd $r\geq 3$, if $p=4k+2$ ($k\in \mathbb{Z}$) and odd $s$, then
we have the vanishing of Reshetikhin-Turaev invariants as follows
\begin{align}
\tau_{r}(M_p;q(s))=0.
\end{align}
\end{theorem}

\begin{proof}
For convenience, we let
\begin{equation}
T_{n}=\left( \sin \frac{s(n+1)\pi }{r}\right)
^{2}e^{-\frac{sp(n^{2}+2n)\pi
\sqrt{-1}}{2r}}J_{n}(\mathcal{K};q(s)).
\end{equation}%
It is clear that
\begin{equation}
\left( \sin \frac{s(n+1)\pi }{r}\right) ^{2}=\left( \sin \left( s\pi -\frac{%
s(n+1)\pi }{r}\right) \right) ^{2}=\left( \sin \frac{s((r-2-n)+1)\pi }{r}%
\right) ^{2}.
\end{equation}%
By formula (7.5) in Theorem 7.3, for $0\leq n\leq r-2$, we have
\begin{equation}
J_{n}(\mathcal{K};q)\equiv J_{r-2-n}(\mathcal{K};q)\mod [r]
\end{equation}%
Thus
\begin{equation}
J_{n}(\mathcal{K};q(s))=J_{r-2-n}(\mathcal{K};q(s)).
\end{equation}%
Moreover, for $p=4k+2$, $r$ and $s$ odd, by a direct computation, we
obtain
\begin{equation}
e^{-\frac{sp(n^{2}+2n)\pi \sqrt{-1}}{2r}}=-e^{-\frac{%
sp((r-2-n)^{2}+2(r-2-n))\pi \sqrt{-1}}{2r}}
\end{equation}

Combing the formulas (7.19), (7.21), (7.22) together, hence for
$0\leq n\leq r-2$,
\begin{align}
T_{n}+T_{r-2-n}=0.
\end{align}

By odd $r\geq 3$, finally we have
\begin{equation}
\tau _{r}(M_{p};q(s))=C(r)\sum_{n=0}^{r-2}T_{n}=C(r)\sum_{n=0}^{\frac{r-3}{2}%
}\left( T_{n}+T_{r-2-n}\right) =0.
\end{equation}
\end{proof}

\begin{remark}
This vanishing result indicate that the Witten-Reshetikhin-Turaev
invariants evaluated at usual roots of unity $q(1)$ vanishes for
certain $3$-manifolds. We communicated with E. Witten after
discovery of this result, he told us there should be a physical
interpretation for this phenomenon.

In Witten's Chern-Simon quantum field theory \cite{Witten}, for an
oriented closed three-manifold $M$, the invariants $Z_k^G(M)$ was
constructed as the path integral partition function with level $k$
and gauge group $G$.  By applying the method of stationary phase to
the path integral $Z_k^G(M)$ as $k\rightarrow \infty$, Witten
conjectured that
\begin{align}
Z_k^G(M) \sim \exp\left(\sqrt{-1} \pi
d\left(\frac{\eta_{\text{grav}}}{2}+\frac{1}{12}\cdot
\frac{I(g)}{2\pi}\right) \right)\cdot
\sum_{\alpha}e^{\sqrt{-1}(k+c_2(G)/2)I(A^{\alpha})}\cdot T(A^\alpha)
\end{align}
(refer to formula (2.23) in \cite{Witten} for details). In fact, for
a suitable choice of the constant $C(r)$, the invariant
$\tau_{r}(M,q(1))$ is equal to Witten's physical definition of
$Z_k^G(M)$ for level $k=r-2$ and $G=SU(2)$. According to the
asymptotic expansion formula (7.25), Witten  \cite{Witten3}
suggested us the vanishing of $\tau_{r}(M_p,q(s))$ should be
interpreted by the cancelations of the contributions of the
different flat connections $A^{\alpha}$.
\end{remark}

\begin{remark}
We refer to \cite{CY} for a detailed discussion of the volume
conjecture for $\tau _{r}(M_{p};q(2))$.
\end{remark}

\subsection{Generalization: SU(n) quantum invariants}

We introduce the $SU(n)$ quantum invariant for a link $\mathcal{L}$ as
follow:
\begin{align}  \label{defSU(n)}
&J_{N}^{SU(n)}(\mathcal{L};q)=\left( \frac{q^{-2lk(\mathcal{L})\kappa
_{(N)}}t^{-2lk(\mathcal{L})N}W_{(N)(N),...,(N)}(\mathcal{L};q,t)}{%
s_{(N)}(q,t)}\right) |_{t=q^{n}} \\
&=\frac{q^{-2lk(\mathcal{L})N(N-1)}q^{-2nlk(\mathcal{L})N}W_{(N)(N),...,(N)}(%
\mathcal{L};q,q^{n})}{s_{(N)}(q,q^{n})}  \notag \\
& =\frac{q^{-2lk(\mathcal{L})(N(N-1)+nN)}W_{(N)(N),...,(N)}(\mathcal{L}%
,q,q^{n})}{s_{(N)}(q,q^{n})}  \notag
\end{align}
where $W_{(N)(N),...,(N)}(\mathcal{L},q,t)$ is the colored HOMFLY-PT
invariants defined in Section 2. In particular, when $n=2$, $J_N^{SU(2)}(%
\mathcal{L};q)=J_N(\mathcal{L};q)$ is the colored Jones polynomial discussed
previously.

We also propose the following congruent skein relations for the $SU(n)$
quantum invariant for $n\geq 3$.
\begin{conjecture} \label{Conjsu(n)}
For a knot $\mathcal{K}$, for any positive integer $N, k$ and $N\geq
k\geq 0$, we have
\begin{align}
J_{N}^{SU(n)}(\mathcal{K}_{+};q)-J_{N}^{SU(n)}(\mathcal{K}_{-};q)
\equiv
J_{k}^{SU(n)}(\mathcal{K}_{+};q)-J_{k}^{SU(n)}(\mathcal{K}_{-};q)
\mod \ [N-k].
\end{align}
\begin{align}
&J_{N}^{SU(n)}(\mathcal{K}_{+};q)-J_{N}^{SU(n)}(\mathcal{K}_{-};q)\\\nonumber
&\equiv
J_{k}^{SU(n)}(\mathcal{K}_{+};q)-J_{k}^{SU(n)}(\mathcal{K}_{-};q)
\mod \ [N+k+n].
\end{align}
\begin{align}
J_{N}^{SU(n)}(\mathcal{K}_{+};q)-J_{N}^{SU(n)}(\mathcal{K}_{-};q)
\equiv
J_{k}^{SU(n)}(\mathcal{K}_{+};q)-J_{k}^{SU(n)}(\mathcal{K}_{-};q)
\mod \ [n-1].
\end{align}
\end{conjecture}
However, for $n\geq 3$, the parallel results for link is very complicated.
We still don't know how to formulate it cleanly.

\begin{theorem}
The Conjecture \ref{Conjsu(n)} holds for figure-eight knot $4_1$.
\end{theorem}

\begin{proof}
By formula (4) in \cite{IMMM}, we have
\begin{align}
J_{N}^{SU(n)}(4_1)=1+\underset{s=1}{\overset{N}{\sum }}\frac{[N]!}{[s]![N-s]!%
}\underset{i=0}{\overset{s-1}{\prod }}[N+i+n][i+n-1].
\end{align}
Since $J_N^{SU(n)}(U)=J_k^{SU(n)}(U)=1$, we only need to compute $%
J_{N}^{SU(n)}(4_1)-J_{k}^{SU(n)}(4_1)$. We have

$J_{N}^{SU(n)}(4_1)-J_{k}^{SU(n)}(4_1)$

$=\left( \underset{s=1}{\overset{k}{\sum }}\frac{[N]!}{[s]![N-s]!}\underset{%
i=0}{\overset{s-1}{\prod }}[N+i+n][i+n-1]-\underset{s=1}{\overset{k}{\sum }}%
\frac{[k]!}{[s]![k-s]!}\underset{i=0}{\overset{s-1}{\prod }}%
[k+i+n][i+n-1]\right) $

$+\left( \underset{s=k+1}{\overset{N}{\sum }}\frac{[N]!}{[s]![N-s]!}\underset%
{i=0}{\overset{s-1}{\prod }}[N+i+n][i+n-1]\right) $

$=I_{1}+I_{2}$.

(i) $I_{2}\equiv 0\mod  [N+k+n][n-1]$, and $I_2 \equiv 0\mod [N-k]$. Since

$I_{2}=\left( \underset{s=k+1}{\overset{N}{\sum }}\frac{[N]!}{[s]![N-s]!}%
\underset{i=0}{\overset{s-1}{\prod }}[N+i+n][i+n-1]\right) $

$=\underset{i=0}{\overset{k}{\prod }}[N+i+n][i+n-1]\left( \underset{s=k+1}{%
\overset{N}{\sum }}\frac{[N]!}{[s]![N-s]!}\left( \underset{i=k+1}{\overset{%
s-1}{\prod }}[N+i+n][i+n-1]\right) \right) $

$\frac{[N]!}{[s]![N-s]!}$ is a quantum number and always a polynomial of $%
q^{\pm 1}$. And it is clear that $\underset{i=0}{\overset{k}{\prod }}%
[N+i+n][i+n-1] \equiv 0 \mod
 [N+k+n][n-1]$.

Next, we show $I_{2}\equiv 0 \mod [N-k]$. By the expression of $I_2$, for
the terms contributed by $s\geq N-k$, $\underset{i=0}{\overset{s-1}{\prod }}%
[N+i+n]$ contains more than $N-k$ consecutive quantum integers, which have
at least one $[p]$ , such that $(N-k)|p$. Thus $\underset{i=0}{\overset{s-1}{%
\prod }}[N+i+n]\equiv 0 \mod [N-k]$.

For the terms of $s\leq N-k-1$, we always have $s\geq k+1$, So $N-s+1\leq
N-k $. Thus we have

$\frac{[N]!}{[s]![N-s]!}\underset{i=0}{\overset{s-1}{\prod }}[N+i+n][i+n-1]$

$=[N]\cdot \cdot \cdot \lbrack N-s+1]\frac{[N+s+n-1]!}{[s]![N+n-1]!}\underset%
{i=0}{\overset{s-1}{\prod }}[i+n-1]$

Where $\frac{[N+s+n-1]!}{[s]![N+n-1]!}$ is a quantum combinatorial number,
it is a polynomial of $q^{\pm 1}$, and $[N]\cdot \cdot \cdot \lbrack N-s+1]$
contain $[N-k]$ due to the condition $N-s+1\leq N-k$.

(ii) $I_1\equiv 0 \mod [N-k][N+k+n]$, and $I_1\equiv 0 \mod [n-1]$.

$I_{1}=\left( \underset{s=1}{\overset{k}{\sum }}\frac{[N]!}{[s]![N-s]!}%
\underset{i=0}{\overset{s-1}{\prod }}[N+i+n][i+n-1]-\underset{s=1}{\overset{k%
}{\sum }}\frac{[k]!}{[s]![k-s]!}\underset{i=0}{\overset{s-1}{\prod }}%
[k+i+n][i+n-1]\right) $

$=\underset{s=1}{\overset{k}{\sum }}\underset{i=0}{\overset{s-1}{\prod }}%
[i+n-1]\left( \frac{[N]\cdot \cdot \cdot \lbrack N-s+1]}{[s]!}\underset{i=0}{%
\overset{s-1}{\prod }}[N+i+n]-\frac{[k]\cdot \cdot \cdot \lbrack k-s+1]}{[s]!%
}\underset{i=0}{\overset{s-1}{\prod }}[k+i+n]\right) $

$=\underset{s=1}{\overset{k}{\sum }}\frac{[n+s-2]!}{[s]![n-2]!}%
([N+s+n-1]\cdot \cdot \cdot \lbrack N+n][N]\cdot \cdot \cdot \lbrack
N-s+1]-[k+s+n-1]\cdot \cdot \cdot \lbrack k+n][k]\cdot \cdot \cdot \lbrack
k-s+1])$

Note that

$[N+n][N]-[k+n][k]=[N-k][N+k+n]$

$...$

$[N+n+s-1][N-s+1]-[k+n+s-1][k-s+1]=[N-k][N+k+n]$.

And $\frac{[n+s-2]!}{[s]![n-2]!}$ is a polynomial of $q^{\pm 1}$. We can
show that $I_1\equiv 0 \mod [N-k][N+k+n]$.

Next, we show that $I_1 \equiv 0 \mod [n-1]$. Since

$I_{1}=\underset{s=1}{\overset{k}{\sum }}[n+s-2]!\left( \frac{[N]!}{%
[s]![N-s]!}\underset{i=0}{\overset{s-1}{\prod }}[N+i+n]-\frac{[k]!}{%
[s]![k-s]!}\underset{i=0}{\overset{s-1}{\prod }}[k+i+n]\right) $

$[n+s-2]!$ contains the term $[n-1]$ for $s=1,..k$. And $\frac{[N]!}{%
[s]![N-s]!}$ and $\frac{[k]!}{[S]![k-s]!}$ are quantum combinatorial numbers
thus they are polynomials of $q^{\pm 1}.$ So we have $I_{1}\equiv 0 \mod [%
n-1]$. Combing (i) and (ii), we finish the proof.
\end{proof}

Now we provide a result as the consequence of the
above congruent skein relations for $SU(n)$ invariants.

\begin{corollary}[Assuming Conjecture 7.10 is right]
When $n\geq 3$, for $N\geq k\geq 0$, for a knot $\mathcal{K}$,
\begin{align}
J_{N}^{SU(n)}(\mathcal{K};q)-J_{k}^{SU(n)}(\mathcal{K};q)&\equiv 0
\mod
[N-k]\\
J_{N}^{SU(n)}(\mathcal{K};q)-J_{k}^{SU(n)}(\mathcal{K};q)&\equiv 0
\mod
[N+k+n]\\
J_{N}^{SU(n)}(\mathcal{K};q)-J_{k}^{SU(n)}(\mathcal{K};q)&\equiv 0
\mod [n-1]
\end{align}
\end{corollary}

Finally these congruent skein relations for $SU(n))$ quantum invariants leads to Volume Conjectures for $SU(n)$ quantum invariants recently studied in \cite{CLZ}.

\section{Appendix: examples for congruent skein relations}

\subsection{Examples for congruent skein relations: colored HOMFLY-PT case}

In this subsection, we provide more examples to support the Conjecture \ref%
{higherorderskein}.\newline
(i).The case of $p=2$:

\begin{example}
For the triple $\mathcal{L}_{+}=4_1, \mathcal{L}_{-}=U$ with two
negative kinks, $\mathcal{L}_{0}=T(2,-2)$ with a positive kink. The
congruent skein relation (\ref{Higherorderskein1}) holds, i.e.
\begin{align}
\check{\mathcal{Z}}_{(2)}(\mathcal{L}_{+})-\check{\mathcal{Z}}_{(2)}
(\mathcal{L}_{-})+\check{\mathcal{Z}}_{(2)(2)}(\mathcal{L}_{0})\equiv
0 \mod \{2\}^{2}
\end{align}

By the formulas (4) and (8)  in \cite{IMMM}, we have

$W_{(p)}(\mathcal{L}_{+};t,q)=\left( 1+\overset{p}{\underset{k=1}{\sum }}%
\frac{\{p\}!}{\{k\}!\{p-k\}!}\underset{i=0}{\overset{k-1}{\prod }}%
(tq^{p+i}-t^{-1}q^{-p-i})(tq^{i-1}-t^{-1}q^{-i+1})\right) \overset{p}{%
\underset{i=1}{\prod }}\frac{tq^{i-1}-t^{-1}q^{-i+1}}{[i]}$

and

$W_{(1^{p})}(\mathcal{L}_{+};t,q)=\left( 1+\overset{p}{\underset{k=1}{\sum }}%
\frac{\{p\}!}{\{k\}!\{p-k\}!}\underset{j=0}{\overset{k-1}{\prod }}%
(tq^{-p-j}-t^{-1}q^{p+j})(tq^{-j+1}-t^{-1}q^{j-1})\right) \overset{p}{%
\underset{i=1}{\prod }}\frac{tq^{1-j}-t^{-1}q^{j-1}}{[j]}$

Thus

$W_{(2)}(\mathcal{L}_{+};t,q)$

$=(1+(q+q^{-1})(tq^{2}-t^{-1}q^{-2})(tq^{-1}-t^{-1}q)$

$+(tq^{3}-t^{-1}q^{-3})(tq^{2}-t^{-1}q^{-2})(t-t^{-1})(tq^{-1}-t^{-1}q))%
\frac{(t-t^{-1})(tq-t^{-1}q^{-1})}{(q-q^{-1})((q^{2}-q^{-2}))}$

and

$W_{(1,1)}(\mathcal{L}_{+};t,q)$

$=(1+(q+q^{-1})(tq^{-2}-t^{-1}q^{2})(tq^{1}-t^{-1}q^{-1})$

$+(tq^{-3}-t^{-1}q^{3})(tq^{-2}-t^{-1}q^{2})(t-t^{-1})(tq^{1}-t^{-1}q^{-1}))%
\frac{(t-t^{-1})(tq^{-1}-t^{-1}q)}{(q-q^{-1})((q^{2}-q^{-2}))}$

So we have

$\check{\mathcal{Z}}_{(2)}(\mathcal{L}_{+})=(W_{(2)}(\mathcal{L}%
_{+};t,q)-W_{(1,1)}(\mathcal{L}_{+};t,q))(q^{2}-q^{-2})$

$%
=(-5t^{-6}+16t^{-4}-18t^{-2}+18t^{2}-16t^{4}+5t^{6})+(-5t^{-6}+24t^{-4}-29t^{-2}+29t^{2}-24t^{4}+5t^{6})z^{2}+(-t^{-6}+9t^{-4}-14t^{-2}+14t^{2}-9t^{4}+t^{6})z^{4}-(t-t^{-1})^{3}(t+t^{-1})z^{6}
$

\bigskip

$\check{\mathcal{Z}}_{(2)}(\mathcal{L}%
_{-})=(q^{-4}t^{-4}s_{(2)}-q^{4}t^{-4}s_{(1,1)})(q^2-q^{-2})$

$%
=(-5t^{-6}+8t^{-4}-3t^{-2})-(5t^{-6}-6t^{-4}+t^{-2})z^{2}+(-t^{-6}+t^{-4})z^{4}
$

\bigskip

In order to compute $\check{\mathcal{Z}}_{(2)(2)}(\mathcal{L}_{0})$,
we use the formula (5.21) in \cite{LZ}.

$W_{(2)(2)}(T(2,-2))=q^{-8}s_{(4)}+s_{(3,1)}+q^{4}s_{(2,2)}$

$W_{(2)(1,1)}(T(2,-2))=q^{-4}s_{(3,1)}+q^{4}s_{(2,1,1)}$

$W_{(1,1)(1,1)}(T(2,-2))=q^{-4}s_{(2,2)}+s_{(2,1,1)}+q^{8}s_{(1,1,1,1)}$

$\check{\mathcal{Z}}_{(2)(2)}(\mathcal{L}_{0})=(q^{2}t^{2}W_{(2)(2)}(T(2,-2))-q^{2}t^{2}W_{(2)(1,1)}(T(2,-2))-q^{-2}t^{2}W_{(1,1)(2)}(T(2,-2))+q^{-2}t^{2}W_{(1,1)(1,1)}(T(2,-2)))(q^{2}-q^{-2})^{2}
$

$%
=(-1+t^{2})^{2}(-t^{-2}+2+3t^{2})+(8t^{-2}+1-9t^{2}-t^{4}+t^{6})z^{2}+6(t^{-2}-t^{2})z^{4}+(t^{-2}-t^{2})z^{6}
$

\bigskip

Now we have

$(\check{\mathcal{Z}}_{(2)}(\mathcal{L}_{+})-\check{\mathcal{Z}}%
_{(2)}(\mathcal{L}_{-})+\check{\mathcal{Z}}_{(2)(2)}(\mathcal{L}%
_{0}))/(q+q^{-1})^{2}$

$%
=(-1+t^{2})^{2}(2t^{-4}-1+2t^{2})+(-1+t^{2})(-4t^{-4}-4t^{2}+t^{4})z^{2}-(-1+t^{2})(t^{-4}+t^{2})z^{4}
$

\end{example}

(ii).The case of $p=3$:
\begin{example}
For the triple $\mathcal{L}_{+}=4_1, \mathcal{L}_{-}=U$ with two
negative kinks, $\mathcal{L}_{0}=T(2,-2)$ with a positive kink. The
congruent skein relation (\ref{Higherorderskein1}) holds, i.e.
\begin{align}
\check{\mathcal{Z}}_{(3)}(\mathcal{L}_{+})-\check{\mathcal{Z}}_{(3)}
(\mathcal{L}_{-})+\check{\mathcal{Z}}_{(3)(3)}(\mathcal{L}_{0})\equiv
0 \mod \{3\}^{2}.
\end{align}

Since

$W_{(3)}(\mathcal{L}%
_{+};t,q)=(1+(q^{2}+1+q^{-2})(tq^{3}-t^{-1}q^{-3})(tq^{-1}-t^{-1}q)$

$+(q^{2}+1+q^{-2})(tq^{4}-t^{-1}q^{-4})(tq^{3}-t^{-1}q^{-3})(t-t^{-1})(tq^{-1}-t^{-1}q)
$

$%
+(tq^{5}-t^{-1}q^{-5})(tq^{4}-t^{-1}q^{-4})(tq^{3}-t^{-1}q^{-3})(tq-t^{-1}q^{-1})(t-t^{-1})(tq^{-1}-t^{-1}q))$

$
\frac{(t-t^{-1})(tq-t^{-1}q^{-1})(tq^{2}-t^{-1}q^{-2})}{%
(q-q^{-1})(q^{2}-q^{-2})(q^{3}-q^{-3})}$

and

$W_{(1,1,1)}(\mathcal{L}%
_{+};t,q)=(1+(q^{2}+1+q^{-2})(tq^{-3}-t^{-1}q^{3})(tq-t^{-1}q^{-1})$

$+(q^{2}+1+q^{-2})(tq^{-4}-t^{-1}q^{4})(tq^{-3}-t^{-1}q^{3})(t-t^{-1})(tq-t^{-1}q^{-1})
$

$%
+(tq^{-5}-t^{-1}q^{5})(tq^{-4}-t^{-1}q^{4})(tq^{-3}-t^{-1}q^{3})(tq^{-1}-t^{-1}q)(t-t^{-1})(tq-t^{-1}q^{-1}))$

$
\frac{(t-t^{-1})(tq^{-1}-t^{-1}q)(tq^{-2}-t^{-1}q^{2})}{%
(q-q^{-1})(q^{2}-q^{-2})(q^{3}-q^{-3})}$

By formula (18) in \cite{AMMM}, we have

$W_{(2,1)}(\mathcal{L}%
_{+};t,q)=(t^{6}+t^{-6})-(q^{6}+q^{2}-1+q^{-2}+q^{-6})(t^{4}+t^{-4})$

$+(q^{10}-q^{8}+3q^{6}-3q^{4}+5q^{2}-4+5q^{-2}-3q^{-4}+3q^{-6}-q^{-8}+q^{-10})(t^{2}+t^{-2})
$

$%
-(2q^{10}-2q^{8}+5q^{6}-6q^{4}+8q^{2}-7+8q^{-2}-6q^{-4}+5q^{-6}-2q^{-8}+2q^{-10})
$

\bigskip

$\check{\mathcal{Z}}_{3}(\mathcal{L}_{+})=(W_{(3)}(\mathcal{L%
}_{+};t,q)-W_{(2,1)}(\mathcal{L}_{+};t,q)+W_{(1,1,1)}(\mathcal{L}%
_{+};t,q))(q^{3}-q^{-3})$

$%
=-28t^{-9}+126t^{-7}-225t^{-5}+173t^{-3}-27t^{-1}+27t-173t^{3}+225t^{5}-126t^{7}+28t^{9}
$

$%
+3(-42t^{-9}+199t^{-7}-356t^{-5}+281t^{-3}+69t^{-1}+69t-281t^{3}+356t^{5}-199t^{7}+42t^{9})z^{2}
$

$%
+(-210t^{-9}+1049t^{-7}-1888t^{-5}+1468t^{-3}-399t^{-1}+399t-1468t^{3}+1888t^{5}-1049t^{7}+210t^{9})z^{4}
$

$%
+3(-55t^{-9}+305t^{-7}-555t^{-5}+415t^{-3}-109t^{-1}+109t-415t^{3}+555t^{5}-305t^{7}+55t^{9})z^{6}
$

$%
+22(t-t^{-1})^{3}(3t^{-6}-11t^{-4}-5t^{-2}-5-5t^{2}-11t^{4}+3t^{6})z^{8}+(t-t^{-1})^{3}(13t^{-6}-80t^{-4}-54t^{-2}-54-54t^{2}-80t^{4}+13t^{6})z^{10}
$

$%
+(t-t^{-1})^{3}(t^{-6}-14t^{-4}-12t^{-2}-12-12t^{2}-14t^{4}+t^{6})z^{12}-(t-t^{-1})^{3}(t^{-4}+t^{-2}+1+t^{2}+t^{4})z^{14}
$

and

$\check{\mathcal{Z}}_{3}(\mathcal{L}%
_{-})=(q^{-12}t^{-6}s_{(3)}-t^{-6}s_{(2,1)}+q^{12}t^{-6}s_{(1,1,1)})(q^{3}-q^{-3})
$

$%
=-28t^{-9}+63t^{-7}-45t^{-5}+10t^{-3}+3(-42t^{-9}+77t^{-7}-40t^{-5}+5t^{-3})z^{2}+7(-30t^{-9}+46t^{-7}-17t^{-5}+t^{-3})z^{4}+(-165t^{-9}+219t^{-7}-55t^{-5}+t^{-3})z^{6}
$

$%
-6(11t^{-9}-13t^{-7}+2t^{-5})z^{8}-(13t^{-9}-14t^{-7}+t^{-5})z^{10}+(-t^{-9}+t^{-7})z^{12}
$

\bigskip

In order to compute $\check{\mathcal{Z}}_{(3)(3)}(\mathcal{L}%
_{0})$, we use the formula (5.21) in \cite{LZ}.

$%
W_{(3)(3)}(T(2,-2))=q^{-18}s_{(6)}+q^{-6}s_{(5,1)}+q^{2}s_{(4,2)}+q^{6}s_{(3,3)}
$

$%
W_{(3)(2,1)}(T(2,-2))=q^{-12}s_{(5,1)}+q^{-4}s_{(4,2)}+s_{(4,1,1)}+q^{6}s_{(3,2,1)}
$

$W_{(3)(1,1,1)}(T(2,-2))=q^{-6}s_{(4,1,1)}+q^{6}s_{(3,1,1,1)}$

$%
W_{(2,1)(2,1)}(T(2,-2))=q^{-10}s_{(4,2)}+q^{-6}s_{(4,1,1)}+q^{-6}s_{(3,3)}+2s_{(3,2,1)}+q^{6}s_{(3,1,1,1)}+q^{6}s_{(2,2,2)}+q^{10}s_{(2,2,1,1)}
$

$%
W_{(2,1)(1,1,1)}(T(2,-2))=q^{-6}s_{(3,2,1)}+s_{(3,1,1)}+q^{4}s_{(2,2,1,1)}+q^{12}s_{(2,1,1,1,1)}
$

$%
W_{(1,1,1)(1,1,1)}(T(2,-2))=q^{-6}s_{(2,2,2)}+q^{-2}s_{(2,2,1,1)}+q^{6}s_{(2,1,1,1,1)}+q^{18}s_{(1,1,1,1,1,1)}
$

\bigskip

$\check{\mathcal{Z}}_{(3)(3)}(\mathcal{L}%
_{0})=(q^{6}t^{3}W_{(3)(3)}(T(2,-2))-q^{6}t^{3}W_{(3)(2,1)}(T(2,-2))+q^{6}t^{3}W_{(3)(1,1,1)}(T(2,-2))
$

$%
-t^{3}W_{(2,1)(3)}(T(2,-2))+t^{3}W_{(2,1)(2,1)}(T(2,-2))-t^{3}W_{(2,1)(1,1,1)}(T(2,-2))
$

$%
+q^{-6}t^{3}W_{(1,1,1)(3)}(T(2,-2))-q^{-6}t^{3}W_{(1,1,1)(2,1)}(T(2,-2))+q^{-6}t^{3}W_{(1,1,1)(1,1,1)}(T(2,-2)))(q^{3}-q^{-3})^{2}
$

$%
=(-1+t^{2})^{2}(t^{-3}-7t^{-1}+3t+2t^{3}+10t^{5})+3(9t^{-3}-2t^{-1}+7t-14t^{3}+2t^{5}-7t^{7}+5t^{9})z^{2}
$

$%
+(99t^{-3}-t^{-1}+8t-106t^{3}+t^{5}-8t^{7}+7t^{9})z^{4}+(111t^{-3}+t-112t^{3}-t^{7}+t^{9})z^{6}
$

$-54(-t^{-3}+t^{3})z^{8}-12(-t^{-3}+t^{3})z^{10}+(t^{-3}-t^{3})z^{12}$

\bigskip

Now we have

$(\check{\mathcal{Z}}_{(3)}(\mathcal{L}_{+})-\check{\mathcal{Z}}_{(3)}(\mathcal{L}_{-})-\check{\mathcal{Z}}_{(3)(3)}(\mathcal{L}_{0}))/(q^{2}+1+q^{-2})^{2}$

$%
=(-1+t^{2})^{2}(7t^{-7}-6t^{-5}-t^{-3}+2t^{-1}+6t-8t^{3}+2t^{5})+(36t^{-7}-92t^{-5}+77t^{-3}-21t^{-1}+20t-77t^{3}+102t^{5}-56t^{7}+11t^{9})z^{2}
$

$%
+(56t^{-7}-133t^{-5}+98t^{-3}-30t^{-1}+30t-98t^{3}+139t^{5}-77t^{7}+15t^{9})z^{4}+(36t^{-7}-80t^{-5}+52t^{-3}-14t^{-1}+14t-52t^{3}+81t^{5}-44t^{7}+7t^{9})z^{6}
$

$%
+(10t^{-7}-21t^{-5}+12t^{-3}-2t^{-1}+2t-12t^{3}+21t^{5}-11t^{7}+t^{9})z^{8}-(-1+t^{2})^{3}(t^{-7}+t^{-5}+t^{-3}+t^{-1}+t)z^{10}
$

\end{example}

\begin{example}
For the triple
$\mathcal{L}_{+}=T(2,3),\mathcal{L}_{-}=T(2,1),\mathcal{L}_{0}=T(2,2)$,
the congruent skein relation (\ref{Higherorderskein1}) holds, i.e.
\begin{align}
\check{\mathcal{Z}}_{(3)}(\mathcal{L}_{+})-\check{\mathcal{Z}}_{(3)}
(\mathcal{L}_{-})-\check{\mathcal{Z}}_{(3)(3)}(\mathcal{L}_{0})\equiv
0 \mod \{3\}^{2}
\end{align}
And for the triple $\mathcal{L}_{+}=T(2,4)$,
$\mathcal{L}_{-}=T(2,2)$, $\mathcal{L}_{0}=T(2,3)$, the congruent
skein relation (\ref{Higherorderskein2}) holds, i.e.
\begin{align}
\check{\mathcal{Z}}_{(3)(3)}(\mathcal{L}_{+})-\check{\mathcal{Z}}_{(3)(3)}(\mathcal{L}_{-})
-3[3]^{2}\check{\mathcal{Z}}_{(3)}(\mathcal{ L}_{0})\equiv 0 \mod
\{3\}^{2}[3]^{2}.
\end{align}

By the formulas in \cite{LZ}, we have
\begin{align}
W_{(3)(3)}(T(2,2k))&=q^{18k}s_{(6)}+q^{6k}s_{(5,1)}+q^{-2k}s_{(4,2)}+q^{-6k}s_{(3,3)}\\\nonumber
W_{(3)(2,1)}(T(2,2k))&=q^{12k}s_{(5,1)}+q^{4k}s_{(4,2)}+s_{(4,1,1)}+q^{-6k}s_{(3,2,1)}\\\nonumber
W_{(3)(1,1,1)}(T(2,2k))&=q^{6k}s_{(4,1,1)}+q^{-6k}s_{(3,1,1,1)}\\\nonumber
W_{(2,1)(2,1)}(T(2,2k))&=q^{10k}s_{(4,2)}+q^{6k}s_{(4,1,1)}+q^{6k}s_{(3,3)}\\\nonumber
&+2s_{(3,2,1)}+q^{-6k}s_{(3,1,1,1)}+q^{-6k}s_{(2,2,2)}+q^{-10k}s_{(2,2,1,1)}
\\\nonumber
W_{(2,1)(1,1,1)}(T(2,2k))&=q^{6k}s_{(3,2,1)}+s_{(3,1,1,1)}+q^{-4k}s_{(2,2,1,1)}+q^{-12k}s_{(2,1,1,1,1)}\\\nonumber
W_{(1,1,1)(1,1,1)}(T(2,2k))&=q^{6k}s_{(2,2,2)}+q^{2k}s_{(2,2,1,1)}+q^{-6k}s_{(2,1,1,1,1)}+q^{-18k}s_{(1,1,1,1,1,1)}\\\nonumber
\end{align}
When $k$ is an odd number,
\begin{align}
W_{(3)}(T(2,k))&=t^{-3k}(q^{3k}s_{(6)}-q^{-3k}s_{(5,1)}+q^{-7k}s_{(4,2)}-q^{-9k}s_{(3,3)})\\\nonumber
W_{(2,1)}(T(2,k))&=t^{-3k}(q^{5k}s_{(4,2)}-q^{3k}s_{(4,1,1)}\\\nonumber
&-q^{3k}s_{(3,3)}+q^{-3k}s_{(2,2,2)}+q^{-3k}s_{(3,1,1,1)}-q^{-5k}s_{(2,2,1,1)})\\\nonumber
W_{(1,1,1)}(T(2,k))&=t^{-3k}(q^{9k}s_{(2,2,2)}-q^{7k}s_{(2,2,1,1)}+q^{3k}s_{(2,1,1,1,1)}-q^{-3k}s_{(1,1,1,1,1,1)})
\end{align}

So we have
\begin{align}
&\check{\mathcal{Z}}_{(3)(3)}(T(2,2k))\\\nonumber
&=[3]^{2}\check{\mathcal{Z}}_{(3)(3)}(T(2,2k))\\\nonumber &=[3]^{2}
\left(W_{(3)(3)}(T(2,2k))
-2W_{(3)(2,1)}(T(2,2k))+2W_{(3)(1,1,1)}(T(2,2k))\right.\\\nonumber
&\left.+W_{(2,1)(2,1)}(T(2,2k))-2W_{(2,1)(1,1,1)}(T(2,2k))+W_{(1,1,1)(1,1,1)}(T(2,2k))\right)
\end{align}
and when $k$ is odd,
\begin{align}
&\check{\mathcal{Z}}_{(3)}(T(2,k))\\\nonumber
&=[3]\mathcal{Z}_{(3)}(T(2,k))\\\nonumber
&=[3]\left(q^{6k}t^{3k}W_{(3)}(T(2,k))-t^{3k}W_{(2,1)}(T(2,k))+q^{-6k}t^{3k}W_{(1,1,1)}(T(2,k))\right)
\end{align}

\bigskip
 $\check{\mathcal{Z}}
_{(3)(3)}(T(2,2))=(t^{-6}-2+t^{6})+27(-1+9t^{2}-18t^{4}+10t^{6})z^{2}+9(-11+117t^{2}-261t^{4}+155t^{6})z^{4}+3(-37+540t^{2}-1431t^{4}+928t^{6})z^{6}
$

$%
+9(-6+139t^{2}-458t^{4}+325t^{6})z^{8}+3(-4+181t^{2}-779t^{4}+602t^{6})z^{10}+(-1+135t^{2}-813t^{4}+679t^{6})z^{12}+9(2t^{2}-19t^{4}+17t^{6})z^{14}
$

$+(t^{2}-20t^{4}+19t^{6})z^{16}+(-t^{4}+t^{6})z^{18}$

\bigskip

$\check{\mathcal{Z}}_{(3)(3)}(T(2,4))=(t^{-6}-2+t^{6})+27(t^{-6}-12t^{-4}+78t^{-2}-268+480t^{2}-420t^{4}+141t^{6})z^{2}+9(2t^{-6}-129t^{-4}+1671t^{-2}-8345+18795t^{2}-19131t^{4}+7137t^{6})z^{4}
$

$%
+3(t^{-6}-558t^{-4}+14949t^{-2}-108364+306819t^{2}-364140t^{4}+151293t^{6})z^{6}+3(-421t^{-4}+25007t^{-2}-268408+963568t^{2}-1337015t^{4}+617269t^{6})z^{8}
$

$%
+(-544t^{-4}+79139t^{-2}-1297345+5974324t^{2}-9709091t^{4}+4953517t^{6})z^{10}+(-135t^{-4}+55768t^{-2}-1452875+8700114t^{2}-16576174t^{4}+9273302t^{6})z^{12}
$

$%
+(-18t^{-4}+26975t^{-2}-1173878+9288219t^{2}-20767112t^{4}+12625814t^{6})z^{14}+(-t^{-4}+9006t^{-2}-698302+7444837t^{2}-19569105t^{4}+12813565t^{6})z^{16}
$

$%
+(2043t^{-2}-308361+4541474t^{2}-14084721t^{4}+9849565t^{6})z^{18}+(301t^{-2}-100854+2120824t^{2}-7809921t^{4}+5789650t^{6})z^{20}
$

$%
+(26t^{-2}-24101+757380t^{2}-3345677t^{4}+2612372t^{6})z^{22}+(t^{-2}-4088+205083t^{2}-1103940t^{4}+902944t^{6})z^{24}+(-466+41357t^{2}-277822t^{4}+236931t^{6})z^{26}
$

$%
+(-32+6015t^{2}-52329t^{4}+46346t^{6})z^{28}+(-1+596t^{2}-7139t^{4}+6544t^{6})z^{30}+18(2t^{2}-37t^{4}-35t^{6})z^{32}+(t^{2}-38t^{4}+37t^{6})z^{34}+(-t^{4}+t^{6})z^{36}
$

\bigskip

$\check{\mathcal{Z}}_{(3)}(T(2,1))=(-1+9t^{2}-18t^{4}+10t^{6})+3(2t^{2}-7t^{4}+5t^{6})z^{2}+(t^{2}-8t^{4}+7t^{6})z^{4}+(-t^{4}+t^{6})z^{6}
$

\bigskip

$\check{\mathcal{Z}}_{(3)}(T(2,3))=(t^{-6}+54t^{-2}-327+675t^{2}-594t^{4}+191t^{6})+9(25t^{-2}-211+560t^{2}-599t^{4}+225t^{6})z^{2}+(321t^{-2}-4506+16365t^{2}-21783t^{4}+9603t^{6})z^{4}
$

$
+(219t^{-2}-5929+30723t^{2}-51270t^{4}+26257t^{6})z^{6}+(78t^{-2}-4809+36918t^{2}-77298t^{4}+45111t^{6})z^{8}+(14t^{-2}-2499+29722t^{2}-78184t^{4}+50947t^{6})z^{10}
$

$%
+(t^{-2}-833+16352t^{2}-54416t^{4}+38896t^{6})z^{12}+(-172+6158t^{2}-26353t^{4}+20367t^{6})z^{14}+4(-5+390t^{2}-2214t^{4}+1829t^{6})z^{16}+(-1+254t^{2}-2024t^{4}+1771t^{6})z^{18}
$

$%
+12(2t^{2}-25t^{4}+23t^{6})z^{20}+(t^{2}-26t^{4}+25t^{6})z^{22}+(-t^{4}+t^{6})z^{24}
$

\bigskip

So we have

$(\check{\mathcal{Z}}_{(3)(3)}(T(2,4))-\check{\mathcal{Z}}_{(3)(3)}(T(2,2))-3[3]^{2}\check{\mathcal{Z}}
_{(3)}(T(2,3)))/((q^{-2}+1+q^{2})^{2}(q^{3}-q^{-3})^{2})$

$%
=-4(-1+t^{2})^{3}(-t^{-4}-t^{-2}+5)+(-9t^{-4}+88t^{-2}-247+336t^{2}-253t^{4}+85t^{6})z^{2}+(-6t^{-4}+272t^{-2}-1760+4341t^{2}-4658t^{4}+1811t^{6})z^{4}
$

$%
+(-t^{-4}+351t^{-2}-4384+15606t^{2}-20968t^{4}+9396t^{6})z^{6}+(228t^{-2}-5874+30260t^{2}-50771t^{4}+26157t^{6})z^{8}+(79t^{-2}-4797+36760t^{2}-77130t^{4}+45088t^{6})z^{10}
$

$%
+(14t^{-2}-2498+29694t^{2}-78155t^{4}+50945t^{6})z^{12}+(t^{-2}-833+16350t^{2}-54414t^{4}+38896t^{6})z^{14}+(-172+6158t^{2}-26353t^{4}+20367t^{6})z^{16}
$

$%
+4(-5+390t^{2}-2214t^{4}+1829t^{6})z^{18}+(-1+254t^{2}-2024t^{4}+1771t^{6})z^{20}+12(2t^{2}-25t^{4}+23t^{6})z^{22}+(t^{2}-26t^{4}+25t^{6})z^{24}+(-t^{4}+t^{6})z^{26}
$

\bigskip

$(\check{\mathcal{Z}}_{(3)}(T(2,3))-\check{\mathcal{Z}}%
_{(3)}(T(2,1))-\check{\mathcal{Z}}%
_{(3)(3)}(T(2,2)))/(q^{-2}+1+q^{2})^{2}$

$%
=2(-1+t^{2})^{2}(3t^{-2}-12+10t^{2})+(21t^{-2}-184+483t^{2}-500t^{4}+180t^{6})z^{2}+3(7t^{-2}-121+457t^{2}-606t^{4}+263t^{6})z^{4}
$

$%
+(8t^{-2}-384+2266t^{2}-3952t^{4}+2062t^{6})z^{6}+(t^{-2}-232+2300t^{2}-5294t^{4}+3225t^{6})z^{8}+(-79+1457t^{2}-4459t^{4}+3081t^{6})z^{10}
$

$%
+(-14+575t^{2}-2395t^{4}+1834t^{6})z^{12}+(-1+137t^{2}-817t^{4}+681t^{6})z^{14}+9(2t^{2}-19t^{4}+17t^{6})z^{16}+(t^{2}-20t^{4}+19t^{6})z^{18}+(-t^{4}+t^{6})z^{20}
$
\end{example}

(iii).The case of $p=5$:
\begin{example}
For the triple
$\mathcal{L}_{+}=T(2,3),\mathcal{L}_{-}=T(2,1),\mathcal{L}_{0}=T(2,2))$,
the congruent skein relation (\ref{Higherorderskein1}) holds, i.e.
\begin{align}
\check{\mathcal{Z}}_{(5)}(\mathcal{L}_{+})-\check{\mathcal{Z}}_{(5)}
(\mathcal{L}_{-})-\check{\mathcal{Z}}_{(5)(5)}(\mathcal{L}_{0})\equiv
0 \mod \{5\}^{2}
\end{align}
And for the triple $\mathcal{L}_{+}=T(2,4)$,
$\mathcal{L}_{-}=T(2,2)$, $\mathcal{L}_{0}=T(2,3))$, the congruent
skein relation (\ref{Higherorderskein2}) holds, i.e.
\begin{align}
\check{\mathcal{Z}}_{(5)(5)}(\mathcal{L}_{+})-\check{\mathcal{Z}}_{(5)(5)}(\mathcal{L}_{-})
-5[5]^{2}\check{\mathcal{Z}}_{(5)}(\mathcal{ L}_{0})\equiv 0 \mod
\{5\}^{2}[5]^{2}.
\end{align}
\end{example}

By direct computation, we have

$(\check{\mathcal{Z}}_{(5)}(\mathcal{L}_{+})-\check{\mathcal{Z}}_{(5)} (%
\mathcal{L}_{-})-\check{\mathcal{Z}}_{(5)(5)}(\mathcal{L}_{0}))/\{5\}^{2}$

$=(t^{10}-t^8)z^{62}+(-62t^8+t^6+61t^{10})z^{60}+(-1830t^8+60t^6$

$+1770t^{10})z^{58}+(-34221t^8-t^4+1712t^6+32510t^{10})z^{56}$

$%
+(30913t^6-455183t^8-56t^4+424326t^{10})z^{54}+(396551t^6-4583657t^8+4188592t^{10}
$

$-1486t^4)z^{52}+(-36314540t^8+t^2+3846105t^6-24858t^4+32493292t^{10})z^{50}$

$+(29307304t^6-232234977t^8+50t^2+203221825t^{10}-294202t^4)z^{48}$

$+(-1220456862t^8+1043084352t^{10}+179992345t^6+1177t^2-2621012t^4)z^{46}$

$+(-18258280t^4+17344t^2-5338525175t^8+4449834979t^{10}+906931132t^6)z^{44}$

$%
+(179446t^2-101988025t^4+3796969350t^6-1-19616783719t^8+15921622949t^{10})z^{42}
$

$%
+(-464813836t^4-60952908330t^8+1385934t^2+48087760043t^{10}+13328576231t^6-42)z^{40}
$

$%
+(8294057t^2-160871020961t^8+39481848313t^6-1749662171t^4-821+123130541583t^{10})z^{38}
$

$+(-9920-5486542181t^4+39406654t^2-361664206114t^8+99119134508t^6$

$%
+267992217053t^{10})z^{36}+(-82992-693563430647t^8+211448626976t^6-14416300201t^4
$

$+496380080487t^{10}+151106377t^2)z^{34}$

$%
+(t^{-2}-510379-31858745363t^4+472885821t^2-1134700927052t^8+782334801490t^{10}
$

$%
+383752495482t^6)z^{32}+(-2390861+1216767809t^2-1582161062596t^8+592446517617t^6
$

$%
-59328308157t^4+1047828476156t^{10}+32t^{-2})z^{30}+(-8725139+466t^{-2}+777025368019t^6
$

$+2585912994t^2-93137737629t^4-1876186568938t^8+1189721750227t^{10})z^{28}$

$%
+(863677705772t^6+4090t^{-2}+4549391424t^2-123137136448t^4+1141059876065t^{10}
$

$%
-25163102-1886124677801t^8)z^{26}+(24156t^{-2}+810683722366t^6+920097890390t^{10}
$

$-57855487+6627591636t^2-136789460108t^4-1600561912953t^8)z^{24}$

$+(-106550629-1140323321609t^8+7983416794t^2+101531t^{-2}+639566830289t^6$

$-127220587250t^4+620100110874t^{10})z^{22}+(313285t^{-2}+7926724033t^2$

$-98575445285t^4-157408461-677573785149t^8+421576033097t^6$

$%
+346803568481t^{10}-t^{-4})z^{20}+(-20t^{-4}-333118990163t^8-186279992+721828t^{-2}
$

$+159585987400t^{10}+230493488079t^6+6456132403t^2-63231059535t^4)z^{18}$

$%
+(1251513t^{-2}-134226255042t^8-171t^{-4}+103609527724t^6+4284647006t^2+59816740388t^{10}
$

$%
-33310034477t^4-175876941)z^{16}+(-817t^{-4}+2296371105t^2-43828867120t^8-131567914
$

$+37883044699t^6+1633939t^{-2}+18047422989t^{10}-14268036881t^4)z^{14}$

$%
+(-77176209-2394t^{-4}+1597437t^{-2}-11440433077t^8+4321694117t^{10}+982216846t^2-
$

$4906902541t^4+11119005821t^6)z^{12}+(-4447t^{-4}+1155428t^{-2}+329954404t^2$

$-2347169045t^8+2576350203t^6+807456598t^{10}-1332775807t^4-34967334)z^{10}$

$%
+(-370283365t^8+85109386t^2+605693t^{-2}+460892423t^6-5236t^{-4}+115191088t^{10}
$

$%
-279542057t^4-11967932)z^8+(-43547558t^8-3805t^{-4}+61644973t^6-2988835-43799029t^4
$

$%
+222385t^{-2}+12179913t^{10}+16291956t^2)z^6+(-513373+2185652t^2+53830t^{-2}-1610t^{-4}
$

$-3630186t^8+5846785t^6-4851154t^4+910056t^{10})z^4+(43554t^{10}+353738t^6$

$+7603t^{-2}-193648t^8+183615t^2-340542t^4-350t^{-4}-53970)z^2$

$-2568+1008t^{10}-4984t^8+7212t^2+456t^{-2}-11372t^4-28t^{-4}+10276t^6$

\bigskip

$(\check{\mathcal{Z}}_{(5)(5)}(\mathcal{L}_{+})-\check{\mathcal{Z}}_{(5)(5)}(%
\mathcal{L}_{-}) -5[5]^{2}\check{\mathcal{Z}}_{(5)}(\mathcal{\ L}_{0}))/
\{5\}^{2}[5]^{2}$

$
=(t^{10}-t^8)z^{82}+(-82t^8+t^6+81t^{10})z^{80}+(3160t^{10}-3240t^8+80t^6)z^{78}
$

$%
+(-t^4+3082t^6+79080t^{10}-82161t^8)z^{76}+(1426501t^{10}-76t^4+76153t^6-1502578t^8)z^{74}
$

$%
+(1356126t^6-2776t^4-21113942t^8+19760592t^{10})z^{72}+(-237161450t^8+t^2+18542440t^6
$

$%
+218683907t^{10}-64898t^4)z^{70}+(1985923180t^{10}+70t^2-1091057t^4+202512874t^6
$

$%
-2187345067t^8)z^{68}+(1814661385t^6-16886185957t^8+15085573342t^{10}-14051117t^4
$

$+2347t^2)z^{66}+(13599386777t^6+97226919549t^{10}-144188490t^4+50184t^2$

$-110682168020t^8)z^{64}+(-622676497218t^8+537430795055t^{10}+86455767593t^6$

$-1210834120t^4+768691t^2-1)z^{62}+(2569099618597t^{10}+471307265746t^6$

$-3031935169121t^8+8982689t^2-8480697849t^4-62)z^{60}$

$+(2221567690004t^6-12861059368999t^8-50242521904t^4+83275982t^2$

$+10689650926748t^{10}-1831)z^{58}+(38908239513646t^{10}-254456165241t^4$

$+9112987721002t^6-34280-47767399877659t^8+628842532t^2)z^{56}$

$+(124359742487515t^{10}-1110667163394t^4+32694749054839t^6+3940874329t^2$

$-456837-155947764796452t^8)z^{54}+(t^{-2}-448846906443525t^8$

$%
+350045225793401t^{10}-4204321670316t^4+20778839744t^2-4612974+102985228093669t^6)z^{52}
$

$+(-1141340907549735t^8+52t^{-2}-13868280033793t^4$

$-36683365+869490786542508t^{10}+93126534843t^2+285625311189490t^6)z^{50}$

$+(1908459230693359t^10+698928943095350t^6-40005767279492t^4$

$+1276t^{-2}-2567739689812845t^8+357519042070t^2-235739718)z^{48}$

$+(19650t^{-2}+1510998739371643t^6+1182542982609t^2-5115084726042823t^8$

$%
+3704092114385920t^{10}-1246594151-101187424122848t^4)z^{46}+(3384462643746t^2
$

$-9024434328508976t^8+6357881069123869t^{10}+213052t^{-2}$

$-5495445720-224801369487153t^4+2887975661461182t^6)z^{44}$

$%
+(-14096763011962560t^8+9646912565800746t^{10}+1729599t^{-2}-439119404565332t^4
$

$+4880583140029513t^6+8407097673177t^2-20388705143)z^{42}+(10915896t^{-2}$

$-754448751935330t^4-19480327820832142t^8-t^{-4}+7289731963256196t^6$

$+18162027255364t^2+12926946674631778t^{10}-64103291761)z^{40}$

$+(54894613t^{-2}+15275625404331864t^{10}-171632992553-23783688726880982t^8$

$+9613858286792171t^6+34161033291281t^2-1139784419436354t^4-40t^{-4})z^{38}$

$+(223688339t^{-2}-392625011018-25608291586538062t^8$

$+15886872578360023t^10+55960383734579t^2-1512817268746841t^4$

$+11178668294513722t^6-742t^{-4})z^{36}+(14505316523918117t^{10}$

$-24260056935512494t^8+79803271893636t^2+747168231t^{-2}-1761541692181458t^4$

$+11437246955606018t^6-8474t^{-4}-768870883576)z^{34}$

$+(-1289903037088-66708t^{-4}-1795798297911490t^4+98956196993576t^2$

$-20163825193767997t^8+2061835399t^{-2}+10270447464186457t^6$

$+11591507671767851t^{10})z^{32}-4t^{-8}-338t^{-2}-248t^{-4}+50t^{-6}$

$%
+(-14653300350220663t^8+4724087849t^{-2}-1598553314097344t^4+106493619799403t^2
$

$%
-1853324401264+8069365927428887t^6+8077842717787167t^{10}-384035t^{-4})z^{30}
$

$+(-9272903852978765t^8+99199240091106t^2+5526366241264904t^6$

$+4887993809842259t^{10}-1238386986749105t^4-1674093t^{-4}+9011207874t^{-2}$

$-2277461004180)z^{28}+(2555420308750096t^{10}-5085518427559615t^8$

$%
-2387905601554+14320762881t^{-2}+79704394509673t^2+t^{-6}-831514849124637t^4
$

$+3284282163908380t^6-5645225t^{-4})z^{26}+(-2403483663828373t^8$

$+1684668285425727t^6-481508442012951t^4+18940833848t^{-2}-2128764276570$

$+26t^{-6}+54994202479399t^2-14916360t^{-4}+1147439456295254t^{10})z^{24}$

$%
+(-972341151323439t^8+32403140723866t^2+741082403066769t^6+20789337733t^{-2}
$

$-1605873010889+301t^{-6}+439450809413130t^{10}-31095489t^{-4}$

$-239010087111982t^4)z^{22}+(18846500531t^{-2}+16189839685740t^2+2045t^{-6}$

$-334008849911749t^8-1018681066011+277395293088506t^6$

$-100937619441431t^4+142361222393147t^{10}-51250778t^{-4})z^{20}$

$%
+(6798448711868t^2-96460923028023t^8+9046t^{-6}-538941876216+14014330369t^{-2}
$

$%
-35928833431175t^4+38616849666404t^10-66648731t^{-4}+87499452266458t^6)z^{18}
$

$+(8661088035698t^{10}+8466635495t^{-2}-67975558t^{-4}$

$+22974430047854t^6-23129834927889t^8+27322t^{-6}-235251910397$

$-10650650846251t^4+2371820913726t^2)z^{16}+(-4530476922105t^8$

$+1580045154036t^{10}+4101545044t^{-2}+4940929911063t^6+57477t^{-6}$

$-2588074367339t^4-83509124317+677037559787t^2-53813646t^{-4})z^{14}$

$+(229300003148t^{10}+1564454718t^{-2}-708903025177t^8+84388t^{-6}$

$-504411305064t^4-23630889732+851286729468t^6+154826506602t^2$

$-32558351t^{-4})z^{12}+(-85770150480t^8+25625376664t^{10}+457705385t^{-2}$

$+85408t^{-6}-t^{-8}-5179098177+113753161407t^6+27508443561t^2$

$-76380806484t^4-14717283t^{-4})z^{10}+(-7610048969t^8+11187223155t^6$

$+98711943t^{-2}+58037t^{-6}+3617921969t^2-840736357-10t^{-8}$

$-8540076013t^4+2091751771t^{10}-4805526t^{-4})z^8+(-93530990$

$-446708793t^8+112446010t^{10}+322602765t^2+14694780t^{-2}$

$+732388036t^6-640842508t^4+25375t^{-6}-1074640t^{-4}-35t^{-8})z^6$

$+(-12867854t^8+24093679t^6+15759915t^2-6022625+2931948t^{10}$

$-25082028t^4-150060t^{-4}+1330405t^{-2}+6670t^{-6}-50t^{-8})z^4$

$+(-33963t^{10}-10867t^{-4}-111197+929t^{-6}+145805t^8+52260t^{-2}-25t^{-8}$

$-221133t^6+73050t^2+105141t^4)z^2$

$%
+6440-3300t^{10}+16568t^8-4t^{-8}-22146t^2-338t^{-2}+37260t^4-248t^{-4}-34282t^6+50t^{-6}
$

\subsection{Comparing of the notations}

In our notation, the colored Jones polynomial is defined as

\begin{align*}
&J_{N}(\mathcal{L};q)=\left( \frac{q^{-2lk(\mathcal{L})\kappa _{(N)}}t^{-2lk(%
\mathcal{L})N}W_{(N)(N),...,(N)}(\mathcal{L};q,t)}{s_{(N)}(q,t)}\right)
|_{t=q^{2}} \\
&=\frac{q^{-2lk(\mathcal{L})N(N-1)}q^{-4lk(\mathcal{L})N}W_{(N)(N),...,(N)}(%
\mathcal{L};q,q^{2})}{s_{(N)}(q,q^{2})}  \nonumber \\
& =\frac{q^{-2lk(\mathcal{L})N(N+1)}W_{(N)(N),...,(N)}(\mathcal{L},q,q^{2})}{%
s_{(N)}(q,q^{2})}  \nonumber
\end{align*}
where $W_{(N)(N),...,(N)}(\mathcal{L},q,t)$ is the colored HOMFLY invariants
defined in Section 3.

HOMFLY polynomial is given by

$H(\mathcal{L};q)=\frac{t^{-2lk(\mathcal{L})}W_{(1)(1),...,(1)}(\mathcal{L}%
;q,t)}{s_{(1)}(q,t)}=\frac{t^{-2lk(\mathcal{L})}W_{(1)(1),...,(1)}(\mathcal{L%
};q,t)}{\frac{t-t^{-1}}{q-q^{-1}}}$

$=\frac{(q-q^{-1})t^{-2lk(\mathcal{L})}W_{(1)(1),...,(1)}(\mathcal{L};q,t)}{%
t-t^{-1}}$

\bigskip

$J_{1}(\mathcal{L};q)$ is the classical Jones polynomial satisfy the
following skein relation:

$q^{2}J_{1}(\mathcal{L}_{+};q)-q^{-2}J_{1}(\mathcal{L}%
_{-};q)=(q-q^{-1})J_{1}(\mathcal{L}_{0};q)$

\bigskip

$H(\mathcal{L};q,t)$ is the classical HOMFLY polynomial satisfy the
following skein relation:

$tH(\mathcal{L}_{+};q,t)-t^{-1}H(\mathcal{L}_{-};q,t)=(q-q^{-1})H( \mathcal{L%
}_{0};q,t)$\bigskip

Remark:

1) HOMFLY polynomial is just the same as the original one.

2) While the original Jones polynomial $V(\mathcal{L})$ satisfies the skein
relation:

$q^{-1}V(\mathcal{L}_{+})-qV(\mathcal{L}_{-})=(q^{\frac{1}{2}}-q^{-\frac{1}{2%
}})V(\mathcal{L}_{0})$

So $J_{1}(\mathcal{L})$ and $V(\mathcal{L})$ is related by: $J_{1}(\mathcal{L%
};q)=(-1)^{L-1}V(\mathcal{L};q^{-2})$.

\bigskip

We list a table of examples:

$%
\begin{array}{ccc}
& HOMFLY\text{ }H(\mathcal{L};q) & Jones\text{ }V(\mathcal{L};q^{2}) \\
3_1=T(2,-3) & 2t^{2}-t^{4}+t^{2}z^{2} & -q^{-8}+q^{-6}+q^{-2} \\
T(2,3) & 2t^{-2}-t^{-4}+t^{-2}z^{2} & q^{2}+q^{6}-q^{8} \\
4\_1 & t^{2}+t^{-2}-1-z^{2} & q^{-4}-q^{-2}+1-q^{2}+q^{4} \\
5\_2 & t^{2}+t^{4}-t^{6}+(t^{2}+t^{4})z^{2} &
-q^{-12}+q^{-10}-q^{-8}+2q^{-6}-q^{-4}+q^{-2} \\
L2a1=T(2,-2) & \frac{t^{3}-t}{z}-tz & -q^{-5}-q^{-1} \\
T(2,2) & \frac{-t^{-3}+t^{-1}}{z}+t^{-1}z & -q-q^{5} \\
T(2,4) & \frac{-t^{-5}+t^{-3}}{z}+(-t^{-5}+3t^{-3})z+t^{-3}z^{3} &
-q^{3}-q^{7}+q^{9}-q^{11} \\
L4a1 & \frac{t^{5}-t^{3}}{z}-(t+t^{3})z & -q^{-9}-q^{-5}+q^{-3}-q^{-1} \\
L5a1 & \frac{t-t^{-1}}{z}+(-t^{-1}+2t-t^{3})z+tz^{3} &
q^{-7}-2q^{-5}+q^{-3}-2q^{-1}+q-q^{3} \\
L6a4 & \frac{t^{2}-2+t^{-2}}{z^{2}}+(2-t^{2}-t^{-2})z^{2}+z^{4} &
\begin{array}{c}
-q^{-6}+3q^{-4}-2q^{-2}+4 \\
-2q^{2}+3q^{4}-q^{6}%
\end{array}
\\
L6a5 & \frac{t^{8}-2t^{6}+t^{4}}{z^{2}}+3t^{4}-3t^{6}+(2t^{4}+t^{2})z^{2} &
\begin{array}{c}
q^{-14}-q^{-12}+3q^{-10}-q^{-8} \\
+3q^{-6}-2q^{-4}+q^{-2}%
\end{array}
\\
L6n1 & \frac{t^{-4}-2t^{-2}+1}{z^{2}}+t^{-4}-3t^{-2}+2-t^{-2}z^{2} &
2+q^{4}+q^{8}%
\end{array}%
$

\bigskip

$%
\begin{array}{cc}
& Our\text{ } \text{ }J_{1}(\mathcal{L};q) \\
3_1=T(2,-3) & q^{2}+q^{6}-q^{8} \\
T(2,3) & -q^{-8}+q^{-6}+q^{-2} \\
4_1 & q^{-4}-q^{-2}+1-q^{2}+q^{4} \\
5_2 & q^{2}-q^{4}+2q^{6}-q^{8}+q^{10}-q^{12} \\
L2a1=T(2,-2) & q+q^{5} \\
T(2,2) & q^{-5}+q^{-1} \\
T(2,4) & q^{-11}-q^{-9}+q^{-7}+q^{-3} \\
L4a1 & q-q^{3}+q^{5}+q^{9} \\
L5a1 & q^{-3}-q^{-1}+2q-q^{3}+2q^{5}-q^{7} \\
L6a4 & -q^{-6}+3q^{-4}-2q^{-2}+4-2q^{2}+3q^{4}-q^{6} \\
L6a5 & q^{2}-2q^{4}+3q^{6}-q^{8}+3q^{10}-q^{12}+q^{14} \\
L6n1 & q^{-8}+q^{-4}+2%
\end{array}%
$

\bigskip

$%
\begin{array}{cc}
& Our\text{ }J_{2}(\mathcal{L};q) \\
3_1=T(2,-3) & q^{4}+q^{10}-q^{14}+q^{16}-q^{18}-q^{20}+q^{22} \\
T(2,3) & q^{-22}-q^{-20}-q^{-18}+q^{-16}-q^{-14}+q^{-10}+q^{-4} \\
4_1 &
\begin{array}{c}
q^{-12}-q^{-10}-q^{-8}+2q^{-6}-q^{-4}-q^{-2}+3 \\
-q^{2}-q^{4}+2q^{6}-q^{8}-q^{10}+q^{12}%
\end{array}
\\
5_2 &
\begin{array}{c}
q^{4}-q^{6}+3q^{10}-2q^{12}-q^{14}+4q^{16}-3q^{18}-q^{20}+3q^{22} \\
-2q^{24}-q^{26}+2q^{28}-q^{30}-q^{32}+q^{34}%
\end{array}
\\
L2a1=T(2,-2) & q^{2}+q^{8}+q^{14} \\
T(2,2) & q^{-14}+q^{-8}+q^{-2} \\
T(2,4) & q^{-30}-q^{-28}+2q^{-24}-q^{-22}+q^{-18}-q^{-16}+q^{-12}+q^{-6} \\
L4a1 & q^{2}-q^{4}+2q^{8}-q^{10}+q^{14}-q^{16}+q^{20}+q^{26} \\
L5a1 &
\begin{array}{c}
q^{-10}-q^{-8}-q^{-6}+3q^{-4}-q^{-2}-2+5q^{2}-q^{4}-3q^{6} \\
+5q^{8}-q^{10}-3q^{12}+4q^{14}-q^{16}-2q^{18}+q^{20}%
\end{array}
\\
L6a4 &
\begin{array}{c}
q^{-18}-3q^{-16}-q^{-14}+8q^{-12}-6q^{-10}-4q^{-8}+14q^{-6}-7q^{-4}-6q^{-2}
\\
+17-6q^{2}-7q^{4}+14q^{6}-4q^{8}-6q^{10}+8q^{12}-q^{14}-3q^{16}+q^{18}%
\end{array}
\\
L6a5 &
\begin{array}{c}
q^{4}-2q^{6}+q^{8}+5q^{10}-6q^{12}+7q^{16}-7q^{18}+q^{20}+8q^{22} \\
-5q^{24}+7q^{28}-3q^{30}-q^{32}+4q^{34}-q^{36}-q^{38}+q^{40}%
\end{array}
\\
L6n1 & q^{-24}+q^{-18}+q^{-12}+q^{-10}+q^{-6}+q^{-4}+2+q^{2}%
\end{array}%
$

where $z=q-q^{-1}$ and $[N]=q^{N}-q^{-N}$.

\subsection{Examples for congruent skein relations: colored Jones case}

In this subsection, we provide a lot of examples for the congruent skein
relation for colored Jones.

\textbf{Conjecture 7.7 (Main part)}
\begin{align*}
J_{N}(\mathcal{L}_{+})-J_{N}(\mathcal{L}_{-})& \equiv 0\mod [N] \\
J_{N}(\mathcal{L}_{+})-J_{N}(\mathcal{L}_{-})& \equiv 0\mod [N+2]
\end{align*}

\textbf{A) Links with two components}

We write the conjecture in the following way

$J_{N}(\mathcal{L}_{+})-J_{N}(\mathcal{L}_{-})\equiv 0$ $mod\frac{[N][N+2]}{%
[1]}$ if $N$ is odd

and

$J_{N}(\mathcal{L}_{+})-J_{N}(\mathcal{L}_{-})\equiv 0$ $mod\frac{[N][N+2]}{%
[2]}$ if $N$ is even

Sample Examples

1) $\mathcal{L}_{+}=T(2,4)$ and $\mathcal{L}_{-}=T(2,2)$

a) $N=1$

$J_{1}(\mathcal{L}_{+})=q^{-11}-q^{-9}+q^{-7}+q^{-3}$

$J_{1}(\mathcal{L}_{-})=q^{-5}+q^{-1}$

So we have

$J_{1}(\mathcal{L}_{+})-J_{1}(\mathcal{L}_{-})=\frac{(q^{3}-q^{-3})(q-q^{-1})%
}{q-q^{-1}}(-q^{-8}+q^{-6}-q^{-4})$

\bigskip

b) $N=2$

$J_{2}(\mathcal{L}%
_{+})=q^{-30}-q^{-28}+2q^{-24}-q^{-22}+q^{-18}-q^{-16}+q^{-12}+q^{-6}$

$J_{2}(\mathcal{L}_{-})=q^{-14}+q^{-8}+q^{-2}$

So we have

$J_{2}(\mathcal{L}_{+})-J_{2}(\mathcal{L}_{-})$

$=\frac{(q^{4}-q^{-4})(q^{2}-q^{-2})}{q^{2}-q^{-2}}%
(-q^{-26}+q^{-24}-2q^{-20}+q^{-16}-q^{-14}-q^{-12}+q^{-10}-q^{-6})$

\bigskip

b) $N=3$

$J_{3}(\mathcal{L}%
_{+})=q^{-57}-q^{-55}+q^{-51}+q^{-49}-q^{-47}-q^{-45}+q^{-43}-q^{-41}-q^{-39}+q^{-35}+q^{-33}-q^{-31}+q^{-25}-q^{-23}+q^{-17}+q^{-9}
$

$J_{3}(\mathcal{L}_{-})=q^{-27}+q^{-19}+q^{-11}+q^{-3}$

So we have

$J_{3}(\mathcal{L}_{+})-J_{3}(\mathcal{L}_{-})$

$=\frac{(q^{5}-q^{-5})(q^{3}-q^{-3})}{q-q^{-1}}%
(-q^{-50}+2q^{-48}-q^{-46}-2q^{-44}+2q^{-42}-q^{-36}-q^{-34}+2q^{-32}-q^{-30}-q^{-26}+q^{-22}-q^{-18}+q^{-12}-q^{-10})
$

\bigskip

2) $\mathcal{L}_{+}=T(2,-2)$ and $\mathcal{L}_{-}=L4a1$

a) $N=1$

$J_{1}(\mathcal{L}_{+})=q^{2}+q^{8}+q^{14}$

$J_{1}(\mathcal{L}%
_{-})=q^{2}-q^{4}+2q^{8}-q^{10}+q^{14}-q^{16}+q^{20}+q^{26} $

So we have

$J_{1}(\mathcal{L}_{+})-J_{1}(\mathcal{L}_{-})=\frac{%
(q^{3}-q^{-3})(q^{1}-q^{-1})}{q-q^{-1}}(-q^{6})$

\bigskip

b) $N=2$

$J_{2}(\mathcal{L}_{+})=q^{2}+q^{8}+q^{14}$

$J_{2}(\mathcal{L}%
_{-})=q^{2}-q^{4}+2q^{8}-q^{10}+q^{14}-q^{16}+q^{20}+q^{26} $

So we have

$J_{2}(\mathcal{L}_{+})-J_{2}(\mathcal{L}_{-})$

$=\frac{(q^{4}-q^{-4})(q^{2}-q^{-2})}{q^{2}-q^{-2}}%
(-q^{8}+q^{12}-q^{14}-q^{16}-q^{22})$

\bigskip

c) $N=3$

$J_{3}(\mathcal{L}_{+})=q^{3}-q^{11}+q^{19}+q^{27}$

$J_{3}(\mathcal{L}%
_{-})=q^{3}-q^{5}+q^{9}+q^{11}-q^{13}-q^{15}+q^{17}+q^{19}-q^{21}+q^{25}+q^{27}-q^{29}+q^{35}-q^{37}+q^{43}+q^{51}
$

So we have

$J_{3}(\mathcal{L}_{+})-J_{3}(\mathcal{L}_{-})$

$=\frac{(q^{5}-q^{-5})(q^{3}-q^{-3})}{q-q^{-1}}%
(-q^{12}+q^{14}+q^{16}-2q^{18}+q^{24}-2q^{28}+q^{30}-q^{38}+q^{42}-q^{44})$

\bigskip

\textbf{B) Links with more than two components}

We still write the conjecture in the following way

$J_{N}(\mathcal{L}_{+})-J_{N}(\mathcal{L}_{-})\equiv 0$ $mod\frac{[N][N+2]}{%
[1]}$ if $N$ is odd

and

$J_{N}(\mathcal{L}_{+})-J_{N}(\mathcal{L}_{-})\equiv 0$ $mod\frac{[N][N+2]}{%
[2]}$ if $N$ is even

Sample Examples

1) $\mathcal{L}_{+}=L6a4$ and $\mathcal{L}_{-}=T(2,-2)\otimes U$

where $\otimes$ denote the disjoint union.

a) $N=1$

$J_{1}(\mathcal{L}_{+})=-q^{-6}+3q^{-4}-2q^{-2}+4-2q^{2}+3q^{4}-q^{6}$

$J_{1}(\mathcal{L}_{-})=1+q^{2}+q^{4}+q^{6}$

So we have

$J_{1}(\mathcal{L}_{+})-J_{1}(\mathcal{L}_{-})=\frac{(q^{3}-q^{-3})(q-q^{-1})%
}{q-q^{-1}}(q^{-3}-3q^{-1}+2q-2q^{3})$

\bigskip

b) $N=2$

$J_{2}(\mathcal{L}%
_{+})=q^{-18}-3q^{-16}-q^{-14}+8q^{-12}-6q^{-10}-4q^{-8}+14q^{-6}-7q^{-4}-6q^{-2}+17-6q^{2}-7q^{4}+14q^{6}-4q^{8}-6q^{10}+8q^{12}-q^{14}-3q^{16}+q^{18}
$

$J_{2}(\mathcal{L}%
_{-})=1+q^{2}+q^{4}+q^{6}+q^{8}+q^{10}+q^{12}+q^{14}+q^{16} $

So we have

$J_{2}(\mathcal{L}_{+})-J_{2}(\mathcal{L}_{-})$

$=\frac{(q^{4}-q^{-4})(q^{2}-q^{-2})}{q^{2}-q^{-2}}%
(-q^{-14}+3q^{-12}+q^{-10}-8q^{-8}+5q^{-6}+7q^{-4}-13q^{-2}-1+11q^{2}-9q^{4}-6q^{6}+7q^{8}-2q^{10}-4q^{12}+q^{14})
$

\bigskip

c) $N=3$

$J_{3}(\mathcal{L}%
_{+})=-q^{-36}+3q^{-34}+q^{-32}-5q^{-30}-6q^{-28}+6q^{-26}+14q^{-24}-10q^{-22}-17q^{-20}+6q^{-18}+28q^{-16}-8q^{-14}-28q^{-12}+2q^{-10}+41q^{-8}-5q^{-6}-34q^{-4}-q^{-2}+44-q^{2}-34q^{4}-5q^{6}+41q^{8}+2q^{10}-28q^{12}-8q^{14}+28q^{16}+6q^{18}-17q^{20}-10q^{22}+14q^{24}+6q^{26}-6q^{28}-5q^{30}+q^{32}+3q^{34}-q^{36}
$

$J_{3}(\mathcal{L}%
_{-})=1+q^{2}+q^{4}+q^{6}+q^{8}+q^{10}+q^{12}+q^{14}+q^{16}+q^{18}+q^{20}+q^{22}+q^{24}+q^{26}+q^{28}+q^{30}
$

So we have

$J_{3}(\mathcal{L}_{+})-J_{3}(\mathcal{L}_{-})$

$=\frac{(q^{5}-q^{-5})(q^{3}-q^{-3})}{q^{3}-q^{-3}}%
(q^{-29}-4q^{-27}+2q^{-25}+7q^{-23}-3q^{-21}-9q^{-19}-5q^{-17}+23q^{-15}+4q^{-13}-27q^{-11}-10q^{-9}+28q^{-7}+19q^{-5}-27q^{-3}-33q^{-1}+32q+26q^{3}-20q^{5}-29q^{7}+10q^{9}+26q^{11}-5q^{13}-23q^{15}+4q^{17}+9q^{19}+3q^{21}-8q^{23}-2q^{25}+4q^{27}-q^{29})
$

\bigskip

\textbf{Conjecture 7.7 (Moreover part)}.

\textbf{A) Links with two components}

Set

$A_{n}=\{q|q^{n}=\pm 1\}$

$B_{n}=\{q|q^{n}=1\}$

$C_{n}=\{q|q^{n}=-1\}$

We write the conjecture in the following way

The root set of $J_{N}(\mathcal{L}_{+})-J_{N}(\mathcal{L}_{-})\equiv J_{k}(%
\mathcal{L}_{+})-J_{k}(\mathcal{L}_{-})$ contains $(B_{N-k}\cup
C_{N+k+2})-(A_{k+1}-A_{1})$

and

The root set of $J_{N}(\mathcal{L}_{-})-J_{N}(\mathcal{L}_{+})\equiv J_{k}(%
\mathcal{L}_{+})-J_{k}(\mathcal{L}_{-})$ contains $(C_{N-k}\cup
B_{N+k+2})-(A_{k+1}-A_{1})$

Remark: this only valid for $k\geq 1$, see our old conjecture for $k=0$.
(for k=1, this is also slightly different from our old one)

Sample Examples

1) $\mathcal{L}_{+}=T(2,4)$, $\mathcal{L}_{-}=T(2,2)$

I) $k=2$

We need to test the following

The root set of $J_{N}(\mathcal{L}_{+})-J_{N}(\mathcal{L}_{-})\equiv J_{2}(%
\mathcal{L}_{+})-J_{2}(\mathcal{L}_{-})$ contains $(B_{N-2}\cup
C_{N+4})-(A_{3}-A_{1})$

and

The root set of $J_{N}(\mathcal{L}_{-})-J_{N}(\mathcal{L}_{+})\equiv J_{2}(%
\mathcal{L}_{+})-J_{2}(\mathcal{L}_{-})$ contains $(C_{N-2}\cup
B_{N+4})-(A_{3}-A_{1})$

a) $N=3$

$J_{3}(\mathcal{L}_{+})-J_{3}(\mathcal{L}_{-})-(J_{2}(\mathcal{L}_{+})-J_{2}(%
\mathcal{L}_{-}))$

$=\frac{(-1+q)(1+q)(1-q+q^{2}-q^{3}+q^{4}-q^{5}+q^{6})}{q^{57}}%
(-1-q-q^{6}-q^{8}-2q^{9}-q^{10}-q^{11}+q^{13}-q^{14}-q^{17}-q^{19}-2q^{20}-2q^{22}-2q^{23}-2q^{24}-3q^{25}-q^{26}+q^{27}-q^{28}+q^{30}+q^{31}+q^{32}+q^{33}+q^{35}-q^{37}+q^{39}-q^{42}-q^{43}+q^{47})
$

contains roots $\{e^{\frac{0\pi \sqrt{-1}}{1}},e^{\frac{1\pi \sqrt{-1}}{7}%
},e^{\frac{3\pi \sqrt{-1}}{7}},e^{\frac{5\pi \sqrt{-1}}{7}},e^{\frac{7\pi
\sqrt{-1}}{7}},e^{\frac{9\pi \sqrt{-1}}{7}},e^{\frac{11\pi \sqrt{-1}}{7}},e^{%
\frac{13\pi \sqrt{-1}}{7}}\}$

and

$J_{3}(\mathcal{L}_{-})-J_{3}(\mathcal{L}_{+})-(J_{2}(\mathcal{L}_{+})-J_{2}(%
\mathcal{L}_{-}))$

$=\frac{(-1+q)(1+q)(1+q+q^{2}+q^{3}+q^{4}+q^{5}+q^{6})}{q^{57}}%
(1-q+q^{6}+q^{8}-2q^{9}+q^{10}-q^{11}+q^{13}+q^{14}-q^{17}-q^{19}+2q^{20}+2q^{22}-2q^{23}+2q^{24}-3q^{25}+q^{26}+q^{27}+q^{28}-q^{30}+q^{31}-q^{32}+q^{33}+q^{35}-q^{37}+q^{39}+q^{42}-q^{43}+q^{47})
$

contains roots $\{e^{\frac{1\pi \sqrt{-1}}{1}},e^{\frac{0\pi \sqrt{-1}}{7}%
},e^{\frac{2\pi \sqrt{-1}}{7}},e^{\frac{4\pi \sqrt{-1}}{7}},e^{\frac{6\pi
\sqrt{-1}}{7}},e^{\frac{8\pi \sqrt{-1}}{7}},e^{\frac{10\pi \sqrt{-1}}{7}},e^{%
\frac{12\pi \sqrt{-1}}{7}}\}$

\bigskip

b) $N=4$

$J_{4}(\mathcal{L}_{+})-J_{4}(\mathcal{L}_{-})-(J_{2}(\mathcal{L}_{+})-J_{2}(%
\mathcal{L}_{-}))$

$=\frac{(-1+q)(1+q)(1+q^{2})(1+q^{4})(1+q^{8})}{q^{92}}%
(-1+q^{2}-q^{6}-q^{10}+2q^{12}-2q^{16}+q^{18}-2q^{20}+q^{22}-2q^{26}+2q^{28}-q^{30}+q^{34}-3q^{36}+q^{38}-q^{40}-q^{42}+2q^{44}-2q^{46}+q^{48}-2q^{52}+q^{54}-q^{56}+q^{60}+q^{68}-q^{70}-q^{72}+q^{74})
$

Contains roots $\{e^{\frac{0\pi \sqrt{-1}}{2}},e^{\frac{2\pi \sqrt{-1}}{2}%
},e^{\frac{\pi \sqrt{-1}}{8}},e^{\frac{3\pi \sqrt{-1}}{8}},e^{\frac{5\pi
\sqrt{-1}}{8}},e^{\frac{7\pi \sqrt{-1}}{8}},e^{\frac{9\pi \sqrt{-1}}{8}},e^{%
\frac{11\pi \sqrt{-1}}{8}},e^{\frac{13\pi \sqrt{-1}}{8}},e^{\frac{15\pi
\sqrt{-1}}{8}}\}$

and

$J_{4}(\mathcal{L}_{-})-J_{4}(\mathcal{L}_{+})-(J_{2}(\mathcal{L}_{+})-J_{2}(%
\mathcal{L}_{-}))$

$=\frac{(-1+q)(1+q)(1+q^{2})(1+q^{4})}{q^{92}}%
(1-q^{2}+q^{6}+q^{8}-2q^{12}+q^{14}+2q^{16}-q^{22}+2q^{24}+q^{26}+q^{34}+q^{36}+q^{40}+q^{44}+q^{46}+q^{50}+q^{54}+q^{60}+q^{62}-q^{64}+2q^{68}+q^{70}-q^{72}+q^{74}+q^{76}-q^{78}+q^{80}+q^{82})
$

Contains roots $\{e^{\frac{1\pi \sqrt{-1}}{2}},e^{\frac{3\pi \sqrt{-1}}{2}%
},e^{\frac{0\pi \sqrt{-1}}{8}},e^{\frac{2\pi \sqrt{-1}}{8}},e^{\frac{6\pi
\sqrt{-1}}{8}},e^{\frac{8\pi \sqrt{-1}}{8}},e^{\frac{10\pi \sqrt{-1}}{8}},e^{%
\frac{14\pi \sqrt{-1}}{8}}\}$

\bigskip

II) $k=3$

We need to test the following

The root set of $J_{N}(\mathcal{L}_{+})-J_{N}(\mathcal{L}_{-})\equiv J_{3}(%
\mathcal{L}_{+})-J_{3}(\mathcal{L}_{-})$ contains $(B_{N-3}\cup
C_{N+5})-(A_{4}-A_{1})$

and

The root set of $J_{N}(\mathcal{L}_{-})-J_{N}(\mathcal{L}_{+})\equiv J_{3}(%
\mathcal{L}_{+})-J_{3}(\mathcal{L}_{-})$ contains $(C_{N-3}\cup
B_{N+5})-(A_{4}-A_{1})$

a) $N=4$

$J_{4}(\mathcal{L}_{+})-J_{4}(\mathcal{L}_{-})-(J_{3}(\mathcal{L}_{+})-J_{3}(%
\mathcal{L}_{-}))$

$=\frac{(-1+q)(1+q)(1-q+q^{2})(1+q+q^{2})(1-q^{2}+q^{4})(1-q^{3}+q^{6})}{%
q^{92}}%
(-1-q^{3}+q^{4}-q^{6}+q^{7}-q^{8}-q^{10}-q^{11}+q^{12}-2q^{13}+2q^{14}+q^{15}-q^{16}+3q^{17}-2q^{18}+q^{19}-2q^{20}-3q^{21}+q^{22}-5q^{23}+q^{24}-2q^{26}+4q^{27}-2q^{28}+3q^{29}-2q^{31}+3q^{32}-4q^{33}+q^{34}+q^{35}-3q^{36}+3q^{37}-2q^{38}-2q^{41}-q^{43}-q^{44}-q^{46}+q^{48}+q^{50}+2q^{51}-q^{52}+2q^{53}-q^{54}-2q^{57}
$

$%
+2q^{58}-q^{59}+2q^{61}-q^{62}+2q^{63}-q^{64}-2q^{67}+q^{68}-q^{69}+q^{71}-q^{72}+q^{73})
$

contains roots $\{e^{\frac{0\pi \sqrt{-1}}{1}},e^{\frac{1\pi \sqrt{-1}}{9}%
},e^{\frac{3\pi \sqrt{-1}}{9}},e^{\frac{5\pi \sqrt{-1}}{9}},e^{\frac{7\pi
\sqrt{-1}}{9}},e^{\frac{9\pi \sqrt{-1}}{9}},e^{\frac{11\pi \sqrt{-1}}{9}},e^{%
\frac{13\pi \sqrt{-1}}{9}},e^{\frac{15\pi \sqrt{-1}}{9}},e^{\frac{17\pi
\sqrt{-1}}{9}}\}$

and

$J_{4}(\mathcal{L}_{-})-J_{4}(\mathcal{L}_{+})-(J_{3}(\mathcal{L}_{+})-J_{3}(%
\mathcal{L}_{-}))$

$=\frac{(-1+q)(1+q)(1-q+q^{2})(1+q+q^{2})(1-q^{2}+q^{4})(1+q^{3}+q^{6})}{%
q^{92}}%
(1-q^{3}-q^{4}+q^{6}+q^{7}+q^{8}+q^{10}-q^{11}-q^{12}-2q^{13}-2q^{14}+q^{15}+q^{16}+3q^{17}+2q^{18}+q^{19}+2q^{20}-3q^{21}-q^{22}-5q^{23}-q^{24}+2q^{26}+4q^{27}+2q^{28}+3q^{29}-2q^{31}-3q^{32}-4q^{33}-q^{34}+q^{35}+3q^{36}+3q^{37}+2q^{38}-2q^{41}-q^{43}+q^{44}+q^{46}-q^{48}-q^{50}+2q^{51}+q^{52}+2q^{53}+q^{54}-2q^{57}
$

$%
-2q^{58}-q^{59}+2q^{61}+q^{62}+2q^{63}+q^{64}-2q^{67}-q^{68}-q^{69}+q^{71}+q^{72}+q^{73})
$

contains roots $\{e^{\frac{1\pi \sqrt{-1}}{1}},e^{\frac{0\pi \sqrt{-1}}{9}%
},e^{\frac{2\pi \sqrt{-1}}{9}},e^{\frac{4\pi \sqrt{-1}}{9}},e^{\frac{6\pi
\sqrt{-1}}{9}},e^{\frac{8\pi \sqrt{-1}}{9}},e^{\frac{10\pi \sqrt{-1}}{9}},e^{%
\frac{12\pi \sqrt{-1}}{9}},e^{\frac{14\pi \sqrt{-1}}{9}},e^{\frac{16\pi
\sqrt{-1}}{9}}\}$

\bigskip

b) $N=5$

$J_{5}(\mathcal{L}_{+})-J_{5}(\mathcal{L}_{-})-(J_{3}(\mathcal{L}_{+})-J_{3}(%
\mathcal{L}_{-}))$

$=\frac{%
(-1+q)(1+q)(1-q+q^{2}-q^{3}+q^{4})(1+q+q^{2}+q^{3}+q^{4})(1-q^{2}+q^{4}-q^{6}+q^{8})%
}{q^{135}}%
(-1+q^{4}-q^{6}-q^{8}+q^{14}+q^{16}-q^{18}-3q^{20}-q^{22}+q^{24}-q^{32}+q^{36}-2q^{40}-2q^{42}-q^{44}-q^{46}+q^{50}-q^{54}-2q^{60}-q^{62}-2q^{66}-q^{68}+2q^{70}-q^{74}+q^{76}-3q^{80}-q^{82}+q^{84}+2q^{94}+q^{96}-q^{98}-q^{100}+q^{102}+q^{104}-q^{108}-q^{110}+q^{114})
$

Contains roots $\{e^{\frac{0\pi \sqrt{-1}}{2}},e^{\frac{2\pi \sqrt{-1}}{2}%
},e^{\frac{\pi \sqrt{-1}}{10}},e^{\frac{3\pi \sqrt{-1}}{10}},e^{\frac{7\pi
\sqrt{-1}}{10}},e^{\frac{9\pi \sqrt{-1}}{10}},e^{\frac{11\pi \sqrt{-1}}{10}%
},e^{\frac{13\pi \sqrt{-1}}{10}},e^{\frac{17\pi \sqrt{-1}}{10}},e^{\frac{%
19\pi \sqrt{-1}}{10}}\}$

and

$J_{5}(\mathcal{L}_{-})-J_{5}(\mathcal{L}_{+})-(J_{3}(\mathcal{L}_{+})-J_{3}(%
\mathcal{L}_{-}))$

$=\frac{(-1+q)(1+q)(1-q+q^{2}-q^{3}+q^{4})(1+q+q^{2}+q^{3}+q^{4})}{q^{135}}%
(1-q^{2}+q^{6}+q^{10}-q^{12}-q^{14}+q^{16}+q^{18}+2q^{20}-q^{22}-q^{24}+q^{28}+2q^{30}-q^{34}+2q^{40}+q^{42}+q^{50}+q^{52}+q^{54}+q^{60}+q^{64}+q^{66}+q^{72}+q^{76}+2q^{78}-q^{80}+2q^{84}+q^{86}+q^{88}-2q^{90}+q^{92}+2q^{94}+q^{96}-q^{100}+2q^{102}+q^{104}+q^{112}+q^{114}-q^{116}+q^{120}+q^{122})
$

Contains roots $\{e^{\frac{0\pi \sqrt{-1}}{10}},e^{\frac{2\pi \sqrt{-1}}{10}%
},e^{\frac{4\pi \sqrt{-1}}{10}},e^{\frac{6\pi \sqrt{-1}}{10}},e^{\frac{8\pi
\sqrt{-1}}{10}},e^{\frac{10\pi \sqrt{-1}}{10}},e^{\frac{12\pi \sqrt{-1}}{10}%
},e^{\frac{14\pi \sqrt{-1}}{10}},e^{\frac{16\pi \sqrt{-1}}{10}},e^{\frac{%
18\pi \sqrt{-1}}{10}}\}$

\bigskip

III) $k=4$

We need to test the following

The root set of $J_{N}(\mathcal{L}_{+})-J_{N}(\mathcal{L}_{-})\equiv J_{4}(%
\mathcal{L}_{+})-J_{4}(\mathcal{L}_{-})$ contains $(B_{N-4}\cup
C_{N+6})-(A_{5}-A_{1})$

and

The root set of $J_{N}(\mathcal{L}_{-})-J_{N}(\mathcal{L}_{+})\equiv J_{4}(%
\mathcal{L}_{+})-J_{4}(\mathcal{L}_{-})$ contains $(C_{N-4}\cup
B_{N+6})-(A_{5}-A_{1})$

a) $N=5$

$J_{5}(\mathcal{L}_{+})-J_{5}(\mathcal{L}_{-})-(J_{4}(\mathcal{L}_{+})-J_{4}(%
\mathcal{L}_{-}))$

$=\frac{%
(-1+q)(1+q)(1-q+q^{2}-q^{3}+q^{4}-q^{5}+q^{6}-q^{7}+q^{8}-q^{9}+q^{10})}{%
q^{135}}%
(-1-q-q^{6}-q^{7}-q^{8}-q^{9}-q^{10}-q^{13}+q^{17}-q^{20}-q^{21}-2q^{22}-2q^{23}-q^{24}-2q^{25}-q^{26}-q^{27}-2q^{28}-q^{29}-q^{30}-q^{32}-q^{35}-q^{37}+q^{39}-2q^{42}-2q^{44}-3q^{45}-2q^{46}-3q^{47}-2q^{48}-2q^{49}-2q^{50}-q^{51}-q^{52}+2q^{53}-q^{54}-q^{55}-q^{57}+2q^{63}-q^{64}-q^{66}-2q^{67}-q^{68}-q^{69}-q^{72}-q^{74}
$

$%
-q^{76}-q^{81}+2q^{83}+q^{87}+q^{89}+q^{90}+q^{92}+q^{93}+q^{95}+q^{99}-q^{101}+q^{107}-q^{110}-q^{111}+q^{119}))
$

Contains roots $\{e^{\frac{0\pi \sqrt{-1}}{1}},e^{\frac{\pi \sqrt{-1}}{11}%
},e^{\frac{3\pi \sqrt{-1}}{11}},e^{\frac{5\pi \sqrt{-1}}{11}},e^{\frac{7\pi
\sqrt{-1}}{11}},e^{\frac{9\pi \sqrt{-1}}{11}},e^{\frac{11\pi \sqrt{-1}}{11}%
},e^{\frac{13\pi \sqrt{-1}}{11}},e^{\frac{15\pi \sqrt{-1}}{11}},$

$e^{\frac{17\pi \sqrt{-1}}{11}},e^{\frac{19\pi \sqrt{-1}}{11}},e^{\frac{%
21\pi \sqrt{-1}}{11}}\}$

and

$J_{5}(\mathcal{L}_{-})-J_{5}(\mathcal{L}_{+})-(J_{4}(\mathcal{L}_{+})-J_{4}(%
\mathcal{L}_{-}))$

$=\frac{%
(-1+q)(1+q)(1+q+q^{2}+q^{3}+q^{4}+q^{5}+q^{6}+q^{7}+q^{8}+q^{9}+q^{10})}{%
q^{135}}%
(1-q+q^{6}-q^{7}+q^{8}-q^{9}+q^{10}-q^{13}+q^{17}+q^{20}-q^{21}+2q^{22}-2q^{23}+q^{24}-2q^{25}+q^{26}-q^{27}+2q^{28}-q^{29}+q^{30}+q^{32}-q^{35}-q^{37}+q^{39}+2q^{42}+2q^{44}-3q^{45}+2q^{46}-3q^{47}+2q^{48}-2q^{49}+2q^{50}-q^{51}+q^{52}+2q^{53}+q^{54}-q^{55}-q^{57}+2q^{63}+q^{64}+q^{66}-2q^{67}+q^{68}-q^{69}+q^{72}+q^{74}
$

$%
+q^{76}-q^{81}+2q^{83}+q^{87}+q^{89}-q^{90}-q^{92}+q^{93}+q^{95}+q^{99}-q^{101}+q^{107}+q^{110}-q^{111}+q^{119}))
$

Contains roots $\{e^{\frac{1\pi \sqrt{-1}}{1}},e^{\frac{0\pi \sqrt{-1}}{11}%
},e^{\frac{2\pi \sqrt{-1}}{11}},e^{\frac{4\pi \sqrt{-1}}{11}},e^{\frac{6\pi
\sqrt{-1}}{11}},e^{\frac{8\pi \sqrt{-1}}{11}},e^{\frac{10\pi \sqrt{-1}}{11}%
},e^{\frac{12\pi \sqrt{-1}}{11}},$

$e^{\frac{14\pi \sqrt{-1}}{11}},e^{\frac{16\pi \sqrt{-1}}{11}},e^{\frac{%
18\pi \sqrt{-1}}{11}},e^{\frac{20\pi \sqrt{-1}}{11}}\}$

\bigskip

2) $\mathcal{L}_{+}=T(2,-2)$ and $\mathcal{L}_{-}=L4a1$

I) $k=2$

We need to test the following

The root set of $J_{N}(\mathcal{L}_{+})-J_{N}(\mathcal{L}_{-})\equiv J_{2}(%
\mathcal{L}_{+})-J_{2}(\mathcal{L}_{-})$ contains $(B_{N-2}\cup
C_{N+4})-(A_{3}-A_{1})$

and

The root set of $J_{N}(\mathcal{L}_{-})-J_{N}(\mathcal{L}_{+})\equiv J_{2}(%
\mathcal{L}_{+})-J_{2}(\mathcal{L}_{-})$ contains $(C_{N-2}\cup
B_{N+4})-(A_{3}-A_{1})$

a) $N=3$

$J_{3}(\mathcal{L}_{+})-J_{3}(\mathcal{L}_{-})-(J_{2}(\mathcal{L}_{+})-J_{2}(%
\mathcal{L}_{-}))$

$%
=-q^{4}(-1+q)(1+q)(1-q+q^{2}-q^{3}+q^{4}-q^{5}+q^{6})(-1+q^{4}-q^{6}-q^{8}-q^{14}+q^{17}+q^{18}+q^{19}+q^{20}+q^{21}+q^{22}+q^{23}+q^{25}+q^{26}+q^{27}+q^{28}+q^{29}+q^{30}+q^{31}+q^{33}+q^{34}+q^{35}+q^{36}+q^{37}+q^{38}+q^{39})
$

contains roots $\{e^{\frac{0\pi \sqrt{-1}}{1}},e^{\frac{\pi \sqrt{-1}}{7}%
},e^{\frac{3\pi \sqrt{-1}}{7}},e^{\frac{5\pi \sqrt{-1}}{7}},e^{\frac{7\pi
\sqrt{-1}}{7}},e^{\frac{9\pi \sqrt{-1}}{7}},e^{\frac{11\pi \sqrt{-1}}{7}},e^{%
\frac{13\pi \sqrt{-1}}{7}}\}$

and

$J_{3}(\mathcal{L}_{-})-J_{3}(\mathcal{L}_{+})-(J_{2}(\mathcal{L}_{+})-J_{2}(%
\mathcal{L}_{-}))$

$%
=q^{4}(-1+q)(1+q)(1+q+q^{2}+q^{3}+q^{4}+q^{5}+q^{6})(1-q^{4}+q^{6}+q^{8}+q^{14}+q^{17}-q^{18}+q^{19}-q^{20}+q^{21}-q^{22}+q^{23}+q^{25}-q^{26}+q^{27}-q^{28}+q^{29}-q^{30}+q^{31}+q^{33}-q^{34}+q^{35}-q^{36}+q^{37}-q^{38}+q^{39})
$

contains roots $\{e^{\frac{1\pi \sqrt{-1}}{1}},e^{\frac{0\pi \sqrt{-1}}{7}%
},e^{\frac{2\pi \sqrt{-1}}{7}},e^{\frac{4\pi \sqrt{-1}}{7}},e^{\frac{6\pi
\sqrt{-1}}{7}},e^{\frac{8\pi \sqrt{-1}}{7}},e^{\frac{10\pi \sqrt{-1}}{7}},e^{%
\frac{12\pi \sqrt{-1}}{7}}\}$

\bigskip

b) $N=4$

$J_{4}(\mathcal{L}_{+})-J_{4}(\mathcal{L}_{-})-(J_{2}(\mathcal{L}_{+})-J_{2}(%
\mathcal{L}_{-}))$

$%
=-q^{4}(-1+q)(1+q)(1+q^{2})(1+q^{4})(1+q^{8})(-1+q^{2}+q^{4}-2q^{6}+q^{12}-q^{16}+q^{18}+q^{22}-q^{26}+q^{28}+q^{32}+q^{34}-q^{36}+q^{38}+q^{44}-q^{46}+q^{48}+q^{54}+q^{64})
$

Contains roots $\{e^{\frac{0\pi \sqrt{-1}}{2}},e^{\frac{2\pi \sqrt{-1}}{2}%
},e^{\frac{\pi \sqrt{-1}}{8}},e^{\frac{3\pi \sqrt{-1}}{8}},e^{\frac{5\pi
\sqrt{-1}}{8}},e^{\frac{7\pi \sqrt{-1}}{8}},e^{\frac{9\pi \sqrt{-1}}{8}},e^{%
\frac{11\pi \sqrt{-1}}{8}},e^{\frac{13\pi \sqrt{-1}}{8}},e^{\frac{15\pi
\sqrt{-1}}{8}}\}$

and

$J_{4}(\mathcal{L}_{-})-J_{4}(\mathcal{L}_{+})-(J_{2}(\mathcal{L}_{+})-J_{2}(%
\mathcal{L}_{-}))$

$%
=q^{4}(-1+q)(1+q)(1+q^{2})(1+q^{4})(1+q^{2}-q^{4}+q^{8}+q^{10}+2q^{12}-q^{16}+q^{18}+q^{20}+q^{22}-q^{24}+q^{28}+q^{30}+q^{32}+q^{38}+q^{40}+q^{42}+q^{48}+q^{52}+q^{56}+q^{62}+q^{64}+q^{72})
$

Contains roots $\{e^{\frac{1\pi \sqrt{-1}}{2}},e^{\frac{3\pi \sqrt{-1}}{2}%
},e^{\frac{0\pi \sqrt{-1}}{8}},e^{\frac{2\pi \sqrt{-1}}{8}},e^{\frac{6\pi
\sqrt{-1}}{8}},e^{\frac{8\pi \sqrt{-1}}{8}},e^{\frac{10\pi \sqrt{-1}}{8}},e^{%
\frac{14\pi \sqrt{-1}}{8}}\}$

\bigskip

II) $k=3$

We need to test the following

The root set of $J_{N}(\mathcal{L}_{+})-J_{N}(\mathcal{L}_{-})\equiv J_{3}(%
\mathcal{L}_{+})-J_{3}(\mathcal{L}_{-})$ contains $(B_{N-3}\cup
C_{N+5})-(A_{4}-A_{1})$

and

The root set of $J_{N}(\mathcal{L}_{-})-J_{N}(\mathcal{L}_{+})\equiv J_{3}(%
\mathcal{L}_{+})-J_{3}(\mathcal{L}_{-})$ contains $(C_{N-3}\cup
B_{N+5})-(A_{4}-A_{1})$

a) $N=4$

$J_{4}(\mathcal{L}_{+})-J_{4}(\mathcal{L}_{-})-(J_{3}(\mathcal{L}_{+})-J_{3}(%
\mathcal{L}_{-}))$

$%
=-q^{5}(-1+q)(1+q)(1-q+q^{2})(1+q+q^{2})(1-q^{3}+q^{6})(-1+q-q^{3}+2q^{4}-q^{5}-q^{6}+2q^{7}-2q^{8}+q^{12}-q^{13}+q^{14}-2q^{16}+2q^{17}-q^{18}-q^{19}+q^{20}-q^{26}+q^{27}-q^{28}+q^{30}+q^{33}+q^{37}+q^{41}+q^{44}+q^{47}+q^{51}+q^{54}+q^{57}+q^{61}+q^{64}+q^{67})
$

contains roots $\{e^{\frac{0\pi \sqrt{-1}}{1}},e^{\frac{1\pi \sqrt{-1}}{9}%
},e^{\frac{3\pi \sqrt{-1}}{9}},e^{\frac{5\pi \sqrt{-1}}{9}},e^{\frac{7\pi
\sqrt{-1}}{9}},e^{\frac{9\pi \sqrt{-1}}{9}},e^{\frac{11\pi \sqrt{-1}}{9}},e^{%
\frac{13\pi \sqrt{-1}}{9}},e^{\frac{15\pi \sqrt{-1}}{9}},e^{\frac{17\pi
\sqrt{-1}}{9}}\}$

and

$J_{4}(\mathcal{L}_{-})-J_{4}(\mathcal{L}_{+})-(J_{3}(\mathcal{L}_{+})-J_{3}(%
\mathcal{L}_{-}))$

$%
=q^{5}(-1+q)(1+q)(1-q+q^{2})(1+q+q^{2})(1+q^{3}+q^{6})(1+q-q^{3}-2q^{4}-q^{5}+q^{6}+2q^{7}+2q^{8}-q^{12}-q^{13}-q^{14}+2q^{16}+2q^{17}+q^{18}-q^{19}-q^{20}+q^{26}+q^{27}+q^{28}-q^{30}+q^{33}+q^{37}+q^{41}-q^{44}+q^{47}+q^{51}-q^{54}+q^{57}+q^{61}-q^{64}+q^{67})
$

contains roots $\{e^{\frac{1\pi \sqrt{-1}}{1}},e^{\frac{0\pi \sqrt{-1}}{9}%
},e^{\frac{2\pi \sqrt{-1}}{9}},e^{\frac{4\pi \sqrt{-1}}{9}},e^{\frac{6\pi
\sqrt{-1}}{9}},e^{\frac{8\pi \sqrt{-1}}{9}},e^{\frac{10\pi \sqrt{-1}}{9}},e^{%
\frac{12\pi \sqrt{-1}}{9}},e^{\frac{14\pi \sqrt{-1}}{9}},e^{\frac{16\pi
\sqrt{-1}}{9}}\}$

\bigskip

b) $N=5$

$J_{5}(\mathcal{L}_{+})-J_{5}(\mathcal{L}_{-})-(J_{3}(\mathcal{L}_{+})-J_{3}(%
\mathcal{L}_{-}))$

$%
=-q^{5}(-1+q)(1+q)(1-q+q^{2}-q^{3}+q^{4})(1+q+q^{2}+q^{3}+q^{4})(1-q^{2}+q^{4}-q^{6}+q^{8})(-1+2q^{4}-2q^{8}-2q^{10}+q^{12}+3q^{14}-2q^{18}-2q^{20}+2q^{24}+q^{26}-q^{28}-q^{30}+q^{34}+q^{36}-q^{42}+q^{46}+q^{48}+q^{50}+q^{52}+q^{58}+q^{60}+q^{62}+q^{64}+q^{70}+q^{76}+q^{82}+q^{88}+q^{90}+q^{100}+q^{102})
$

Contains roots $\{e^{\frac{0\pi \sqrt{-1}}{2}},e^{\frac{2\pi \sqrt{-1}}{2}%
},e^{\frac{\pi \sqrt{-1}}{10}},e^{\frac{3\pi \sqrt{-1}}{10}},e^{\frac{7\pi
\sqrt{-1}}{10}},e^{\frac{9\pi \sqrt{-1}}{10}},e^{\frac{11\pi \sqrt{-1}}{10}%
},e^{\frac{13\pi \sqrt{-1}}{10}},e^{\frac{17\pi \sqrt{-1}}{10}},e^{\frac{%
19\pi \sqrt{-1}}{10}}\}$

and

$J_{5}(\mathcal{L}_{-})-J_{5}(\mathcal{L}_{+})-(J_{3}(\mathcal{L}_{+})-J_{3}(%
\mathcal{L}_{-}))$

$%
=q^{5}(-1+q)(1+q)(1-q+q^{2}-q^{3}+q^{4})(1+q+q^{2}+q^{3}+q^{4})(1+q^{2}-q^{4}-q^{6}+q^{8}+2q^{10}+q^{12}+q^{22}+2q^{24}+q^{26}-q^{30}+q^{32}+2q^{34}+2q^{36}-q^{40}+q^{44}+q^{46}+q^{48}+q^{56}+q^{58}+q^{60}+q^{62}+q^{68}+q^{70}+q^{74}+q^{80}+q^{86}+q^{90}+q^{98}+q^{100}+q^{110})
$

Contains roots $\{e^{\frac{0\pi \sqrt{-1}}{10}},e^{\frac{2\pi \sqrt{-1}}{10}%
},e^{\frac{4\pi \sqrt{-1}}{10}},e^{\frac{6\pi \sqrt{-1}}{10}},e^{\frac{8\pi
\sqrt{-1}}{10}},e^{\frac{10\pi \sqrt{-1}}{10}},e^{\frac{12\pi \sqrt{-1}}{10}%
},e^{\frac{14\pi \sqrt{-1}}{10}},e^{\frac{16\pi \sqrt{-1}}{10}},e^{\frac{%
18\pi \sqrt{-1}}{10}}\}$

\bigskip

III) $k=4$

We need to test the following

The root set of $J_{N}(\mathcal{L}_{+})-J_{N}(\mathcal{L}_{-})\equiv J_{4}(%
\mathcal{L}_{+})-J_{4}(\mathcal{L}_{-})$ contains $(B_{N-4}\cup
C_{N+6})-(A_{5}-A_{1})$

and

The root set of $J_{N}(\mathcal{L}_{-})-J_{N}(\mathcal{L}_{+})\equiv J_{4}(%
\mathcal{L}_{+})-J_{4}(\mathcal{L}_{-})$ contains $(C_{N-4}\cup
B_{N+6})-(A_{5}-A_{1})$

a) $N=5$

$J_{5}(\mathcal{L}_{+})-J_{5}(\mathcal{L}_{-})-(J_{4}(\mathcal{L}_{+})-J_{4}(%
\mathcal{L}_{-}))$

$%
=q^{6}(-1+q)(1+q)(1-q+q^{2}-q^{3}+q^{4}-q^{5}+q^{6}-q^{7}+q^{8}-q^{9}+q^{10})(-1+q^{4}-2q^{10}-q^{12}+q^{14}+q^{16}+q^{18}-2q^{20}-2q^{22}+q^{26}+q^{28}-q^{30}-2q^{32}-q^{34}+q^{38}-q^{42}-q^{44}-q^{46}+q^{48}+q^{49}+q^{50}+q^{51}+q^{52}+q^{53}+q^{55}+q^{57}+q^{58}+q^{59}+q^{61}+q^{62}+q^{63}+q^{64}+q^{65}+q^{67}+q^{68}+q^{69}+q^{70}+q^{71}+q^{73}+q^{74}+q^{75}
$

$%
+q^{76}+q^{77}+q^{78}+q^{79}+q^{80}+q^{81}+q^{82}+q^{83}+q^{85}+q^{86}+q^{87}+q^{88}+q^{89}+q^{90}+q^{91}+q^{92}+q^{93}+q^{94}+q^{95}+q^{97}+q^{98}+q^{99}+q^{100}+q^{101}+q^{102}+q^{103}+q^{104}+q^{105}+q^{106}+q^{107})
$

Contains roots $\{e^{\frac{0\pi \sqrt{-1}}{1}},e^{\frac{\pi \sqrt{-1}}{11}%
},e^{\frac{3\pi \sqrt{-1}}{11}},e^{\frac{5\pi \sqrt{-1}}{11}},e^{\frac{7\pi
\sqrt{-1}}{11}},e^{\frac{9\pi \sqrt{-1}}{11}},e^{\frac{11\pi \sqrt{-1}}{11}%
},e^{\frac{13\pi \sqrt{-1}}{11}},e^{\frac{15\pi \sqrt{-1}}{11}},$

$e^{\frac{17\pi \sqrt{-1}}{11}},e^{\frac{19\pi \sqrt{-1}}{11}},e^{\frac{%
21\pi \sqrt{-1}}{11}}\}$

and

$J_{5}(\mathcal{L}_{-})-J_{5}(\mathcal{L}_{+})-(J_{4}(\mathcal{L}_{+})-J_{4}(%
\mathcal{L}_{-}))$

$%
=q^{6}(-1+q)(1+q)(1+q+q^{2}+q^{3}+q^{4}+q^{5}+q^{6}+q^{7}+q^{8}+q^{9}+q^{10})(1-q^{4}+2q^{10}+q^{12}-q^{14}-q^{16}-q^{18}+2q^{20}+2q^{22}-q^{26}-q^{28}+q^{30}+2q^{32}+q^{34}-q^{38}+q^{42}+q^{44}+q^{46}-q^{48}+q^{49}-q^{50}+q^{51}-q^{52}+q^{53}+q^{55}+q^{57}-q^{58}+q^{59}+q^{61}-q^{62}+q^{63}-q^{64}+q^{65}+q^{67}-q^{68}+q^{69}-q^{70}+q^{71}+q^{73}-q^{74}+q^{75}
$

$%
-q^{76}+q^{77}-q^{78}+q^{79}-q^{80}+q^{81}-q^{82}+q^{83}+q^{85}-q^{86}+q^{87}-q^{88}+q^{89}-q^{90}+q^{91}-q^{92}+q^{93}-q^{94}+q^{95}+q^{97}-q^{98}+q^{99}-q^{100}+q^{101}-q^{102}+q^{103}-q^{104}+q^{105}-q^{106}+q^{107})
$

Contains roots $\{e^{\frac{1\pi \sqrt{-1}}{1}},e^{\frac{0\pi \sqrt{-1}}{11}%
},e^{\frac{2\pi \sqrt{-1}}{11}},e^{\frac{4\pi \sqrt{-1}}{11}},e^{\frac{6\pi
\sqrt{-1}}{11}},e^{\frac{8\pi \sqrt{-1}}{11}},e^{\frac{10\pi \sqrt{-1}}{11}%
},e^{\frac{12\pi \sqrt{-1}}{11}},e^{\frac{14\pi \sqrt{-1}}{11}},$

$e^{\frac{16\pi \sqrt{-1}}{11}},e^{\frac{18\pi \sqrt{-1}}{11}},e^{\frac{%
20\pi \sqrt{-1}}{11}}\}$

\bigskip

\textbf{B) Links with odd components}

Conjecture:

The root of $J_{N}(\mathcal{L}_{+})-J_{N}(\mathcal{L}_{-})\equiv J_{k}(%
\mathcal{L}_{+})-J_{k}(\mathcal{L}_{-})$ contains $(A_{N-k}\cup
A_{N+k+2})-(A_{k+1}-A_{1})$

Sample Examples

1) $\mathcal{L}_{+}=L6a4$ and $\mathcal{L}_{-}=T(2,-2)\cup Unknot$

where $\cup $ denote the disjoint union.

I) $k=2$

We need to test the following

The root set of $J_{N}(\mathcal{L}_{+})-J_{N}(\mathcal{L}_{-})\equiv J_{2}(%
\mathcal{L}_{+})-J_{2}(\mathcal{L}_{-})$ contains $(A_{N-2}\cup
A_{N+4})-(A_{3}-A_{1})$

a) $N=3$

$J_{3}(\mathcal{L}_{+})-J_{3}(\mathcal{L}_{-})-(J_{2}(\mathcal{L}_{+})-J_{2}(%
\mathcal{L}_{-}))$

$=-\frac{%
(-1+q)(1+q)(1-q+q^{2}-q^{3}+q^{4}-q^{5}+q^{6})(1+q+q^{2}+q^{3}+q^{4}+q^{5}+q^{6})%
}{q^{36}}%
(-1+3q^{2}+q^{4}-5q^{6}-6q^{8}+6q^{10}+14q^{12}-11q^{14}-14q^{16}+6q^{18}+26q^{20}-13q^{22}-30q^{24}+22q^{26}+34q^{28}-33q^{30}-21q^{32}+31q^{34}+14q^{36}-25q^{38}-5q^{40}+15q^{42}+12q^{44}-13q^{46}-5q^{48}+7q^{50}+6q^{52}-q^{54}-3q^{56}+q^{58})
$

contains roots $\{e^{\frac{0\pi \sqrt{-1}}{1}},e^{\frac{1\pi \sqrt{-1}}{1}%
},e^{\frac{\pi \sqrt{-1}}{7}},e^{\frac{2\pi \sqrt{-1}}{7}},e^{\frac{3\pi
\sqrt{-1}}{7}},e^{\frac{4\pi \sqrt{-1}}{7}},e^{\frac{5\pi \sqrt{-1}}{7}},e^{%
\frac{6\pi \sqrt{-1}}{7}},e^{\frac{8\pi \sqrt{-1}}{7}},e^{\frac{9\pi \sqrt{-1%
}}{7}},$

$e^{\frac{10\pi \sqrt{-1}}{7}},e^{\frac{11\pi \sqrt{-1}}{7}},e^{\frac{12\pi
\sqrt{-1}}{7}},e^{\frac{13\pi \sqrt{-1}}{7}}\}$

\bigskip

b) $N=4$

$J_{4}(\mathcal{L}_{+})-J_{4}(\mathcal{L}_{-})-(J_{2}(\mathcal{L}_{+})-J_{2}(%
\mathcal{L}_{-}))$

$=\frac{(-1+q)(1+q)(1+q^{2})(1+q^{4})(1+q^{8})}{q^{60}}%
(-1+3q^{2}+q^{4}-5q^{6}-3q^{8}-6q^{10}+16q^{12}+12q^{14}-8q^{16}-10q^{18}-32q^{20}+24q^{22}+33q^{24}-q^{26}-4q^{28}-61q^{30}+19q^{32}+48q^{34}-q^{36}+6q^{38}-80q^{40}+15q^{42}+64q^{44}-7q^{46}+16q^{48}-96q^{50}-q^{52}+94q^{54}-18q^{56}+5q^{58}-66q^{60}-17q^{62}+78q^{64}-8q^{66}-q^{68}-50q^{70}-21q^{72}+59q^{74}+2q^{76}-34q^{80}
$

$%
-25q^{82}+31q^{84}+9q^{86}+7q^{88}-13q^{90}-17q^{92}+6q^{94}+3q^{96}+5q^{98}-q^{100}-3q^{102}+q^{104})
$

Contains roots $\{e^{\frac{0\pi \sqrt{-1}}{2}},e^{\frac{1\pi \sqrt{-1}}{2}%
},e^{\frac{2\pi \sqrt{-1}}{2}},e^{\frac{3\pi \sqrt{-1}}{2}},e^{\frac{\pi
\sqrt{-1}}{8}},e^{\frac{2\pi \sqrt{-1}}{8}},e^{\frac{3\pi \sqrt{-1}}{8}},e^{%
\frac{5\pi \sqrt{-1}}{8}},e^{\frac{6\pi \sqrt{-1}}{8}},e^{\frac{7\pi \sqrt{-1%
}}{8}},$

$e^{\frac{9\pi \sqrt{-1}}{8}},e^{\frac{10\pi \sqrt{-1}}{8}},e^{\frac{11\pi
\sqrt{-1}}{8}},e^{\frac{13\pi \sqrt{-1}}{8}},e^{\frac{14\pi \sqrt{-1}}{8}%
},e^{\frac{15\pi \sqrt{-1}}{8}}\}$

\bigskip

II) $k=3$

We need to test the following

The root set of $J_{N}(\mathcal{L}_{+})-J_{N}(\mathcal{L}_{-})\equiv J_{2}(%
\mathcal{L}_{+})-J_{2}(\mathcal{L}_{-})$ contains $(A_{N-3}\cup
A_{N+5})-(A_{4}-A_{1})$

a) $N=4$

$J_{4}(\mathcal{L}_{+})-J_{4}(\mathcal{L}_{-})-(J_{3}(\mathcal{L}_{+})-J_{3}(%
\mathcal{L}_{-}))$

$=\frac{(-1+q)(1+q)(1-q+q^{2})(1+q+q^{2})(1-q^{3}+q^{6})(1+q^{3}+q^{6})}{%
q^{60}}%
(-1+3q^{2}+q^{4}-5q^{6}-3q^{8}-6q^{10}+16q^{12}+12q^{14}-7q^{16}-14q^{18}-30q^{20}+30q^{22}+30q^{24}+5q^{26}-25q^{28}-62q^{30}+33q^{32}+57q^{34}+31q^{36}-58q^{38}-100q^{40}+51q^{42}+104q^{44}+22q^{46}-101q^{48}-103q^{50}+102q^{52}+100q^{54}-23q^{56}-105q^{58}-52q^{60}+99q^{62}+57q^{64}-32q^{66}-58q^{68}-34q^{70}
$

$%
+61q^{72}+24q^{74}-6q^{76}-31q^{78}-31q^{80}+29q^{82}+13q^{84}+6q^{86}-13q^{88}-17q^{90}+6q^{92}+3q^{94}+5q^{96}-q^{98}-3q^{100}+q^{102})
$

Contains roots $\{e^{\frac{0\pi \sqrt{-1}}{1}},e^{\frac{1\pi \sqrt{-1}}{1}%
},e^{\frac{1\pi \sqrt{-1}}{9}},e^{\frac{2\pi \sqrt{-1}}{9}},e^{\frac{3\pi
\sqrt{-1}}{9}},e^{\frac{4\pi \sqrt{-1}}{9}},e^{\frac{5\pi \sqrt{-1}}{9}},e^{%
\frac{6\pi \sqrt{-1}}{9}},e^{\frac{7\pi \sqrt{-1}}{9}},$

$e^{\frac{8\pi \sqrt{-1}}{9}},e^{\frac{10\pi \sqrt{-1}}{9}},e^{\frac{11\pi
\sqrt{-1}}{9}},e^{\frac{12\pi \sqrt{-1}}{9}},e^{\frac{13\pi \sqrt{-1}}{9}%
},e^{\frac{14\pi \sqrt{-1}}{9}},e^{\frac{15\pi \sqrt{-1}}{9}},e^{\frac{16\pi
\sqrt{-1}}{9}},e^{\frac{17\pi \sqrt{-1}}{9}}\}$

\bigskip

\subsection{Examples for congruent skein relations: $SU(n)$-invariants case}

In this subsection, we provide a lot of examples for the congruent skein
relation for the $SU(n)$-invariants $J_N^{SU(n)}$, see formula (\ref%
{defSU(n)}) for the definition.

\textbf{Conjecture of congruent skein relation}

For a knot $\mathcal{K}$, for any positive integer $N, k$ and $N\geq k\geq 1$%
, we have

$J_{N}^{SU(n)}(\mathcal{K}_{+};q)-J_{N}^{SU(n)}(\mathcal{K}_{-};q) \equiv
J_{k}^{SU(n)}(\mathcal{K}_{+};q)-J_{k}^{SU(n)}(\mathcal{K}_{-};q) \mod \ %
[N-k]. $

$ J_{N}^{SU(n)}(\mathcal{K}_{+};q)-J_{N}^{SU(n)}(\mathcal{K}_{-};q)\equiv
J_{k}^{SU(n)}(\mathcal{K}_{+};q)-J_{k}^{SU(n)}(\mathcal{K}_{-};q) \mod \ %
[N+k+2]. $

$ J_{N}^{SU(n)}(\mathcal{K}_{+};q)-J_{N}^{SU(n)}(\mathcal{K}_{-};q) \equiv
J_{k}^{SU(n)}(\mathcal{K}_{+};q)-J_{k}^{SU(n)}(\mathcal{K}_{-};q) \mod \ %
[n-1]. $

\bigskip

\textbf{A) $SU(3)$ }

1)$\mathcal{K}_+=T(2,3)$ and $\mathcal{K}_-=T(2,1)$.

$N=4$:

$J_4^{SU(3)}(\mathcal{K}%
_+)=q^{-16}+q^{-26}+q^{-28}-q^{-34}+q^{-38}+q^{-40}-q^{-42}-2q^{-44}-q^{-46}+q^{-50}-2q^{-54} -q^{-56}+q^{-60}+2q^{-62}+q^{-70}+q^{-72}-q^{-76}-q^{-78}-q^{-80}+q^{-84}
$

$J_3^{SU(3)}(\mathcal{K}%
_+)=q^{-12}+q^{-20}+q^{-22}-q^{-26}+q^{-30}-2q^{-34}-q^{-36}-q^{-44}+q^{-46}+q^{-48}+q^{-50}-q^{-54}
$.

$J_2^{SU(3)}(\mathcal{K}%
_+)=q^{-8}+q^{-14}+q^{-16}-q^{-18}-q^{-24}-q^{-26}+q^{-30}$.

$J_1^{SU(3)}(\mathcal{K}_+)=q^{-4}+q^{-8}-q^{-12}$.

$J_4^{SU(3)}(\mathcal{K}_-)=J_3^{SU(3)}(\mathcal{K}_-)=J_2^{SU(3)}(\mathcal{K%
}_-)=J_1^{SU(3)}(\mathcal{K}_-)=1$.

\bigskip

So we have

$J_4^{SU(3)}(\mathcal{K}_+)-J_4^{SU(3)}(\mathcal{K}_-)-(J_3^{SU(3)}(\mathcal{%
K}_+)-J_3^{SU(3)}(\mathcal{K}_-))$

$%
=[1][10](-q^{-23}-q^{-25}-q^{-31}-2q^{-33}-2q^{-35}+q^{-39}-q^{-43}+2q^{-47}+3q^{-49}+3q^{-51}+q^{-53}-q^{-61}-2q^{-63}-2q^{-65}-q^{-67}+q^{-71}+q^{-73})
$

where $[10]$ implies the factor $[2]$ which is $[n-1]$. And similarly in the
following. \newline

$J_4^{SU(3)}(\mathcal{K}_+)-J_4^{SU(3)}(\mathcal{K}_-)-(J_2^{SU(3)}(\mathcal{%
K}_+)-J_2^{SU(3)}(\mathcal{K}_-))$

$%
=[2][9](-q^{-19}-q^{-23}-q^{-25}-q^{-27}-q^{-31}+q^{-37}+q^{-39}-q^{-45}+q^{-47}+2q^{-51}+q^{-55}-q^{-63}-q^{-65}-q^{-67}+q^{-73})
$\newline

$J_4^{SU(3)}(\mathcal{K}_+)-J_4^{SU(3)}(\mathcal{K}_-)-(J_1^{SU(3)}(\mathcal{%
K}_+)-J_1^{SU(3)}(\mathcal{K}_-))$

$%
=[3][8](-q^{-15}-q^{-19}-q^{-21}+q^{-23}-q^{-25}+q^{-29}-2q^{-31}-q^{-37}+2q^{-39}+q^{-45}-q^{-47}+q^{-49}+q^{-51}-q^{-53}+q^{-55}+q^{-61}-q^{-63} -q^{-65}-q^{-69}+q^{-73})
$\newline

$J_3^{SU(3)}(\mathcal{K}_+)-J_3^{SU(3)}(\mathcal{K}_-)-(J_1^{SU(3)}(\mathcal{%
K}_+)-J_1^{SU(3)}(\mathcal{K}_-))$

$%
=[2][7](-q^{-13}-2q^{-17}-q^{-27}+q^{-29}-q^{-31}+q^{-33}+q^{-37}+q^{-39}-q^{-45})
$ \newline

$J_3^{SU(3)}(\mathcal{K}_+)-J_3^{SU(3)}(\mathcal{K}_-)-(J_2^{SU(3)}(\mathcal{%
K}_+)-J_2^{SU(3)}(\mathcal{K}_-))$

$%
=[1][8](-q^{-17}-q^{-19}-q^{-23}-2q^{-25}-q^{-27}+q^{-31}+q^{-33}+q^{-35}+2q^{-37}+q^{-39}-q^{-43}-q^{-45})
$

\bigskip

\textbf{B) $SU(4)$ }

1)$\mathcal{K}_+=T(2,5)$ and $\mathcal{K}_-=T(2,3)$.\newline

$J_5^{SU(4)}(\mathcal{K}_+)-J_5^{SU(4)}(\mathcal{K}_-)-(J_4^{SU(4)}(\mathcal{%
K}_+)-J_4^{SU(4)}(\mathcal{K}_-))$

$%
=[3][13](-q^{-214}-q^{-210}-q^{-208}+2q^{-200}+q^{-198}+q^{-196}+2q^{-194}+2q^{-192}+2q^{-190}+q^{-188}+q^{-186} -q^{-184}-2q^{-182}-2q^{-180}-q^{-176}+q^{-174}-q^{-172}-2q^{-170}-3q^{-168}-q^{-166}+q^{-162}+q^{-160}-q^{-158} -4q^{-156}-5q^{-154}-3q^{-152}+q^{-150}+3q^{-148}+2q^{-146}-3q^{-142}-3q^{-140}+5q^{-136}+6q^{-134}+4q^{-132} +q^{-130}-2 q^{-128}-2 q^{-126}+4 q^{-124}+5 q^{-122}+5 q^{-120}+3 q^{-118}-2 q^{-116}-3 q^{-114}-2 q^{-112}+3 q^{-108}+q^{-106}-q^{-104}-q^{-102}-2 q^{-100}+q^{-98}+2 q^{-96}+2 q^{-94}+2 q^{-92}-q^{-90}-2 q^{-88}-2 q^{-86}-3 q^{-84}-q^{-82}-q^{-80}-2 q^{-78}-2 q^{-76}-3 q^{-74}-2 q^{-72}-2 q^{-70}+q^{-66}-q^{-58}+q^{-56}+q^{-54}+q^{-52}+q^{-50}+q^{-40} )
$\newline

$J_5^{SU(4)}(\mathcal{K}_+)-J_5^{SU(4)}(\mathcal{K}_-)-(J_3^{SU(4)}(\mathcal{%
K}_+)-J_3^{SU(4)}(\mathcal{K}_-))$

$%
=[2][12](-q^{-216}-2q^{-212}-2q^{-208}+q^{-206}-q^{-204}+3q^{-202}+4q^{-198}+5q^{-194}+5q^{-190}-2q^{-188}+3q^{-186}-5q^{-184}+q^{-182}-4q^{-180} +3q^{-178}-2q^{-176}+2q^{-174} -4q^{-172}-q^{-170}-5q^{-168}-2q^{-164}+q^{-162}-4q^{-160}-3q^{-158}-6q^{-156}-3q^{-154}+2q^{-150}+4q^{-148}+q^{-144}-4q^{-142}+q^{-140}-q^{-138}+5q^{-136} +4 q^{-132}-3 q^{-130}+4 q^{-128}-q^{-126}+8 q^{-124}+q^{-122}+8 q^{-120}-3 q^{-118}+5 q^{-116}-6 q^{-114}+4 q^{-112}-5 q^{-110}+5 q^{-108}-4 q^{-106}+7 q^{-104}-2 q^{-102}+9 q^{-100}-2 q^{-98}+9 q^{-96}-3 q^{-94}+7 q^{-92}-5 q^{-90}+5 q^{-88}-8 q^{-86}+q^{-84}-9 q^{-82}+2 q^{-80}-6 q^{-78}+3 q^{-76}-7 q^{-74}-7 q^{-70}+q^{-68}-5 q^{-66}+q^{-64}-7 q^{-62}-q^{-60}-6 q^{-58}+2 q^{-56}-2 q^{-54}+3 q^{-52}-q^{-50}+2 q^{-48}+2 q^{-44}+q^{-42}+2 q^{-40}+q^{-36}+q^{-32})
$\newline

$J_5^{SU(4)}(\mathcal{K}_+)-J_5^{SU(4)}(\mathcal{K}_-)-(J_2^{SU(4)}(\mathcal{%
K}_+)-J_2^{SU(4)}(\mathcal{K}_-))$

$%
=[3][11](-q^{-216}-q^{-212}-q^{-210}+2q^{-202}+q^{-200}+q^{-198}+2q^{-196}+q^{-194}+2q^{-192}+q^{-190}-q^{-184}-2q^{-182}+2q^{-180}-2q^{-172}-3q^{-170}-2q^{-168}-2q^{-166}-q^{-162}-3q^{-160}-3q^{-156}-3q^{-154}+2q^{-152}+q^{-150}+q^{-146}-q^{-144}-2q^{-142}-q^{-140}+4q^{-136}+2q^{-134}+q^{-132}+3q^{-130}+q^{-126}+2q^{-124}+2q^{-122}+2q^{-120}+2q^{-114}-2q^{-112}+q^{-110}+3q^{-108}+3q^{-104}+2q^{-102}+2q^{-98}+3q^{-92}-2q^{-90}-3q^{-84}+q^{-82}-q^{-80}-2q^{-78}-3q^{-74}-4q^{-72}+q^{-70}-4q^{-68}-q^{-66}-3q^{-62}-q^{-58}-2q^{-56}+q^{-54}-q^{-52}-2q^{-50}+q^{-48}-2q^{-46}-2q^{-44}+q^{-38}+q^{-34}+2q^{-32}+q^{-26}
$


\begin{thebibliography}{99}
\bibitem{Ai} A. K. Aiston, \emph{Skein theoretic idempotents of Hecke
algebras and quantum group invariants}. PhD. thesis, University of
Liverpool, 1996.

\bibitem{AMMM} A. Anokhina, A. Mironov, A. Morozov, An. Morozov, \emph{Knot
polynomials in the first non-symmetric representation}, arXiv:1211.6375.

\bibitem{BS} Y. Berest and P. Samuelson, \emph{Double affine Hecke algebras
and generalized Jones polynomials}, arXiv:1402.6032.

\bibitem{CLZ} Q. Chen, K. Liu and S. Zhu,
\emph{Volume conjecture for $SU(n)$-invariants}, arxiv:1511.00658.


\bibitem{CY} Q. Chen and T. Yang, \emph{A volume conjecture for a family of
Turaev-Viro type invariants of 3-manifolds with boundary}, arXiv: 1503.02547.

\bibitem{GV} R. Gopakumar and C. Vafa, \emph{On the gauge theory/geometry
correspondence}. Adv. Theor. Math. Phys., 3(5):1415-1443, 1999.

\bibitem{Hab} K. Habiro, \emph{A unified Witten-Reshetikhin-Turaev invariant
for integral homology spheres}, Invent. Math. 171 (2008), no. 1, 1-81. MR
2358055 (2009b:57020).

\bibitem{Hikami} K. Hikami, \emph{Quantum invariant for torus link and
modular forms}, Commun. Math. Phys. 246, 403-426 (2004). arXiv: 0305039v2.

\bibitem{HM} R. J. Hadji, H. R. Morton, \emph{A basis for the full Homfly
skein of the annulus}, arXiv: 0408078v2.

\bibitem{IMMM} H. Itoyama, A. Mironov, A. Morozov, An. Morozov, \emph{HOMFLY
and superpolynomials for figure eight knot in all symmetric and
antisymmetric representations}, arXiv:1203.5978.

\bibitem{Jones} V. Jones, \emph{Hecke algebra representations of braid
groups and link polynomial}, Ann. Math. 126 (1987), 335-388.

\bibitem{Ka} R. M. Kashaev, \emph{The hyperbolic volume of knots from the
quantum dilogarithm}, Lett. Math. Phys. \textbf{39} (1997), no. 3,
269-275.

\bibitem{L} S. G. Lukac, \emph{Idempotents of the Hecke algebra become Schur
functions in the skein of the annulus}, Math. Proc. Camb. Phil. Soc 138
(2005), 79-96.

\bibitem{Lu} S. G. Lukac, \emph{Homfly skeins and the Hopf link}. PhD.
thesis, University of Liverpool, 2001.

\bibitem{LP1} K. Liu and P. Peng, \emph{Proof of the
Labastida-Mari\~no-Ooguri-Vafa conjecture}. J. Differential Geom.,
85(3):479-525, 2010.

\bibitem{LP2} K. Liu and P. Peng, \emph{New structures of knot invariants}.
Commun. Number Theory Phys. 5 (2011), 601-615.

\bibitem{LP3} K. Liu and P. Peng, \emph{Framed knot and $U(N)$ Chern-Simons
gauge theory}, preprint.

\bibitem{LM0} W. B. R Lickorish and K. C. Millett, \emph{A polynomial
invariant of oriented links}, Topology \textbf{26} (1987) 107.

\bibitem{Lick} W. B. R. Lickorish, \emph{The skein method for three-manifold
invariants}, J. Knot Theory Ramifications 2 (1993), no. 2, 171-194.

\bibitem{LM} J. M. F. Labastida and M. Mari\~no, \emph{A new point of view
in the theory of knot and link invariants}. J. Knot Theory Ramifications,
11(2):173-197, 2002.

\bibitem{LMV} J. M. F. Labastida, Marcos Mari\~no and Cumrun Vafa. \emph{%
Knots, links and branes at large N}. J. High Energy Phys., (11):Paper 7-42,
2000.

\bibitem{LZ} X.-S. Lin and H. Zheng, \emph{On the Hecke algebra and the
colored HOMFLY polynomial}, math.QA/0601267.

\bibitem{Mac} I. G. Macdonald, \emph{Symmetric functions and Hall polynomials%
}. Oxford Mathematical Monographs. The Clarendon Press Oxford University
Press, New York, second edition, 1995. With contributions by A. Zelevinsky,
Oxford Science Publications.

\bibitem{Mu} H. Murakami, \emph{An Introduction to the Volume Conjecture},
arXiv:1002.0126.

\bibitem{MuMu} H. Murakami and J. Murakami, The colored Jones polynomials
and the simplicial volume of a knot, Acta Math. \textbf{186} (2001), no. 1,
85-104.

\bibitem{MM} H. R. Morton and P. M. G. Manchon, \emph{Geometrical relations
and plethysms in the Homfly skein of the annulus}, J. London Math. Soc. 78
(2008), 305-328.

\bibitem{MV} M. Mari\~no and C. Vafa. \emph{Framed knots at large N}. In
Orbifolds in mathematics and physics (Madison, WI, 2001), volume 310 of
Contemp. Math., pages 185-204. Amer. Math. Soc., Providence, RI, 2002.

\bibitem{OV} H. Ooguri and C. Vafa. \emph{Knot invariants and topological
strings}. Nuclear Phys. B, 577(3):419-438, 2000.

\bibitem{O} Edited by T. Ohtsuki, \emph{Problems on invariants of knots and
3-manifolds}, Geometry and Topology Monographs, Volume 4 (2002), 377-572.

\bibitem{RT} N. Y. Reshetikhin and V. G. Turaev, \emph{Invariants of
3-manifolds via link polynomials and quantum groups}, Invent. Math.,
103(1):547-597, 1991.

\bibitem{So} M. Soroush, \emph{Worldsheet Interpretation of the Level-Rank
Duality}, Arxiv: 1501.06542v1.

\bibitem{Turaev} V. G. Turaev, \emph{The Yang-Baxter equation and invariants
of links}, Invent. Math. 92(1988), 527-553.

\bibitem{Turaev2} V. G. Turaev, \emph{The Conway and Kauffman modules of a
solid torus}. Zap. Nauchn. Sem. Leningrad. Otdel. Mat. Inst. Steklov. (LOMI)
167 (1988), Issled. Topol. 6, 79-89.

\bibitem{Turaev3} V. G. Turaev, \emph{Quantum invariants of knots and
3-manifolds}, Second revised edition. de Gruyter Studies in Mathematics, 18.
Walter de Gruyter \& Co., Berlin, 2010

\bibitem{Witten} E. Witten, \emph{Quantum field theory and the Jones
polynomial}, Commun. Math. Phys. \textbf{121} (1989) , 351.

\bibitem{Witten2} E. Witten, \emph{Chern-Simons gauge theory as a string
theory}. In The Floer memorial volume, volume 133 of Progr. Math., pages
637-678. Birkh%
\"{}%
auser, Basel, 1995.

\bibitem{Witten3} E. Witten, \emph{Priviate communications}, 2015.

\bibitem{Zhu} S. Zhu, \emph{Colored HOMFLY polynomials via skein theory}, J.
High. Energy. Phys. 10(2013), 229
\end{thebibliography}
\end{document}